\PassOptionsToPackage{style=alphabetic,backend=biber}{biblatex} 
\PassOptionsToPackage
{left=2.5cm,
  right=2.5cm,
  top=1cm,
  bottom=1cm,
  footskip=18pt,
  includehead,
  includefoot,
  headsep=12pt}
{geometry} 
\documentclass[reqno,11pt]{article}
\usepackage{amsthm}
\usepackage{amsfonts}
\usepackage[colorlinks=true,linkcolor=blue,citecolor=blue]{hyperref}
\usepackage{biblatex}
\usepackage{geometry}
\usepackage{lmodern} 
\usepackage{stmaryrd}
\usepackage{longtable}
\usepackage{amssymb}
\usepackage{amsmath}
\usepackage{graphicx}
\usepackage{eucal, mathrsfs, bbm}
\usepackage{enumitem}

\numberwithin{equation}{section}
\theoremstyle{plain}

\newtheorem{theorem}{Theorem}[section]
\newtheorem{corollary}[theorem]{Corollary}
\newtheorem{problem}[theorem]{Problem}
\newtheorem{lemma}[theorem]{Lemma}
\newtheorem{proposition}[theorem]{Proposition}

\newtheorem{definition}[theorem]{Definition}
\newtheorem{remark}[theorem]{Remark}

\theoremstyle{remark}
\newtheorem{example}[theorem]{Example}

\renewcommand{\d}{{\mathrm d}}

\newcommand{\supp}{\mathop{\rm supp}\nolimits}
\newcommand{\restr}[1]{\lower3pt\hbox{$|_{#1}$}}

\newcommand{\mres}{\mathbin{\vrule height 1.6ex depth 0pt width 0.13ex\vrule height 0.13ex depth 0pt width 1.3ex}}

\newcommand{\la}{{\langle}}                  
\newcommand{\ra}{{\rangle}}

\newcommand{\eps}{\varepsilon}  
\newcommand{\nchi}{{\raise.3ex\hbox{$\chi$}}}

\newcommand{\N}{\mathbb{N}}

\renewcommand{\P}{\mathbb{P}}
\newcommand{\Q}{\mathbb{Q}}
\newcommand{\R}{\mathbb{R}}

\newcommand{\BB}{\mathscr{B}}

\newcommand{\FF}{\mathscr{F}}
\newcommand{\GG}{\mathscr{G}}

\newcommand{\cA}{{\ensuremath{\mathcal A}}}
\newcommand{\cB}{{\ensuremath{\mathcal B}}}

\newcommand{\cF}{{\ensuremath{\mathcal F}}}

\newcommand{\cH}{{\ensuremath{\mathcal H}}}
\newcommand{\cK}{{\ensuremath{\mathcal K}}}
\newcommand{\cI}{{\ensuremath{\mathcal I}}}
\newcommand{\cL}{{\ensuremath{\mathcal L}}}

\newcommand{\cN}{{\ensuremath{\mathcal N}}}
\newcommand{\cO}{{\ensuremath{\mathcal O}}}
\newcommand{\cP}{{\ensuremath{\mathcal P}}}
\newcommand{\cQ}{{\ensuremath{\mathcal Q}}}
\newcommand{\cR}{{\ensuremath{\mathcal R}}}
\newcommand{\cS}{{\ensuremath{\mathcal S}}}

\newcommand{\cU}{{\ensuremath{\mathcal U}}}
\newcommand{\cV}{{\ensuremath{\mathcal V}}}

\newcommand{\cZ}{{\ensuremath{\mathcal Z}}}

\newcommand{\bb}{{\boldsymbol b}}

\newcommand{\ff}{{\boldsymbol f}}

\newcommand{\ii}{{\boldsymbol i}}

\newcommand{\vv}{{\boldsymbol v}}

\newcommand{\bB}{{\boldsymbol B}}

\newcommand{\qQ}{{\boldsymbol Q}}

\newcommand{\ggamma}{{\boldsymbol \gamma}}

\newcommand{\mmu}{{\boldsymbol \mu}}

\newcommand{\ppi}{{\boldsymbol \pi}}

\newcommand{\xxi}{{\boldsymbol  \xi}}

\newcommand{\ttheta}{{\boldsymbol  \theta}}

\newcommand{\sfc}{{\mathsf c}}
\newcommand{\sfd}{{\mathsf d}}

\newcommand{\sfg}{{\mathsf g}}
\newcommand{\sfh}{{\mathsf h}}

\newcommand{\sfm}{{\mathsf m}}

\newcommand{\sfw}{{\mathsf w}}

\newcommand{\sfD}{{\mathsf D}}
\newcommand{\sfE}{{\mathsf E}}

\newcommand{\sfI}{{\mathsf I}}

\newcommand{\sfM}{{\mathsf M}}

\newcommand{\sfW}{{\mathsf W}}

\newcommand{\frd}{{\mathfrak  d}}

\newcommand{\frg}{{\mathfrak  g}}

\newcommand{\frm}{{\mathfrak  m}}

\newcommand{\frp}{{\mathfrak  p}}

\newcommand{\frr}{{\mathfrak  r}}

\newcommand{\frw}{{\mathfrak  w}}

\newcommand{\frF}{{\mathfrak  F}}

\newcommand{\rme}{{\mathrm e}}

\newcommand{\rmt}{{\mathrm t}}

\newcommand{\rmA}{{\mathrm A}}

\newcommand{\rmC}{{\mathrm C}}
\newcommand{\rmD}{{\mathrm D}}
\newcommand{\rmE}{{\mathrm E}}
\newcommand{\rmF}{{\mathrm F}}
\newcommand{\rmG}{{\mathrm G}}
\newcommand{\rmH}{{\mathrm H}}

\newcommand{\rmL}{{\mathrm L}}

\newcommand{\rmP}{{\mathrm P}}
\newcommand{\rmQ}{{\mathrm Q}}

\newcommand{\rmS}{{\mathrm S}}

\newcommand{\rmX}{{\mathrm X}}
\newcommand{\rmY}{{\mathrm Y}}

\newcommand{\ttG}{{\mathtt G}}

\newcommand{\ttM}{{\mathtt M}}
\newcommand{\ttN}{{\mathtt N}}

\newcommand{\ttP}{{\mathtt P}}

\newcommand{\ttR}{{\mathtt R}}

\newcommand{\ttW}{{\mathtt W}}

\newcommand{\bttM}{{\boldsymbol{\mathtt M}}}

\newcommand{\pPi}{{\text{$\mathit{\Pi}$}}}
\renewcommand{\P}{\ensuremath{\mathbb{M}}}
\newcommand{\OOmega}{\rmQ}
\newcommand{\oomega}{q}
\newcommand{\Thetao}{\Omega}
\newcommand{\thetao}{\omega}
\renewcommand{\Q}{\ensuremath{\mathbb P}}
\newcommand{\id}{\mathrm{Id}}

\title{Totally convex functions, 
\texorpdfstring{$L^2$}{}-Optimal transport for laws of random measures, and solution to the Monge problem}
\begin{document}

\author{
  Alessandro Pinzi
\thanks{Bocconi University, Department of Decision Sciences, via Roentgen 1, 20136 Milano (Italy). 
Email:
 \textsf{alessandro.pinzi@phd.unibocconi.it}},
 Giuseppe Savar\'e
 \thanks{Bocconi University, Department of Decision Sciences and BIDSA, via Roentgen 1, 20136 Milano (Italy).
 Email: \textsf{giuseppe.savare@unibocconi.it}
 }
}

\date{\today}

\maketitle

\begin{abstract}
We study the Optimal Transport problem 
for laws of random measures in the Kantorovich-Wasserstein space
$\cP_2(\cP_2(\rmH))$,
associated with a Hilbert space $\rmH$
(with finite or infinite dimension) and 
for the corresponding quadratic cost induced by the squared Wasserstein metric in $\cP_2(\rmH).$

Despite the lack of smoothness of the
cost, the fact that the space $\cP_2(\rmH)$ is not Hilbertian, and 
the curvature distortion induced 
by the underlying Wasserstein metric, we will show how to recover at the level of random measures in 
$\cP_2(\cP_2(\rmH))$ the same deep and powerful results linking 
Euclidean Optimal Transport problems 
in $\cP_2(\rmH)$ and  convex 
analysis.

Our approach relies on the notion of totally convex functionals, on
their total subdifferentials, 
and their Lagrangian liftings 
in the space 
square integrable $\rmH$-valued maps $L^2(\OOmega,\P;\rmH).$

With these tools, we identify a natural class of regular measures in $\cP_2(\cP_2(\rmH))$ 
for which the Monge formulation of 
the OT problem has a unique solution
and we will show that this class
includes relevant examples 
of measures with full support
in $\cP_2(\rmH)$ arising from 
the push-forward transformation
of nondegenerate Gaussian measures in $L^2(\OOmega,\P;\rmH).$

\end{abstract}

{\small\tableofcontents}

\newcommand{\mduality}[2]{[#1,#2]}
\newcommand{\rW}[1]{\cP_{#1}(\rmH)}
\newcommand{\cPP}{\cP\kern-4.5pt\cP}
\newcommand{\RW}[1]{\cP\kern-4.5pt\cP_{#1}(\rmH)}
\newcommand{\RWfin}[1]{\ensuremath{\cP\kern-4.5pt\cP_{#1}(\R^d)}}
\newcommand{\RWr}[1]{\cP\kern-4.5pt\cP_{#1}^{r\kern-1pt r}(\rmH)}
\newcommand{\RWrd}[1]{\ensuremath{\cP\kern-4.5pt\cP_{#1}^{r\kern-1pt r}(\R^d)}}
\newcommand{\RWru}[1]{\ensuremath{\cP\kern-4.5pt\cP_{#1}^{r\kern-1pt r}(\R)}}
\newcommand{\RWgr}[1]{\cP\kern-4.5pt\cP_{#1}^{gr\kern-1pt r}(\rmH)}
\newcommand{\RWgrd}[1]{\ensuremath{\cP\kern-4.5pt\cP_{#1}^{g\kern-1pt r\kern-1pt r}(\R^d)}}
\newcommand{\RWW}[1]{\cP\kern-4.5pt\cP_{#1}(\rmH\times\rmH)}
\newcommand{\RWWdet}[1]{\cP\kern-4.5pt\cP_{#1}^{\rm det} (\rmH\times\rmH)}
\newcommand{\msc}[2]{[#1,#2]}
\newcommand{\mmsc}[2]{[\kern-2pt[#1,#2]\kern-2pt]}
\newcommand{\shto}{\shortrightarrow}
\newcommand{\Rinf}{\ensuremath{\R\cup\{\infty\}}}
\newcommand{\bpartial}{\boldsymbol{\partial}}
\newcommand{\bpartialt}{\boldsymbol{\partial}_{\rmt}}
\newcommand{\Ast}{{\mathop{\scalebox{1}{\raisebox{-0.2ex}{$\ast$}}}}}%
\newcommand{\RGamma}{\mathrm{R}\Gamma}
\newcommand{\lift}[1]{#1^\ell}
\newcommand{\cPdet}[1]{{\cP^{\rm det}_{#1}}}
\newcommand{\LF}{\hat{\brmF}}
\newcommand{\bttP}{{\boldsymbol{\mathtt P}}}

\section{Introduction}

One of the most elegant and fascinating aspects of Optimal Transport Theory
\cite{RRI-II,Villani09}
for the classical quadratic cost $\sfc(x,y):=\frac 12 |x-y|^2$ in
$\R^d$ is its link with convex analysis,
which has been thoroughly exploited
by Knott-Smith, Rachev-R\"uschendorf, and Brenier 
\cite{Knott-Smith84,RR90,Brenier91}
(see also 
\cite{Cuesta-Matran89}).
If $\mu,\nu$ belong to the space
$\cP_2(\R^d)$ of probability measures with finite quadratic moment
\begin{equation}
    \label{eq:2moment}    \sfm_2^2(\mu):=\int_{\R^d}|x|^2\,\d\mu(x),
\end{equation}
is in fact possible to prove that
a coupling $\ggamma\in \cP(\R^d\times \R^d)$
with marginals $\mu,\nu$ 
(we say that $\ggamma\in \Gamma(\mu,\nu)$) is optimal for the $L^2$-Kantorovich-Wasserstein metric
\begin{equation}
  \label{eq:110}
  \sfw_2^2(\mu,\nu):=
  \min\Big\{\int_{\R^d\times \R^d}|x-y|^2\,\d\ggamma:
  \ggamma\in \Gamma(\mu,\nu)\Big\}
  =\min\Big\{ 
  \mathbb{E}\big[
  |X-Y|^2\big]:X\sim \mu,\ Y\sim\nu\Big\}
\end{equation}
if and only if its support $S:=\supp(\ggamma)\subset \R^d\times \R^d$
is cyclically monotone, i.e.
\begin{equation}
  \label{eq:111}
  \begin{gathered}
      \sum_{n=1}^N \langle y_n,x_n-x_{\sigma(n)}\rangle\ge0
  \quad\text{for every $N\in \N$, every choice of }(x_n,y_n)\in S\\
  \text{ and every permutation $\sigma\in \rmS_N$}.
  \end{gathered}
\end{equation}
The dual formulation of \eqref{eq:110}
and the fact that every cyclically monotone subset of $\R^d\times \R^d$ is 
contained in the graph of the convex subdifferential
$\partial\varphi$
of a convex and lower semicontinuous function
$\varphi:\R^d\to \Rinf$
implies that we can find optimal conjugate Kantorovich potentials
$\varphi,\varphi^{\Ast}$ such that
for all optimal couplings $\ggamma_{\rm opt}$ 
solving \eqref{eq:110}
the support of 
$\ggamma_{\rm opt}$ is contained in the graph of $\partial\phi$ and 
\begin{equation}
  \label{eq:110bis}
\int_{\R^d\times \R^d}
x\cdot y\,\d\ggamma_{\rm opt}(x,y)=
\int_{\R^d}\varphi(x)\,\d\mu(x)+
\int_{\R^d}\varphi^{\Ast}(y)\,\d\nu(y).
\end{equation}
The simple but crucial link between 
the first integral in \eqref{eq:110bis}
and the Optimal Transport problem
\eqref{eq:110}
is guaranteed by the specific property of the Euclidean norm in $\R^d$, for which
\begin{equation}
    \label{eq:Rd-miracle}
    \sfw_2^2(\mu,\nu)=
    \sfm_2^2(\mu)
    +\sfm_2^2(\nu)-
    2\msc\mu\nu,\quad
    \msc\mu\nu:=
    \max_{\ggamma\in \Gamma(\mu,\nu)}
    \int x\cdot y\,\d\ggamma,
\end{equation}
so that the minimum problem \eqref{eq:110}
and the maximum problem \eqref{eq:Rd-miracle}
defining $\msc\mu\nu$ share the same class of
optimizers. 

Since the subdifferential of a convex function $\varphi$
is a singleton at every differentiability point of $\varphi$
and the set of non-Gateux-differentiability points of a convex Lipschitz function can be covered by a countable union of 
\textrm{d.c.}~hypersurfaces \cite{Zajicek79} (a result which holds even in infinite dimensional Hilbert spaces), 
one can prove that when $\mu\in \cP_2(\R^d)$
does not give mass to 
\textrm{d.c.}~hypersurfaces 
(we say that $\mu\in \cP_2^r(\R^d)$
is \emph{regular})
there exists a unique optimal coupling $\ggamma_{\rm opt}$
which is moreover concentrated on the graph of a Borel map $f$, thus satisfying $f_\sharp \mu=\nu$. 
The class of regular measures $\cP^r_2(\R^d)$ coincide with
the class of atomless measures when $d=1$
and contains all the measures absolutely continuous with respect to the $d$-dimensional Lebesgue measure $\cL^d$ for every dimension $d$.

This remarkable combination of analytic and geometric 
arguments yields the solution of the Monge formulation of \eqref{eq:110}-\eqref{eq:Rd-miracle},
i.e.~the existence of an optimal transport map $f$ such that 
\begin{equation}
    \label{eq:Monge-intro}
    f_\sharp\mu=\nu,\quad
    \sfw_2^2(\mu,\nu)=
    \int_{\R^d}|f(x)-x|^2\,\d\mu(x),\quad
    \msc\mu\nu=
    \int_{\R^d}f(x)\cdot x\,\d\mu(x).
\end{equation}
\subsubsection*{The convex theory for random measures}
Whenever $(\rmX,\sfd_\rmX)$ is 
a (complete, separable) metric space,
the construction of the $L^2$-Kantorovich-Wasserstein metric can be naturally extended to 
$\cP_2(\rmX)$, the space of 
probability measures on $\rmX$ with finite quadratic moment;
the resulting (squared) metric
\begin{equation}
  \label{eq:110ter}
  \sfW_{2,\sfd_\rmX}^2(\mu,\nu):=
  \min\Big\{\int_{\rmX\times \rmX}\sfd_\rmX^2(x,y)\,\d\ggamma:
  \ggamma\in \Gamma(\mu,\nu)\Big\}
  =
  \min\Big\{ 
  \mathbb{E}\big[
  \sfd^2_\rmX(X,Y)\big]:X\sim \mu,\ Y\sim\nu\Big\}
\end{equation}
inherits relevant geometric properties from the underlying space $\rmX$ and the problem still admits a dual
formulation involving Kantorovich potentials. 

Since $(\cP_2(\R^d),\sfw_{2})$ 
is a complete and separable metric space,
a nice example of applications of the 
metric perspective is 
provided by 
the possibility to lift 
Optimal Transport problems
at the level of the laws of the so-called 
random measures, i.e.~probability measures
in the space 
$\cPP_2(\R^d):=\cP_2(\cP_2(\R^d))$
endowed with 
the Kantorovich-Wasserstein metric
$\sfW_{2}:=\sfW_{2,\sfw_{2}}$
\begin{align}
    \label{eq:WW2-intro}
    \sfW_2^2(\ttM ,\ttN):={}&
    \min\Big\{\int_{\cP_2(\R^d)
    \times\cP_2(\R^d)}
    \sfw_2^2(\mu,\nu)\,\d
    \text{$\pPi$}(\mu,\nu):
    \text{$\pPi$}\in \Gamma(\ttM,\ttN)\Big\}
    \\={}&
    \min\Big\{
    \mathbb{E}\big[
  \sfw_2^2(M,N)\big]:M\sim \ttM,\ N\sim\ttN    \Big\},\quad
    \ttM,\ttN\in \cPP_2(\R^d).
\end{align}
A similar class of problems,
starting however from a 
smooth and compact Riemannian manifold instead of $\R^d$, have recently been studied by 
\cite{Emami-Pass}. 
Here we want to focus on the Euclidean case 
(also including infinite dimensional separable Hilbert spaces), which shows 
distinguished and remarkable features
and has recently attracted 
a lot of attention in view
of many interesting applications
\cite{bonet2025flowing-329,fornasier2023approximation-b42,acciaio2025,Pinzi-Savare25, catalano2024hierarchical}.

It is well known that in general metric spaces $(\rmX,\sfd_\rmX)$
(as, in particular, $(\cP_2(\R^d),\sfw_2)$)
the link between Optimal Transport problems and convex analysis 
is typically lost, mainly due to the possible lack of a linear structure in $\rmX$;
even when
$\rmX=\R^d$ but the 
metric cost is 
induced by a non-Hilbertian norm $\|\cdot\|$,
convexity and Legendre duality of Kantorovich potentials do not hold, since  the optimality condition
and the structure of optimal transport maps involve the \emph{nonlinear} differential associated with $\frac12\|\cdot\|^2.$

Even if $(\cP_2(\R^d),\sfw_2)$ is a genuine
metric space which is positively curved in the sense of Alexandrov
\cite{AGS08},
the aim of the present paper is to show that 
\begin{quote}
    \em    a large part of the convexity landscape of the Euclidean case remarkably holds also
for the Optimal Transport problem 
\eqref{eq:WW2-intro}  
for laws of random measures in $\cPP_2(\R^d),$
if we use the appropriate notion
of \emph{total displacement} convexity in 
$\cP_2(\R^d)$,
\end{quote}
i.e.~convexity
along interpolation curves induced by arbitrary couplings.

Such a nice and somehow unexpected 
result relies on two important properties.
First of all, 
as for \eqref{eq:Rd-miracle},
the lifted Wasserstein metric $\sfW_2$
given by \eqref{eq:WW2-intro}
still satisfies a similar identity
\begin{equation}
    \begin{gathered}
        \sfW_2^2(\ttM_1,\ttM_2)=
    \sfM_2^2(\ttM_1)+
    \sfM_2^2(\ttM_2)-
    2\mmsc{\ttM_1}{\ttM_2},\\
    \sfM_2^2(\ttM) = \int \sfm_2^2(\mu) d\ttM(\mu), \quad
    \mmsc{\ttM_1}{\ttM_2}=
    \max_{\pPi\in \Gamma(\ttM_1,\ttM_2)}
    \int \msc{\mu_1}{\mu_2}\,
    \d\pPi(\ttM_1,\ttM_2),
    \end{gathered}
\end{equation}
so that we can study
the equivalent formulation in terms of 
the maximization of the function $(\mu_1,\mu_2)\mapsto 
\msc{\mu_1}{\mu_2}.$

Even if $\msc\cdot\cdot$ is not bilinear,
it exhibits many analogies with 
a scalar product, in particular 
along displacement interpolation of measures,
i.e.~curves obtained by 
general couplings $\mmu\in \cP(R^d\times \R^d)$
via the dynamic push forward
\begin{equation}
\label{eq:disp-interpolation-intro}
    \mu_t=(\pi^{1\shortrightarrow2}_t)_\sharp\mmu,\quad
    \pi^{1\shortrightarrow2}_t(x_1,x_2):=
    (1-t)x_1+tx_2,\quad
    t\in [0,1],\quad 
    \mmu\in \Gamma(\mu_0,\mu_1).
\end{equation}
When $\mmu$ is an optimal coupling
between $\mu_0$ and $ \mu_1$ 
for the Wasserstein metric \eqref{eq:110}
the curve $t\mapsto\mu_t$ is in fact
a minimal, constant speed geodesic 
in $\cP_2(\R^d)$ (which we call
\emph{optimal displacement interpolation}) 
and plays a crucial
geometric role
in the Optimal Transport setting,
since the pioneering paper of McCann
\cite{McCann97}. 
In particular a function $\phi:\cP_2(\R^d)\to\Rinf$ is called
geodesically (or displacement) convex
if for every $\mu_0,\mu_1\in \cP_2(\R^d)$
there exists a 
geodesic $(\mu_t)_{t\in [0,1]}$ 
connecting $\mu_0$ to $\mu_1$ such that 
$t\mapsto \phi(\mu_t)$ is convex in $[0,1].$

A more restricted class of functions 
are in fact convex along 
\emph{any} displacement interpolation, induced by arbitrary couplings 
$\mmu\in \Gamma(\mu_0,\mu_1)$ 
as in \eqref{eq:disp-interpolation-intro},
thus satisfying
\begin{equation}
    \label{eq:total-convexity-intro}
    \phi\big((\pi^{1\shto2}_t)_\sharp\mmu
    \big)
    \le 
    (1-t)\phi\big(\mu_0)
    +t\phi\big(\mu_1\big)
    \quad
    \text{for every }
    \mu_0,\mu_1\in \cP_2(\R^d),\ 
    \mmu\in \Gamma(\mu_0,\mu_1).
\end{equation}
Such a class of \emph{totally displacement convex} functionals, thoroughly studied in \cite{CSS22}, enjoys better properties.
Starting from the fact 
that
(see also the inspiring notes by
Brenier \cite{Brenier20})
\begin{equation}
    \label{eq:msc-is-totally-convex}
    \mu\mapsto \msc\mu\nu\quad\text{is totally convex for every $\nu\in \cP_2(\R^d)$},
\end{equation}
we can introduce the 
Kantorovich-Legendre-Fenchel transform
\begin{equation}
    \label{eq:KLF-intro}
    \phi^\star(\nu):=
    \sup_{\mu\in \cP_2(\R^d)}
    \msc \nu\mu-\phi(\mu)
\end{equation}
and prove that 
proper, totally convex, and lower semicontinuous functions are characterized by
the identity $\phi=\phi^{\star\star}$
as for convex functions in Euclidean spaces
(see Section \ref{subsec:KLF}).

The Rockafellar type 
transformation 
\eqref{eq:KLF-intro} 
is a typical technique
in Optimal Transport (where it is applied
to the concave version of the potentials
and it is called $\sfc$-transform).
What distinguishes the Euclidean 
and the current random-Euclidean setting
is the possibility to characterize
$\sfc$-transforms and self-biconjugate functions in a simple
way using convexity.

This remarkable property 
allows us to retrace the same strategy as 
in the finite-dimensional Euclidean theory and has
relevant applications:
\begin{enumerate}
    \item it provides an intrinsic
    characterization for the optimal Kantorovich potentials
    associated with 
    the Wasserstein metric 
    \eqref{eq:WW2-intro}
    (Section \ref{subsec:c-concave}):
    they coincide with the class of totally convex functionals;
    \item 
    it allows for a deeper understanding
    of the Wasserstein (total) subdifferential
    of $\phi$
    \cite{AGS08,CSS22}, interpreted as 
    a 
    \emph{Multivalued Probability Vector Field} (Section \ref{sec:MPVF});
    \item it clarifies the structure of optimal couplings 
    and minimal geodesics
    (Sections \ref{subsec:L2OT}, 
    \ref{subsec:G-structure})
    \item it suggests a general lifting strategy to the $L^2$-space of Lagrangian maps, importing a Hilbertian perspective for regularity
    of laws of random measures
    (Section \ref{subsec:regular-measures})
    which plays a crucial role in 
    proving uniqueness for solution
    to the Monge formulation.
\end{enumerate}
Let us briefly summarize the main points: 
combining Kantorovich duality and 
total displacement convexity, we will prove that for every $\ttM,\ttN\in \cPP_2(\R^d)$
there exists a pair of conjugate totally convex function $\phi,\phi^\star:\cP_2(\R^d)\to\Rinf$
such that 
\begin{equation}
    \label{eq:dudality-intro}
    \int \phi(\mu)\,\d\ttM(\mu)+
    \int \phi^\star(\nu)\,\d\ttN(\nu)=
    \mmsc\ttM\ttN.
\end{equation}
We will then show that 
there is a natural correspondence between
optimal couplings $\pPi \in \Gamma_o(\ttM,\ttN)$ 
minimizing 
\eqref{eq:WW2-intro}
and random coupling laws
$\bttP\in \cPP_2(\R^d\times \R^d)$
which can be used to characterize $\sfW_2$ as 
\begin{equation}
    \sfW_2^2(\ttM,\ttN)
    =\min\bigg\{
    \int \int_{\R^d\times \R^d}
    |x-y|^2\,\d\ggamma(x,y)
    \,
    \d\boldsymbol\ttP(\ggamma):
    (\pi^1_\sharp)_\sharp 
    \bttP=\ttM,
    \ 
    (\pi^2_\sharp)_\sharp 
    \boldsymbol{\bttP} =\ttN\bigg\}. 
\end{equation}
Random couplings are optimal
if and only if their support is contained
in the so-called total subdifferential
$\bpartialt\phi$
of the optimal potential 
$\phi$ in \eqref{eq:dudality-intro}, a closed subset of 
$\cP_2(\R^d\times \R^d).$
Optimality can also be characterized by 
total cyclical monotonicity and in particular 
implies that optimal random couplings are concentrated
on the subset $\cP_{2,o}(\R^d\times \R^d)$
of the usual optimal couplings in $\R^d\times \R^d.$

As for the classic subdifferentials of convex functions, among all the elements of 
the total subdifferential 
$\bpartialt\phi$ there is a minimal distinguished one (called the minimal section
and denoted by $\bpartialt^\circ\phi$)
which can be represented by 
a nonlocal deterministic field 
$\ff^\circ:\R^d\times \cP_2(\R^d)\to\R^d:$
\begin{equation}
    \label{eq:minimal-section}
    \ggamma\in \bpartialt^\circ\phi(\mu)\quad
    \Leftrightarrow
    \quad
    \ggamma=(\ii\times \ff^\circ(\cdot,\mu)_\sharp 
    \mu.
\end{equation}
Total cyclical monotonicity 
can be more easily understood in the case of $\ff^\circ$, where it reads as 
\begin{equation}
    \label{eq:total-cyclic-intro}
    \begin{gathered}
        \sum_{n=1}^N\int \la
    \ff^\circ(x_n,\mu_n),x_n-x_{\sigma(n)}\ra\,\d\mmu(x_1,\cdots,x_N)\ge0\\
    \text{for every }
      \mu_i\in D(\bpartialt\phi),\ \mmu\in \Gamma(\mu_1,\mu_2,\cdots,\mu_N),\ 
      \sigma\in \rmS_N=\operatorname{Sym}(\{1,\cdots,N\}).
    \end{gathered}
\end{equation}
As a byproduct, when $\bpartialt\phi[\mu]$ reduces to a singleton for $\ttM$-a.e.~$\mu\in \cP_2(\R^d)$,
the OT problem \eqref{eq:WW2-intro}
has a unique solution $\pPi_o$ which
is also concentrated on a Monge map 
$\FF:\cP_2(\R^d)\to \cP_2(\R^d)$,
$\pPi=(\id\times \FF)_\sharp \ttM$.
Moreover, 
$\FF$ can be expressed 
as a push-forward via $\ff^\circ$:
\begin{equation}
    \label{eq:push-intro}
    \FF(\mu)=\ff^\circ(\cdot,\mu)_\sharp\mu,\quad 
    \sfW_2^2(\ttM,\ttN)=
    \int \sfw_2^2(\mu,\FF(\mu))\,\d\ttM(\mu)
    =\int\bigg(
    \int |\ff^\circ(x,\mu)
    -x|^2\,\d\mu(x)\bigg)
    \,\d\ttM(\mu).
\end{equation}
It turns out that $\ff^\circ$ solves
the strict Monge formulation 
of \eqref{eq:WW2-intro}
\begin{equation}
    \label{eq:strict-Monge-intro}
    \inf\bigg\{
    \int\bigg(
    \int |\ff(x,\mu)
    -x|^2\,\d\mu(x)\bigg)
    \,\d\ttM(\mu):
    \ff:\R^d\times \cP_2(\R^d)\to\R^d,\
    \ttN=\int \delta_{\ff(\cdot,\mu)_\sharp\mu}\,\d\ttM(\mu)\bigg\}.
\end{equation}
\subsubsection*{Lagrangian parametrizations, Hilbertian liftings, and laws of Gaussian-generated random measures}
The above discussion raises the crucial question to find general condition on $\ttM$ ensuring that $\bpartialt\phi$ is concentrated on a singleton $\ttM$-a.e.
A similar problem has also been considered in 
\cite{Emami-Pass} by assuming 
suitable regularity properties on the Dirichlet form associated with $\ttM$ in $\cP_2(M)$, in particular the Rademacher property studied in \cite{DelloSchiavo20}.

Here we adopt a different perspective, 
inspired by the crucial fact that 
Optimal Kantorovich potentials are totally convex on 
$\cP_2(\R^d)$, i.e.~convex along arbitrary couplings as in \eqref{eq:total-convexity-intro}.
It is then possible to apply a natural
Lions-Lagrangian lifting technique
that has been systematically studied in 
\cite{CSS22} in the context of convex analysis and evolution problems
(but see also the relevant notions
of $\rmL$-convexity, $\rmL$-monotonicity
\cite{Cardaliaguet13,Carmona-Delarue18},
Fr\'echet differentiability
\cite{Gangbo-Tudorascu19},
and the discussion
of \cite[Remark 5.4]{CSS22}).

The main idea is to introduce a standard Borel space
$(\OOmega,\cF_\OOmega)$ endowed with 
a nonatomic probability measure $\P$
which can represent every 
measure of $\mu\in\cP_2(\R^d)$
as the law $X_\sharp \P$ 
of a map $X$
in the Hilbert space $\cH:=L^2(\OOmega,\P;\R^d).$
The law map $\iota=\iota_\P:X\mapsto X_\sharp \P$
is a $1$-Lipschitz surjection 
from $\cH$ to $\cP_2(\R^d)$
and total convexity of $\phi$ in 
$\cP_2(\R^d)$ is equivalent to 
the usual convexity of $\hat \phi:=\phi\circ\iota$ in the Hilbert space $\cH.$ Such a correspondence also holds at the levels of the respective subdifferentials,
so that the measures $\mu$ where $\bpartialt\phi[\mu]$ reduces to a singleton 
correspond to maps $X\in \cH$ where
$\hat\phi$ is Gateux-differentiable.

Since every measure $\ttM$ on 
$\cP_2(\R^d)$ can be obtained as the push-forward $\iota_\sharp \frm$ 
of a measure $\frm$ on $\cH$,
it is tempting to lift the problem of $\ttM$-a.e.~differentiability in $\cP_2(\R^d)$
to the problem of $\frm$-a.e.~differentiability of $\hat\phi$ in $\cH$,
for which many powerful result are known.
In particular, 
\textrm{d.c.}~hypersurfaces in a Hilbert space are 
Gaussian null 
\cite{Aronszajn76,Phelps78,Bogachev84,Csonryei99}), i.e.~are negligible
w.r.t.~every nondegenerate Gaussian measure.

We will systematically pursue this direction, showing that the strict Monge problem in $\cPP_2(\R^d)$ has a unique solution
if $\ttM$ satisfies two conditions: 
\begin{enumerate}
    \item[(R1)] 
    it is concentrated on
    the set of 
    \emph{regular} measures $\cP^r_2(\R^d)$, i.e.~$\mu\in \cP_2^r(\R^d)$ for $\ttM$-a.e.~$\mu$;
    \item[(R2)] $\ttM(B)=0$ on
    every Borel set $B\subset \cP_2^r(\R^d)$
    such that $\iota^{-1}(B)$ 
    is contained in a \textrm{d.c.}~hypersurface of the Hilbert space
    $\cH=L^2(\OOmega,\P;\R^d).$
\end{enumerate}
We call $\RWrd2$ the 
set of \emph{super-regular measures} satisfying the two conditions above; it is worth noticing that 
the first condition has also been assumed by 
\cite{Emami-Pass}, whereas the second one
is strongly related to the lifting procedure.
We will show that 
those conditions are nearly optimal if
we look for measures $\ttM$ for which
the usual Monge formulation has 
at least one solution for every target measure $\ttN$. Both conditions are stable if we replace $\ttM$ by $\ttM'\ll\ttM$. 

A simple way to construct measures
satisfying (R2) is to 
start from a regular measure $\frm\in \cP_2^r(\cH)$ and 
taking its push forward $\ttM=\iota_\sharp \frm$. 
We will focus on the relevant case
of Laws of Gaussian-Generated Random Measures (LGGRM), i.e.~measures 
in $\cPP_2(\R^d)$ of the form 
$\ttG=\iota_\sharp \frg$ 
where $\frg$ is a non-degenerate Gaussian measure in $\cH$ \cite{Bogachev98}.
They have full support in $\cP_2(R^d)$ 
and will satisfy condition (R2). 
We will show that \emph{every} LGGRM 
in dimension $d=1$ is super-regular
and we will exhibit a large class
of super-regular 
LGGRM in every dimension.
As a byproduct, we obtain that the class
of super-regular measures is dense in 
$\cPP_2(\R^d).$

We have developed our analysis in the
case of the ``Euclidean'' $2$-Wasserstein metric, since we believe that its distinguished features  deserve a separate analysis.
Since (finite) dimension play a role only in the final discussion of super-regular measures, 
we decided to develop our theory in an arbitrary separable Hilbert space $\rmH$, 
even if all the results are new also in finite dimension.
In addition, we think that many tools we have introduced in this paper may be useful to study
the more general case of the $L^p$-Wasserstein metric induced by a smooth norm in $\R^d$. We are also confident that 
the class of LGGRM measures may reveal 
interesting features from the viewpoint of 
the induced Dirichlet form in $\cP_2(\R^d)
$ (see also \cite{Fornasier-Savare-Sodini23}), and we plan to address both questions in a forthcoming paper.

\subsubsection*{Plan of the paper}
After a quick recap of the main notions and notation in Section \ref{sec:prel},
we will deal with totally convex function
and their link with optimal Kantorovich potentials in Section \ref{sec:totally-convex}.
The related notions of multivalued probability vector fields and total subdifferentials are developed in Section
\ref{sec:MPVF}.

Section 
\ref{sec:ROT} is devoted
to applications of these tools to the Optimal Transport problem in $\cPP_2(\R^d)$.
The (strict) Monge formulation, its solution
for super-regular measures, and the discussion of the relevant examples
is presented in the last section
\ref{sec:Monge}.

\subsubsection*{Acknowledgments}
We wish to thank Eugenio Regazzini and Luciano Tubaro
for stimulating discussions. We are grateful to Anna Korba for insightful discussions and for providing a draft of the paper \cite{bonet2025flowing-329}. \\
GS has been supported by the MIUR-PRIN 202244A7YL project Gradient Flows and Non-Smooth
Geometric Structures with Applications to Optimization and Machine Learning, by the INDAM project
E53C23001740001, and by the Institute for Advanced Study of the Technical University of Munich, funded
by the German Excellence Initiative.


\section{Notation and preliminary results}
\label{sec:prel}

The following table contains the main notation that we shall use throughout the paper:
\begin{center}
\newcommand{\specialcell}[2][c]{%
  \begin{tabular}[#1]{@{}l@{}}#2\end{tabular}}
\begin{small}
\begin{longtable}{lll}
&$\rmH$&separable Hilbert space;\\
&$\cP_2(\rmH)$&the space of probability measures on $\rmH$ with finite quadratic moment, \ref{subsec:W2};\\
&$\cP_2^r(\rmH)$&regular measures, Def.~\ref{def:usual-regular};\\
&$\cP_{2,o}(\rmH\times \rmH)$&the set of optimal couplings, \ref{subsec:W2};\\
&$\sfw_2$&the $L^2$-Wasserstein metric on $\cP_2(\rmH)$, \ref{subsec:W2};\\
&$\msc\cdot\cdot$&the maximal correlation pairing, Definition \ref{def:MC-function};\\
&$\mu,\nu,\cdots$&notation for typical measures in $\cP_2(\rmH)$;\\
&$\cPP_2(\rmH)=\cP_2(\cP_2(\rmH))$&
the space of probability measures on $\cP_2(\rmH)$ with finite quadratic moment,
\ref{subsec:W2};
\\
&$\sfW_2$&the $L^2$-Wasserstein metric on
$(\cP_2(\rmH),\sfw_2)$, 
\ref{subsec:W2};\\
&$\ttM,\ttN,\cdots$&notation for typical
laws of random measures in $\cPP_2(\rmH)$;\\
&$f_\sharp \mu$&push forward of a measure $\mu$ via the map $f$, \ref{sec:prel};\\
&$f_{\sharp\sharp} \ttM=
(f_\sharp)_\sharp \ttM$&iterated push forward of a measure $\ttM$ via the map $f$ defined in $\rmH$;\\
&$\pi^i$&projection on the $i$-th coordinate in a product space, \ref{sec:prel};\\
&$\rmS_N=\operatorname{Sym}(\{1,\cdots,N\})$
&
the symmetric group of permutations 
of $\{1,\cdots,N\}$;\\
&$\Gamma(\mu,\nu),\ \Gamma_o(\mu,\nu)$&set of (resp.~optimal) couplings with marginals $\mu,\nu$, \ref{sec:prel};\\
&$\cP^{\rm det}(\rmX^1\times \rmX^2)$&
set of deterministic couplings, concentrated on the graph of a Borel map, \eqref{eq:91};\\
&$(\OOmega,\cF_\OOmega,\P),\ (\Thetao,\cF,\Q)$ &
standard Borel spaces endowed with nonatomic probability measures, \ref{subsec:lagrangian};\\
&$\cH=L^2(\OOmega,\cF_\OOmega,\P;\rmH)$& the $L^2$ space
of $\rmH$ valued maps, \ref{subsec:lagrangian};\\
&$\iota$&the law-pushforward map $\iota(X):=X_\sharp\P$ from $\cH$ to $\cP_2(\rmH)$, \ref{subsec:lagrangian};\\
&$\ii$&the identity vector field $\ii(x)=x$ in $\rmH$;\\
&$\id$&the identity map in a general set;\\
&$\frm,\frg$&typical probability measures
in $\cP_2(\cH)$;\\
&$\rmG(\OOmega)$&the group of measure preserving isomorphisms of $(\OOmega,\cF_\OOmega,\P)$,
\ref{subsec:lagrangian}\\
&$\hat \phi$&$=\phi\circ\iota$, the lifting
of a function defined in $\cP_2(\rmH)$ to
$\cH$, \ref{subsec:lagrangian};\\
&$\phi^\star$&the Kantorovich-Legendre-Fenchel conjugate of a function
$\phi$, Def.~\ref{def:measure-conj};\\
&$\boldsymbol{\rmF},\ \hat{\boldsymbol\rmF}$&a typical MPVF in 
$\cP_2(\rmH\times \rmH)$ and its lifting to $\cH\times \cH$, \ref{sec:MPVF};\\
&$\bpartialt\phi,
\bpartialt^\circ\phi$&the total subdifferential of $\phi$ and its minimal section, Def.~\ref{def:total-sub};\\
&$\nabla_W\phi$&the nonlocal field associated
to $\bpartialt^\circ\phi$;\\
&$\RGamma(\ttM,\ttN)$&random couplings between two laws in $\cPP_2(\rmH)$, 
\ref{susbec:random-couplings};\\
&$\mmsc\cdot\cdot$  &cost in $\cPP_2(\rmH)$
induced by the maximal correlation pairing $\msc\cdot\cdot$, 
Def.~\ref{def:msc2};
\\
&$\RWr2$&set of super-regular measures in $\cPP_2(\rmH)$, Def.~\ref{Random Gauss-null sets};\\
\end{longtable}
\end{small}
\end{center}

If $\rmX$ is a Polish topological space (i.e.~it is separable and its
topology is induced by a complete metric)
we denote by $\cP(\rmX)$ the space of Borel probability measures on
$\rmX$ endowed with the weak (Polish) topology in duality with continuous and
bounded real functions.
Given a Borel map 
between Polish spaces 
$f:\rmX\to \rmY$ and $\mu\in \cP(\rmX)$
we denote by $f_\sharp\mu$ the image measure $\mu\circ f^{-1}$ given by
$f_\sharp\mu(B):=\mu(f^{-1}(B))$ for every Borel subset $B$ of $\rmY$.
In the case of a cartesian product
of Polish spaces
$\boldsymbol \rmX:=\Pi_{i=1}^N\rmX^i$
we will denote by $\pi^i:\boldsymbol \rmX\to \rmX^i$ the map
$\pi^i(x_1,\cdots, x_N):=x_i$, and similarly
$\pi^{ij}(x_1,\cdots, x_N):=(x_i,x_j)$.

A probability kernel is a Borel map $\kappa:\rmX
\ni x\mapsto \kappa_x \in  \cP(\rmY)$; if
$\mu\in \cP(\rmX)$ there exists a unique Borel probability measure
$\mu\otimes \kappa
=\int_{\rmX\times \rmY}\delta_x\otimes \kappa_x\,\d\mu(x)
\in \cP(\rmX\times \rmY)$ which satisfies
\begin{equation}
  \label{eq:60}
  \int_{\rmX\times \rmY}f(x,y)\,\d( \mu\otimes\kappa)(x,y):=
  \int_{\rmX}\Big(\int_\rmY f(x,y)\,\d\kappa_x(y)\Big)\,\d\mu(x)
\end{equation}
for every bounded (or nonnegative) real Borel function $f$ defined in
$\rmX\times \rmY$.

Elements of $\cP(\rmX\times \rmY)$ are also called
\emph{couplings}
or \emph{transport plans}.
The universal disintegration Theorem \cite[Corollary 1.26]{Kallenberg17}
says that
there exists a Borel map
$\cK:\rmX\times \cP(\rmX\times \rmY)\to \cP(\rmY)$
such that
$\kappa_x:=\cK(x,\ggamma)$ provides a disintegration of $\ggamma$
with respect to its first marginal, i.e.
\begin{equation}
  \label{eq:61}
  \ggamma=\mu\otimes \kappa,\quad
  \mu:=\pi^1_\sharp\ggamma,\quad
  \kappa_x:=\cK(x,\ggamma)\quad x\in \rmX,
  \quad
  \text{for every }\ggamma\in \cP(\rmX\times \rmY).
\end{equation}
We will denote by $\cPdet{}(\rmX\times \rmY)$ the set of
deterministic couplings:
\begin{equation}
  \label{eq:71}
  \cPdet{}(\rmX\times \rmY):=
  \Big\{\ggamma\in \cP(\rmX\times \rmY):
  \ \exists f:\rmX\to\rmY\text{ Borel, such that }
  \ggamma=(\id\times f)_\sharp\mu,\ \mu=\pi^1_\sharp\ggamma\Big\}.
\end{equation}
Since disintegrations are uniquely defined almost everywhere
with respect to the marginal, we also have
\begin{equation}
  \label{eq:91}
  \begin{gathered}
    \ggamma\in \cPdet{}(\rmX\times \rmY) \quad\text{if and only
      if} \quad \cK(\cdot,\ggamma)\in \cP_\delta(\rmY) \text{
      $\pi^1_\sharp\ggamma$-a.e.}  \\
    \text{where}\qquad
    \cP_\delta(\rmY):=\Big\{\delta_y:y\in \rmY\Big\}.
  \end{gathered}
\end{equation}
It is possible to prove
\cite{LS25} that $  \cPdet{}(\rmX\times \rmY)$
is a $G_\delta$ (thus Borel) subset of $\cP(\rmX\times \rmY)$.


Given $\mu_i\in \cP(\rmX^i)$ we will denote by
$\Gamma(\mu_1,\cdots,\mu_N)$
the subset of $\cP(\boldsymbol\rmX)$ whose elements $\mmu$ satisfy
$\pi^i_\sharp(\mmu)=\mu_i$, $i=1,\cdots, N$.

We will often use the following consequence of Von Neumann selection
Theorem \cite[Theorem 13, pag. 127]{Schwartz73}:
\begin{theorem}
  \label{thm:VN}
  Let $\rmX,\rmY$ be Polish spaces,
  $\nu\in \cP(\rmY)$ concentrated on the Borel set $B\subset \rmY$
  and let $f:\rmX\to \rmY$ be a Borel map.
  If $f(\rmX)\supset B$ then there exists a measure
  $\mu\in \cP(\rmX)$ 
  such that
  \begin{equation}
    \label{eq:87}
    \text{$\mu$ is concentrated on $f^{-1}(B)$,}
    \quad
    f_\sharp\mu=\nu.
  \end{equation}
\end{theorem}

\subsection{The \texorpdfstring{$L^2$}{}-Wasserstein space}
\label{subsec:W2}
Let $(\rmX,\sfd_\rmX)$ be a complete and separable metric space and
let $x_o$ a point of $\rmX$.
We denote by $\cP_2(\rmX)$ the space of Borel probability measures on
$\rmX$
with finite quadratic moment
\begin{equation}
  \label{eq:3}
  \int_{\rmX} \sfd_\rmX^2(x,x_o)\,\d\mu(x)<\infty.
\end{equation}
It is easy to check that the definition is independent of the choice
of the reference point $x_o$.
The space $\cP_2(\rmX)$ can be
endowed with the $L^2$-Kantorovich-Wasserstein metric
$\sfW_{2,\sfd_\rmX}$
\begin{equation}
  \label{eq:2}
  \sfW^2_{2,\sfd_\rmX}(\mu_1,\mu_2):=
  \min\Big\{\int \sfd_\rmX^2(x_0,x_1)\,\d\mmu(x_0,x_1):\mmu\in \Gamma(\mu_1,\mu_2)\Big\};
\end{equation}
we will denote by $\Gamma_o(\mu_1,\mu_2)$
the (compact, non-empty) subset of $\Gamma(\mu_1,\mu_2)$ where
the minimum is attained.
We will also denote by $\cP_{2,o}(\rmX\times \rmX)$
the set of couplings $\mmu\in \cP_{2}(\rmX\times \rmX)$ such that
$\mmu\in \Gamma_o(\pi^1_\sharp\mmu,\pi^2_\sharp\mmu)$.

$(\cP_2(\rmX),\sfW_{2,\sfd_\rmX})$ is a complete and separable
metric space as well
\cite[Chap.~7]{AGS08}.
In this paper we will mainly consider three important cases,
when $\rmX=\rmH$ is a Hilbert space, when $\rmX=L^2(\OOmega,\P;\rmH)$
is the space of $\rmH$-valued 
(Bochner) square-integrable Lagrangian maps defined in some probability space $(\OOmega,\P)$, and  
when $\rmX=\cP_2(\rmH)$ is the Wasserstein space itself.
\paragraph{The Hilbertian case \texorpdfstring{$\cP_2(\rmH)$}{} and the maximal correlation pairing.} In this paper we will denote by 
$\rmH$
a separable Hilbert space  with 
scalar product
$\la\cdot,\cdot\ra$, norm $|\cdot|$, 
and induced metric $\sfd_\rmH$.
The finite dimensional case $\rmH=\R^d$ provides an important example
covered by the theory.

The previous construction 
applied to $(\rmH,\sfd_\rmH)$ 
yields the space $\rW2$; 
the quadratic moment of $\mu\in \rW2$ is
\begin{equation}
  \sfm^2_2(\mu):=
  \int_\rmH |x|^2\,\d\mu(x)
  =\int_\rmH\sfd_\rmH^2(x,0) d\mu(x),\quad
  \text{corresponding to the choice of }x_o:=0.
\end{equation}
To simplify notation, we will simply denote the Wasserstein metric
$\sfW_{2,\sfd_\rmH}$ by
$\sfw_2$:
\begin{equation}
  \label{eq:4}
  \sfw_{2}^2(\mu_1,\mu_2):=
  \min\Big\{\int_{\rmH\times \rmH} |x_0-x_1|^2\,\d\mmu(x_0,x_1):\mmu\in \Gamma(\mu_1,\mu_2)\Big\}.
\end{equation}
The Euclidean structure of $\sfd_\rmH$ allows for a useful
decomposition of $\sfw_2$.
\begin{definition}[The maximal correlation pairing]
    \label{def:MC-function}
    For every $\mu_1,\mu_2\in \cP_2(\rmH)$
    we set
    \begin{equation}
        \label{eq:MC}
        \msc{\mu_1}{\mu_2}:=
        \max         \Big\{\int_{\rmH\times \rmH} \la x_0,x_1\ra\,\d\mmu:
        \mmu\in \Gamma(\mu_1,\mu_2)\Big\}.
    \end{equation}
\end{definition}
\noindent
It is clear that $\msc\cdot\cdot$ is finite in $\cP_2(\rmH)$
and satisfies
\begin{gather}
    \label{eq:CS-inequality}
    \Big|\msc {\mu_1}{\mu_2}\Big|\le 
    \sfm_2(\mu_1)\,\sfm_2(\mu_2),\\
    \label{eq:W-decomp}\quad 
    \sfw_2^2(\mu_1,\mu_2)=
    \sfm_2^2(\mu_1)+\sfm_2^2(\mu_2)-2
    \msc{\mu_1}{\mu_2}\quad
    \text{for every }\mu_1,\mu_2\in \rW2.
\end{gather}
In particular, a coupling $\mmu$ belongs to $\Gamma_o(\mu_1,\mu_2)$ if
and only if it attains the maximum in \eqref{eq:MC}.
\paragraph{The space \texorpdfstring{$\RW2=\cP_2(\rW2)$}{}}
Since $(\rW2,\sfw_2)$ is a complete and separable metric space,
we can iterate the Wasserstein construction and consider the space
$\cP_2(\rW2)$, which we will also denote by $\RW2$;
its elements (also called laws of random measures) will be denoted by capital letters $\ttM,\ttN,\cdots$.

Using $\delta_0\in \cP_2(\rmH)$ as a reference measure and observing that
$\sfw_2^2(\mu,\delta_0)=\sfm_2^2(\mu)$, 
the quadratic moment in $\RW2$ is
\begin{equation}
  \label{eq:6}
  \sfM_2^2(\ttM):=\int_{\rW2}\sfm_2^2(\mu)\,\d\ttM(\mu)=
  \int_{\rW2}\int_\rmH |x|^2\,\d\mu\,\d\ttM(\mu).
\end{equation}
The corresponding Wasserstein
metric $\sfW_{2,\sfw_2}$
will be denoted by $\sfW_2$:
\begin{equation}
  \label{eq:5}
    \sfW_{2}^2(\ttM_0,\ttM_1):=
    \min\Big\{\int_{\rW2\times \rW2}
    \sfw_2^2(\mu_1,\mu_2)\,\d
    \pPi(\mu_1,\mu_2):\pPi\in \Gamma(\ttM_0,\ttM_1)\Big\}.
\end{equation}
\subsection{Regularity of measures and Monge formulation in \texorpdfstring{$\cP_2(\rmH)$}{}}
\label{subsec:regularity}
A particularly relevant case when the optimal coupling $\mmu$ 
solving \eqref{eq:4} or 
\eqref{eq:MC} is
unique and deterministic according to \eqref{eq:71},  is related to the regularity of (at least one of) the
marginals $\mu_i$.
The most refined description of such a regularity is characterized by the vanishing of 
the marginal $\mu_i$ 
on all the 
so-called \textrm{d.c.}~(or $\delta$-convex or {cc}) hypersurfaces, i.e.~graphs of difference of convex Lipschitz functions in a suitable coordinate system.
More precisely, we recall that a subset 
$S\subset \rmH$ is a \textrm{d.c.}~hypersurface
if we can find a decomposition of $\rmH$
in a orthogonal sum
$\rmH=\rmE\oplus v\R$, where
$v\in \rmH$ and $\rmE$ is the hyperplane of $\rmH$ orthogonal to $v$, and two Lipschitz convex functions
$f,g:\rmE\to\R$ such that 
\begin{equation}
    \label{eq:dc}
    S=\Big\{x+(f(x)-g(x))v:x\in \rmE\Big\}.
\end{equation}
When $\rmH=\R$, \textrm{d.c.}~hypersurfaces just
reduce to a point.
We say that 
\begin{equation}
    \label{eq:exceptional}    
    \text{$B\subset\rmH$ is 
    {\em $\sigma$-\textrm{d.c.}~hypersurface}
    if it can be covered by a countable union of \textrm{d.c.}~hypersurfaces.}
\end{equation}
A larger class of 
negligible subsets is provided by the so-called
Gaussian-null sets introduced by Phelps in 
\cite{Phelps78}: 
a Borel subset $B$ of $\rmH$ is Gaussian-null, if
$\frg(B)=0$ for every non-degenerate Gaussian measure $\frg\in \cP(\rmH)$
(see Section 
\ref{subsec:L-representation} below).
Thanks to
\cite{Csonryei99} 
the class of Gaussian null Borel sets 
coincides with 
the $\sigma$-ideals of 
cube null Borel sets and 
of Aronsajn null Borel sets 
\cite{Aronszajn76} (see also
\cite{Bogachev84} and 
the other comparisons and extensions discussed in \cite{Bogachev18}).
Since every \textrm{d.c.}~hypersurface is Aronsajn null, it 
is immediate to check that 
$\sigma$-\textrm{d.c.}~hypersurfaces sets 
are Gaussian-null.
\begin{definition}[Atomless, regular , and G-regular measures]
  \label{def:usual-regular}
  We denote by $\cP_2^{al}(\rmH)$
  the collection of \emph{atomless} probability measures in $\cP_2(\rmH)$,
  thus satisfying
    $\mu(\{x\})=0$ for every $x\in \rmH$.\\
We denote by $\cP_2^r(\rmH)$
the class of \emph{regular} probability measures that vanish on all \textrm{d.c.}~hypersurfaces (and therefore on all
$\sigma$-\textrm{d.c.}~hypersurfaces).\\
We denote by $\cP_2^{gr}(\rmH)$
the class of \emph{G-regular} 
probability measures, that vanish on
all Gaussian null Borel subsets of
$\rmH$.
\end{definition}
\begin{remark}
    \label{rem:cases}
    It is clear that $\cP_2^{gr}(\rmH)
    \subset \cP_2^{al}(\rmH).$
    It is useful to 
    recall other important relations 
    and some particular cases
    covered by the above definitions.
    \begin{enumerate}
        \item When $\rmH=\R$ is one-dimensional,
        the class of atomless and regular measures coincide, i.e.    $\cP_2^r(\R)=\cP_2^{al}(\R)$.
    \item 
    When $\rmH$ has finite dimension $d$,
    $\mu\in \cP_2^{gr}(\rmH)$ 
    if and only if it is absolutely
continuous
with respect to the $d$-dimensional Lebesgue measure.
\item 
Since \textrm{d.c.}~hypersurfaces are Gaussian null, 
it is immediate to check that 
the class of G-regular measures 
$\cP_2^{gr}(\rmH)$
is included in $\cP_2^r(\rmH).$
Therefore we have
\begin{equation}
    \label{eq:regularities}
    \cP_2^{gr}(\rmH)\subset 
    \cP_2^r(\rmH)\subset 
    \cP_2^{al}(\rmH),\qquad
    \cP_2^r(\R)=
    \cP_2^{al}(\R).
\end{equation}
\end{enumerate}
\end{remark}
The crucial property of 
regular measures $\mu\in \cP_2^r(\rmH)$
is that every
convex and Lipschitz function
$\varphi:\rmH\to \R$ is 
Gateaux-differentiable
$\mu$-almost everywhere.
This follows by a nice result of
Zaj\'i\v cek \cite{Zajicek79}
(see also \cite[Theorem 4.20]{BL00})
showing that 
the set of points where $\varphi:\rmH\to \R$
is not Gateaux-differentiable
is contained in an $\sigma$-\textrm{d.c.}~hypersurface according to \eqref{eq:exceptional}.

Using this property and the structure of optimal couplings,
it is possible to prove the celebrated Br\'enier Theorem.
\begin{theorem}
  \label{thm:Brenier}
  If $\mu_1\in \cP_2^r(\rmH), \mu_2\in \cP_2(\rmH)$ then
  $\Gamma_o(\mu_1,\mu_2)$ contains a unique element $\ggamma$
  which is deterministic, i.e.~concentrated on the graph of a Borel
  map
  $f\in L^2(\rmH,\mu_1;\rmH)$ with $f_\sharp\mu_1=\mu_2$.
\end{theorem}
We will  also be interested to measurability properties
of the above sets. We collect them in the following statement, in the case when $\rmH$ has finite dimension.
\begin{proposition}[Measurability of 
$\cP_2^r(\rmH)$ and $\cP_2^{gr}(\rmH)$]
    \label{prop:measurability}
    Let us assume that 
    $\rmH$ has finite dimension $d.$
    \begin{enumerate}
    \item 
    $\cP_2^r(\rmH)$ is a $G_\delta$ (thus Borel) subset of $\cP_2(\rmH).$
    \item
    $\cP_2^{gr}(\rmH)$ is a Borel subset of 
    $\cP_2(\rmH).$
    \end{enumerate}    
\end{proposition}
\begin{proof}
    \underline{Claim 1} ($d=1$, $\rmH=\R$). 
    It is sufficient to note that 
    $\mu\in \cP_2^r(\R)$ if and only if 
    $\mu\times \mu(D)=0$, 
    where 
    $D:=\{(x,x):x\in \R\} $ is the 
    (closed) diagonal of $\R^2$.
    Since the map $\mu\mapsto \mu\times \mu(D)$
    is upper semicontinuous, 
    we have 
    $\cP_2^r(\R)=
    \bigcap_{k\in \N_+}
    \big\{\mu\in \cP_2(\R):
    \mu\times \mu(D)>1/k\big\}.$

    \smallskip\noindent
    \underline{Claim 1} ($\rmH=\R^d$, $d>1$).
    Let $\cB$ 
    be the collection of all
    the compact boxes 
    $B=\Pi_{k=1}^{d-1}
    [a_k,b_k]\subset \R^{d-1}$
     with rational endpoints
    $a_k<b_k$, let 
    $\cR$ 
    a dense countable set of rotations $R$ of $\R^d$.
    For every $L\in \N_+$ and $B\in \cB$
    we consider the set 
    $$\rmC(B,L):=\{f:B\to \R:
    f\text{ convex, $L$-Lipschitz, $\sup_B|f|\le L$}\}.$$
    $\rmC(B,L)$ is a compact set of the
    Polish space $\rmC(B).$
    For every $g,h\in \rmC(B,L)$ and $R\in \cR$
    we consider the compact graph
    $$\rmS_{B,R}(g,h):=
    \Big\{R(y,g(y)-h(y)):y\in B\Big\}\subset \R^d.$$
    Since every
    d.c.-hypersurface can be covered by
    a countable collections of 
    sets of  the 
    form 
    $\rmS_{B,R}(g,h)$
    for $g,h$ in some $\rmC(B,L)$,
    we have
    \begin{displaymath}
        \begin{aligned}
        \cP_2^r(\R^d)={}&\bigcap\Big\{
        G(B,R,L):B\in \cB,\ R\in \cR,L\in \N_+\Big\},\\
        G(B,R,L):={}&
        \Big\{\mu\in \cP_2(\R^d):
        \mu(\rmS_{B,R}(g,h))=0
        \text{ for every }
        g,h\in \rmC(B,L)\Big\}.
        \end{aligned}
    \end{displaymath}
    It is then sufficient to show that 
    each set $ G(B,R,L)$ is a $G_\delta$. 

    It is not difficult to check that 
    the map
    $(\mu,g,h)\mapsto \mu(\rmS_{B,R}(g,h))$
    is jointly upper semicontinuous
    in $\cP_2(\R^d)\times (\rmC(B,L))^2$
    (uniform convergence in 
    $(\rmC(B,L))^2$ implies 
    Hausdorff convergence of the compact graphs 
    $\rmS_{B,R}(\cdot,\cdot)$) so that 
    the function 
    $$\mu\mapsto U(\mu):=
    \sup_{g,h\in \rmC(B,L)} \mu(\rmS_{B,R}(g,h)$$
    is upper semicontinuous as well, thanks to
    the compactness of $\rmC(B,L).$
    On the other hand
    \begin{displaymath}
        G(B,R,L)=
        \bigcap_{k\in \N_+}
        \Big\{\mu\in \cP_2(\R^d):U(\mu)<1/k\Big\}
    \end{displaymath}
    so it is the intersection of 
    a countable collection of open sets.

    \smallskip\noindent
    \underline{Claim 2.}
    Let $\frg$ denote the 
    standard Gaussian measure in $\R^d$
    and let $C$ be
    the $F_\sigma$ (thus Borel) set 
    $$C:=\Big\{f\in L^1(\R^d,\frg):
    f\ge 0,\ \int f(x)\,\d\frg(x)=1,\ 
    \int |x|^2f(x)\,\d\frg(x)<\infty\Big\}$$
    of the Polish space $L^1(\R^d;\frg_d).$
    The map 
    $J:C\to \cP_2(\R^d)$
    defined by 
    $J:f\mapsto f\frg$,
    is continuous and injective 
    so that  its image    
    $J(C)=\cP_2^r(\R^d)$ is a Borel subset of 
    $\cP_2(\R^d)$ by Lusin Theorem.
\end{proof}

\subsection{Lagrangian representations.}
\label{subsec:lagrangian}
Let us consider a standard Borel space $(\OOmega,\cF_\OOmega)$
endowed with a reference
nonatomic Borel probability measure $\P$.
We denote by
$\cH$ the Hilbert space
$L^2(\OOmega,\cF_\OOmega,\P;\rmH)$ endowed with the scalar product and norm
\begin{equation}
  \label{eq:10}
  \la X,Y\ra_\cH:=\int \la X(\oomega),Y(\oomega)\ra\,\d\P(\oomega),\quad
  \|X\|_\cH^2:=\la X,X\ra_\cH.
\end{equation}
A map $\sfg:\OOmega\to \OOmega$ is a \emph{measure preserving
  isomorphism} (m.p.i.) if it is $\cF_\OOmega-\cF_\OOmega$-measurable, there exists a set of full measure
$\OOmega_0\subset \OOmega$ such that $\sfg\restr \OOmega_0$ is
injective and $\sfg_\sharp\P=\P$. We denote by $\rmG(\OOmega)$ the
group of measure-preserving isomorphisms.
A m.p.i.~$\sfg\in \rmG(\OOmega)$ induces a linear isometry $\sfg^*:\cH\to\cH$ of $\cH$ by 
the map
\begin{equation}
  \label{eq:69}
  \sfg^* X:=X\circ \sfg.
\end{equation}
There is a natural 
$1$-Lipschitz surjective map
$\iota:\cH\to\cP_2(\rmH)$
given by
\begin{equation}
  \label{eq:65}
    \iota(X):=X_\sharp \P\quad\text{for every }X\in \cH,
  \end{equation}
  and satisfying
\begin{equation}
  \label{eq:8}
  \sfm_2(\iota(X))=\|X\|_\cH,\quad
  \sfw_2(\iota( X_1),\iota(X_2))\le \|X_1-X_2\|_{\cH},\quad
  \la X_1,X_2\ra_\cH\le
  \msc{\iota(X_1)}{\iota(X_2)}.
\end{equation}
Notice that
\begin{equation}
  \label{eq:63}
  \sfg\in \rmG(\OOmega),\quad
  Y=X\circ \sfg\quad\Rightarrow\quad
  \iota(Y)=\iota(X),
\end{equation}
i.e.~$\iota\circ \sfg^*=\iota$ for
every $\sfg\in \rmG(\OOmega)$.
The next Lemma shows that we can considerably refine the previous
inequalities.
We refer to \cite[Sect.~3]{CSS2} for the proof.
\begin{lemma}
  \label{le:param}
  For every $k\in \N$ the map $\iota^k:\cH^k\to
  \cP_2(\rmH^k)$ defined by
  \begin{equation}
    \label{eq:37-1}
    \iota^k(X_1,X_2,\cdots,X_k):=(X_1,X_2,\cdots,X_k)_\sharp\P
  \end{equation}
  is surjective. In particular, for every
  $\ggamma\in \cP_{2,o}(\rmH\times \rmH)$
  there exists a pair $(X_{\ggamma,1},X_{\ggamma,2})\in \cH\times \cH$ such that
  $\iota^2(X_{\ggamma,1},X_{\ggamma,2})=\ggamma$ and 
  for every $(X_1,X_2)\in \cH\times \cH$ we have
  \begin{equation}
    \label{eq:33*}
    \iota^2(X_1,X_2)\in \cP_{2,o}(\rmH\times \rmH)
    \ \Leftrightarrow\ 
  \|X_1-X_2\|_{\cH}=\sfw_2(\iota( X_1),\iota(X_2))
  \ \Leftrightarrow\
   \la X_1,X_2\ra_\cH=
   \msc{\iota(X_1)}{\iota(X_2)}.
 \end{equation}
 Moreover, for every $\mu_1,\mu_2\in \rW2$ and $X_1,X_2\in \cH$
 with $\iota(X_i)=\mu_i$ we have
 \begin{equation}
   \label{eq:66}
   \msc {\mu_1}{\mu_2}=\sup_{\sfg\in \rmG(\OOmega)}\la X_1,\sfg^* X_2\ra,\quad
   \sfw_2^2(\mu_1,\mu_2)=\inf_{\sfg\in \rmG(\OOmega)}\|X_1-\sfg^* X_2\|_\cH^2.
 \end{equation}
  Finally, if $\mu_1=\mu_2$ (i.e.~$\iota(X_1)=\iota(X_2)$) then there
  exists a sequence $\sfg_n\in \rmG(\OOmega)$ such that
  $\sfg_n^*X_1\to X_2$ strongly in $\cH$.
\end{lemma}
Every function $\phi:\cP_2(\rmH)\to\Rinf$ induces a function
$\hat\phi:=\phi\circ\iota$ on $\cH$ which satisfies the obvious
law-invariance property
\begin{equation}
  \label{eq:62}
  \iota(X)=\iota(Y)\quad\Rightarrow\quad
  \hat\phi(X)=\hat\phi(Y),
\end{equation}
and, in particular, is invariant by the action of
measure-preserving isomorphisms:
\begin{equation}
  \label{eq:12}
  \hat \phi(\sfg^* X)=\hat\phi(X)\quad
  \text{for every }\sfg\in \rmG(\OOmega).
\end{equation}
\begin{lemma}
  \label{le:invariance}
  Let $\phi:\rW2\to\Rinf$ and let $\hat \phi:=\phi\circ\iota:\cH\to\Rinf$. We have
  \begin{enumerate}
  \item $\phi$ is proper if and only if $\hat\phi$ is proper.
  \item $\phi$ is l.s.c.~if and only if $\hat\phi$ is l.s.c.
  \item $\phi$ is continuous if and only if $\hat\phi$ is continuous.
  \item $\phi$ is $L$-Lipschitz if and only if $\hat\phi$ is $L$-Lipschitz.
  \end{enumerate}
  Moreover, suppose that $\Phi:\cH\to\Rinf$ is
  proper, l.s.c.~and invariant by measure preserving isomorphisms,
  i.e.~$\Phi\circ\sfg^*=\Phi$.
  Then $\Phi$ is law-invariant and $\Phi=\phi\circ\iota$ for a
  (unique) proper, l.s.c.~function $\phi:\rW2\to\Rinf$.
\end{lemma}
\begin{proof}
  The first claim and all the left-to-right implications of Claims
  2,3,4 are trivial, thanks to \eqref{eq:8}. We thus consider only the right-to-left implications.

  Concerning Claim 2, let us suppose that $\hat\phi$ is lower semicontinuous and let 
  $(\mu_n)_{n\in\N}$ converging to $\mu$ in $\cP_2(\rmH)$
  with $\phi(\mu_n)\le c$. By \eqref{eq:66}
  there exists a sequence $X_n$ converging to $X$ in $\cH$ such that
  $\mu_n=\iota(X_n)$ and $\mu=\iota(X)$. Since $\hat\phi$ is lower
  semicontinuous we deduce that $\phi(\mu)=\hat\phi(X)\le c$ as well.

  Claim 3 immediately follows by Claim 2 applied to $\phi$ and
  $-\phi$.
  Claim 4 follows from \eqref{eq:33*}.

  Let now $\Phi$ be as in the last statement of the Lemma
  and let $X_i\in \cH$ with $\iota(X_1)=\iota(X_2)$.
  By Lemma \ref{le:param} we can find a sequence $\sfg_n\in
  \rmG(\OOmega)$ such that
  $\sfg_n^*X_1\to X_2$ as $n\to\infty$ so that
  \begin{displaymath}
    \Phi(X_2)\le \liminf_{n\to\infty}\Phi(\sfg_n^* X_1)=
    \liminf_{n\to\infty}\Phi( X_1)=\Phi(X_1).   
  \end{displaymath}
  Inverting the role of $X_1$ and $X_2$ we deduce that $\Phi(X_1)=\Phi(X_2).$
\end{proof}

\subsection{Lagrangian representation of the laws of random measures of \texorpdfstring{$\cPP_2(\rmH)$}{}}
\label{subsec:L-representation}

The $1$-Lipschitz and surjective law map $\iota:\cH\to \cP_2(\rmH)$
provides a natural way 
to construct measures 
in $\cPP_2(\rmH)$
(the space of laws of random measures in 
$\cP_2(\rmH)$) starting
from measures in $\cP_2(\cH)$
(the space of laws of random Lagrangian maps in $\cH$).
In fact, the corresponding push-forward transform $\iota_\sharp$ is
a surjective map from $\cP_2(\cH)$
to $\RW2$ so that 
\begin{equation}
    \label{eq:obvious}
    \text{for every $\frm\in \cP_2(\cH)$
    the push-forward
    $\ttM=\iota_\sharp \frm$ belongs to 
    $\cPP_2(\rmH).$}
\end{equation}
In this way, we will use (maps from) $(\OOmega,\P)$ as 
a sort of ``labelling'' space for measures 
in $\rmH$
and not as a source of randomness. 
The latter can be introduced by using a further measure $\frm$, not in $\OOmega$ but in the Hilbert space $L^2(\OOmega,\P; \rmH)$;
in turn $\frm$ can be represented
as the law of a random vector $\xxi$
defined in another 
(standard Borel)
probability space
$(\Thetao,\cF,\Q)$. 
Thus in many situations
it will be useful to 
deal with the pair of spaces 
$(\OOmega,\cF_\OOmega,\P)$ and $(\Thetao,\cF,\Q)$.
Such a construction has been used, with different aims, in various contexts,
see e.g.~\cite{Konarovskyi-VonRenesse24},
\cite{VonRenesse-Sturm09,Sturm11}.

Let us briefly describe this simple construction: we can assume that 
\begin{equation}
    \label{eq:representation}
    \begin{gathered}
        \text{the measure $\frm\in \cP_2(\cH)$ 
is the law of a $\cH$-valued random vector $\xxi$ defined 
in }(\Thetao,\cF,\Q):
        \\
        \text{$\frm=\xxi_\sharp \Q$
        \quad 
and
\quad $\ttM=(\iota\circ \xxi)_\sharp \Q$}
\text{ is the law of the random measure }
\mu_\cdot =\xxi[\cdot]_\sharp \P.
    \end{gathered}
\end{equation}
We first represent $\frm$ as 
the distribution of a $\cF\otimes \cF_\OOmega$-measurable stochastic process.
\begin{lemma}
    \label{le:representation}
    If $\xxi:\Omega\to \cH$
    is a Borel vector field satisfying \eqref{eq:representation}
    then there exists a 
    $\cF\otimes \cF_\OOmega$-measurable stochastic process
    $\Xi:\Thetao\times \OOmega\to \rmH$
such that 
\begin{equation}
    \label{eq:joint}
    \text{for $\Q$-a.e.~$\thetao\in \Thetao$}
    \quad
    \Xi(\thetao,\cdot)=
    \xxi[\thetao]
    \quad \text{$\P$-a.e.~in $\OOmega.$}
\end{equation}
\end{lemma}
\begin{proof}
    We select an orthonormal basis
    $(\sfh_k)_{k\in \N}$ of 
    $\rmH$
    and set $\xi_k:=\langle \xxi,\sfh_k\rangle$.
    $\xi_k$ is a Borel map from 
    $\Thetao$ to $L^2(\OOmega,\P)$
    satisfying
    \begin{equation}
        \label{eq:bounds}
        \sum_k \int 
        \|\xi_k[\thetao]\|_{L^2(\OOmega,\P)}^2\,\d\Q(\thetao)
        =
        \int \|\xxi[\thetao]\|_{\cH}^2
        \,\d\Q(\thetao)=
        \int \|X\|_\cH^2\,\d\frm(X)<\infty.
    \end{equation}
    We can then set $\nu_k(\thetao,B):=
    \int \xi_k[\thetao]\nchi_B\,\d\P$
    for every $\thetao\in \Thetao$ and
    $B\in \cF_\OOmega$
    obtaining a family of signed measures
    $\nu_k(\thetao,\cdot)$ on $\cF_\OOmega$
    depending on $\thetao$ in a 
    $\cF$-measurable way, with
    $$ |\nu_k|(\thetao,B)\le 
    \|\xi_k[\thetao]\|_{L^2(\OOmega,\P)}
    \P(B)$$
        By the Doob's measurable Radon-Nykodim theorem
    \cite[Thm. 58]{Dellacherie-MeyerB}
    (see also 
    \cite[Ex.~6.10.72]{Bogachev07})
    we can 
    find a 
    $\cF\times \cF_\OOmega$-measurable density
    $f_k:\Thetao\times \OOmega\to\R$
    such that 
    $\nu_k(\thetao,\cdot)=
    f_k(\thetao,\cdot)\P$
    so that 
    \begin{displaymath}
        \xi_k[\thetao]=
        f_k(\thetao,\cdot)
        \quad 
        \text{$\P$-a.e.~and}
        \quad 
        \int f_k^2(\thetao,\oomega)
        \,\d\P(\oomega)=
        \|\xi_k[\thetao]\|_{L^2(\OOmega,\P)}^2
        \quad \text{for 
        $\Q$-a.e.~$\thetao\in \Thetao.$}
    \end{displaymath}
    Setting
    \begin{displaymath}
        \Xi_n(\thetao,\oomega):=
        \sum_{k=1}^nf_k(\thetao,\oomega)\sfh_k
    \end{displaymath}
    and using \eqref{eq:bounds}
    it is not difficult to check that 
    $\Xi_n$ converges pointwise 
    $\Q\times \P$-a.e.~to 
    an element $\Xi\in L^2(\Thetao\times
    \OOmega,\Q\times \P;\rmH)$
    which satisfies \eqref{eq:joint}.
\end{proof}
We can use $\Xi$ to respresent $\ttM$
and its $k$-projections:
for every $\thetao\in \Thetao$ 
the measure $\mu_\thetao=
\iota(\xxi(\thetao))
\in \cP(\rmH)$
satisfies
\begin{equation}
    \label{eq:push2}
    \int_\rmH \zeta(x)\,\d\mu_\thetao(x)=
    \int_\OOmega
    \zeta(\Xi(\thetao,\oomega))\,\d
    \P(\oomega)
\end{equation}
and we can recover 
$\ttM$ as the law with respect to $\Q$
of the random measure $\mu_\thetao$:
\begin{equation}
    \label{eq:M}
    \ttM=\int_\Thetao \delta_{\mu_\thetao}
    \,\d \Q(\thetao)
    =
    \int_\Thetao \delta_{\Xi(\thetao,\cdot)_\sharp \P}
    \,\d \Q(\thetao).
\end{equation}
By using $\Xi$ we can easily express 
$k$-projections of $\ttM$. 
For every $k\in \N$ and $\ttM\in \cPP_2(\rmH) $ let us define
\begin{equation}
    \label{eq:k-proj}
    \operatorname{pr}^k[ \ttM]:=
    \int \mu^{\otimes k}\,\d \ttM(\mu)
    \in \cP_2(\rmH^k),
\end{equation}
which satisfies
\begin{equation}
    \label{eq:k-integrals1}
    \int_{\rmH^k}
    \zeta\,\d
    \operatorname{pr}^k[ \ttM]=
    \int\Big(
    \int_{\rmH^k}
    \zeta\,\d\mu^{\otimes k}\Big)
    \,\d \ttM(\mu)
    \quad \text{for every bounded Borel }\zeta:\rmH^k\to \R.
\end{equation}
We then have
\begin{equation}
     \label{eq:k-integrals}
    \int_{\rmH^k}
    \zeta\,\d
    \operatorname{pr}^k[ \ttM]=
    \int_\Thetao 
    \Big(
    \int_{\OOmega}
    \zeta(\Xi(\thetao,\oomega_1),
    \cdots,\Xi(\thetao,\oomega_k))
    \,\d\P^{\otimes k}(\oomega_1,\cdots,\oomega_k)\Big)
    \,\d \Q(\thetao),
\end{equation}
i.e.~
\begin{equation}
    \label{eq:k3}
    \operatorname{pr}^k[ \ttM]=\Xi^k_\sharp (\Q\otimes \P^{\otimes k})
    \quad\text{where}\quad
\Xi^k(\thetao,\oomega_1,\cdots,\oomega_k):=(\Xi(\thetao,\oomega_1),\cdots,\Xi(\thetao,\oomega_k)).
\end{equation}
\begin{example}[Laws of Gaussian generated random measures (LGRRM)]
    \label{ex:gaussian}
    Let us focus on the particular case
 of a nondegenerate centered Gaussian measure 
$\frg\sim N(0,K)$ in $\cH$
with covariance operator $K$.
By Karhunen-Lo\`eve expansion, 
we can find 
\begin{enumerate}
    \item an orthonormal basis $(\sfE_n)_{n\in \N_+}$ of $\cH$,
    given by the eigenvectors of $K$;
    \item the corresponding sequence of 
    nonnegative eigenvalues 
    of $K$ $(\lambda^2_n)_{n\in \N_+}$
    with 
    $K\sfE_n=\lambda^2_n\sfE_n$ and 
    $\Lambda:=\sum_n \lambda_n^2<\infty $    
    \item a sequence of 
    independent Gaussian random variables 
    $(\xi_n)_{n\in \N_+}$
    defined in $\Thetao$
    with $\xi_n\sim N(0,\lambda_n^2)$
\end{enumerate}
such that 
\begin{equation}
\label{eq:G-representation}
    \frg=\xxi_\sharp \Q,\quad 
    \xxi:=\sum_n \xi_n \sfE_n.
\end{equation}
Since 
$\Xi_n(\thetao,\oomega):=\xi_n(\thetao)
\sfE_n(\oomega)$
is an orthogonal system in 
$L^2(\Thetao\times \OOmega,\Q\otimes \P)$
the series
\begin{equation}
    \label{eq:joint-expansion}
    \Xi(\thetao,\oomega):=
    \sum_n \Xi_n(\thetao,\oomega)=
    \sum_n \xi_n(\thetao)
\sfE_n(\oomega),\quad
\|\Xi_n\|_{L^2(\Thetao\times \OOmega,\Q\otimes \P)}=
\lambda_n
\end{equation}
converges in 
$L^2(\Thetao\times \OOmega,\Q\otimes \P)$
and provides a measurable version
satisfying \eqref{eq:joint}.
\end{example}
\begin{example}
    \label{ex:gaussian2}
We can also proceed in a slightly different way:
assuming that $\OOmega$ is a compact metrizable space, we can start from a $\rmH$-valued
measurable process $\Xi_\oomega=\Xi(\cdot,\oomega)$
indexed by $\oomega\in \OOmega$
with continuous paths, i.e.~$\oomega\mapsto 
\Xi(\thetao,\oomega)\in \rmC^0(\OOmega)$
for $\Q$-a.e.~$\thetao$.
In this way the map $\xxi:\thetao \to 
\Xi(\thetao,\cdot)$ 
can be considered as a measurable map from 
$\Thetao$ to the Banach space
$\cB:=\rmC^0(\OOmega;\rmH)$;
assuming that 
$\xxi\in L^2(\Thetao,\Q;\cB)$, the distribution 
of the process $\frm=\xxi_\sharp\Q$
is a Radon measure in $\cP_2(\cB)$.
Any choice of diffuse measure 
$\P\in \cP(\OOmega)$
induces a push-forward map
$\iota_\P(X):=X_\sharp\P$ 
and a measure $\ttM=
(\iota_\P)_\sharp \frm$ as in \eqref{eq:M}.
If in particular $\Xi$ is a Gaussian process,
then $\frm$ is a Gaussian measure in $\cB$
and thus also in $\cH$.
\end{example}
\section{Totally convex functionals}
\label{sec:totally-convex}
We first recall the definition of totally convex functionals on
$\cP_2(\rmH)$ \cite[Sec.~5]{CSS22}
\begin{definition}[Totally convex functionals]
  \label{def:totally-convex}
  A functional $\phi:\cP_2(\rmH)\to \Rinf$ is \emph{totally
    convex}
  if for every coupling $\mmu\in \cP_2(\rmH\times \rmH)$
  and $t\in [0,1]$ 
  \begin{equation}
    \label{eq:1}
    \phi(\mu_t)\le (1-t)\phi(\mu_0)+\phi(\mu_1)\quad
    t\in [0,1],\quad
    \mu_t=(\pi^{1\shortrightarrow2}_t)_\sharp\mmu,\quad
    \pi^{1\shto2}_t(x_1,x_2):=(1-t)x_1+tx_2.
  \end{equation}
  Equivalently, the lifted functional $\hat\phi:=\phi\circ\iota$ is
  convex in $\cH$.

  We say that $\phi$ is totally $\lambda$-convex if
  $\phi-\frac \lambda 2\sfm_2^2$ is totally convex
  (equivalently $\hat \phi$ is $\lambda$-convex in $\cH$).
\end{definition}
Notice that the lifting $\phi\to\hat\phi$  inherits also properness
and lower semicontinuity.
Conversely, if $\Phi:\cH\to \Rinf$ is convex, lower semicontinuous and
invariant by m.p.i.~then there exists a unique totally convex
functional $\phi$ such that $\Phi=\hat \phi=\phi\circ \iota$.

If we restrict the convexity inequality \eqref{eq:1} to displacement
interpolation induced by \emph{optimal couplings} $\mmu$, we obtain a
larger class of functionals, which are called \emph{geodesically (or
  displacement)} convex, according to \cite{McCann97}.
However if $\dim(\rmH)\ge2$, every \emph{continuous} geodesically convex functionals
$\phi$ is also totally convex \cite[Thm.~5.5]{CSS22}.
We quote here a few important examples: 
\begin{example}[Potential energy]
  If $f:\rmH\to \Rinf$ is $\lambda$-convex (proper, l.s.c.), then 
  $$V_f:\mu\mapsto \int_\rmH f(x)\,\d\mu(x)$$
  is totally $\lambda$-convex (proper, l.s.c.) in $\cP_2(\rmH)$.
  In particular, $\sfm_2^2$ (corresponding to $f(x)=|x|^2$) is totally
  $1$-convex.
\end{example}
\begin{example}[Multiple interaction energy]
  If $g:\rmH^k\to \Rinf$ is proper, convex, l.s.c., and invariant
  respect to arbitrary permutations of its entries, then
  $$W_g:\mu\mapsto \int g(x_1,x_2,\cdots,x_k)\,\d\mu^{\otimes
    k}(x_1,\cdots,x_k)$$
  is totally convex (proper and
  l.s.c.) in $\cP_2(\rmH)$.
\end{example}

\normalcolor

\begin{example}
      For every $\nu\in \cP_2(\rmH)$ we 
      define the maximal pairing functional
  $\mathsf k_\nu:\cP_2(\rmH)\to\R$ defined by 
  \begin{equation}
    \label{eq:7}
    \mathsf k_\nu(\mu):= 
    \msc \mu\nu\quad \text{for every }\mu\in \cP_2(\rmH).
  \end{equation}
\end{example}
\begin{proposition}
    \label{prop:total-convexity-C}
    For every $\nu\in \cP_2(\rmH)$
    the function
    $\mathsf k_\nu$ 
    is $\sfm_2(\nu)$-Lipschitz in $\cP_2(\rmH)$ and
    totally convex.
\end{proposition}
\begin{proof}
  We give two proofs. The first one is direct: let us fix $\nu,\mu_1,\mu_2\in \cP_2(\rmH)$
    and let 
    $\mmu^{12}\in \Gamma(\mu_1,\mu_2)$.
    Let us set
    \begin{equation}
    \pi^{1\shto 2}_t:\rmH\times \rmH\to \rmH,\quad
    \pi^{1\shto2}_t(x_1,x_2):=(1-t)x_1+tx_2,\quad
    t\in [0,1],
    \label{eq:47}
  \end{equation}
    and $\mu^{1\shortrightarrow2}_t:=(\pi^{1\shto2}_t)_\sharp \mmu.$
    We have to prove that 
    \begin{equation}
        \label{eq:to-prove}
        \mathsf k_\nu(\mu^{1\shortrightarrow2}_t)\le 
        (1-t)\mathsf k_\nu(\mu_1)+t\mathsf k_\nu(\mu_2)
        \quad \text{for every }t\in [0,1].
    \end{equation}
    We select $\mmu_t\in \Gamma_o(\mu_t,\nu)$
    and we apply 
    \cite[Prop.~7.3.1]{AGS08}
    to find $\mmu\in \Gamma(\mu_1,\mu_2,\nu) \subset \cP_2(\rmH\times\rmH \times\rmH)$
    such that 
    $\pi^{12}_\sharp \mmu=\mmu^{12}$ and
    $(\pi^{1\shto2}_t,\pi^3)_\sharp\mmu=\mmu_t$.
    It follows that 
    \begin{align*}
        \mathsf k_\nu(\mu^{1\shortrightarrow2}_t)
        &=
        \int \la x,z\ra\,\d\mmu_t(x,z)
        =
        \int \la \pi^{1\shto2}_t(x_1,x_2),x_3\ra\,\d\mmu (x_1,x_2,x_3)
        \\&=
        \int \la (1-t) x_1+t x_2, x_3\ra\,\d\mmu (x_1,x_2,x_3)
      \\&=
        (1-t)
        \int \la x_1,x_3\ra\,\d\mmu (x_1,x_2,x_3)
        +
        t
        \int \la x_2,x_3\ra\,\d\mmu (x_1,x_2,x_3)
        \\&\le 
        (1-t)\mathsf k_\nu(\mu_1)+t\mathsf k_\nu(\mu_2).
    \end{align*}
    The second argument uses the representation result \eqref{eq:66};
    we fix $Y\in \cH$ such that $\iota(Y)=\nu$ and observe that 
    \begin{equation}
      \label{eq:68}
      \hat{\mathsf k}_\nu(X)={\mathsf k}_\nu(\iota(X))=\sup\Big\{\la
      X,Y\circ\sfg\ra:\sfg\in \rmG(\OOmega)\Big\},
    \end{equation}
    so that $\hat{\mathsf k}_\nu$ is the supremum of a family of
    $L$-Lipschitz (with $L=\|Y\|_\cH=\sfm_2(\nu)$) and
    linear functionals: therefore, it is convex and $L$-Lipschitz as well.
  \end{proof}
  \subsection{A Kantorovich version of the Legendre-Fenchel transform
    in \texorpdfstring{$\cP_2(\rmH)$}{}.}
    \label{subsec:KLF}
Recall that the Legendre-Fenchel transform of a function $\Phi:\cH\to
\Rinf$ is defined as
\begin{equation}
  \label{eq:11}
  \Phi^{\Ast}(Y):=\sup_{X\in \cH}\la Y,X\ra_\cH-\Phi(X).
\end{equation}
Inspired by this formula we define a corresponding transformation
for functionals on $\rW2$.
\begin{definition}[Kantorovich-Legendre-Fenchel transformation in $\cP_2(\rmH)$]
  \label{def:measure-conj}
  Let $\phi:\cP_2(\rmH)\to (-\infty,+\infty]$
  be a proper function.
  We call Kantorovich-Legendre-Fenchel conjugate of 
$\phi$ 
  the function   $\phi^\star:\cP_2(\rmH)\to (-\infty,+\infty]$ defined by
  \begin{equation}
    \label{eq:MLF}
    \phi^\star(\nu):=
    \sup_{\mu\in \cP_2(\rmH)}
    \msc\nu\mu-\phi(\mu).
  \end{equation}
\end{definition}
It is clear that $\phi^\star$ is totally convex and lower
semicontinuous, as the supremum of totally convex and continuous functions. It is also proper (i.e.~not identically $=\infty$)
if $\phi$ satisfies the lower bound
  \begin{equation}
  \label{eq:lower-B}
  \phi(\mu)\ge -a+\msc \mu\nu\quad
  \text{for every }\mu\in \cP_2(\rmH)
\end{equation}
for some $a\in \R$ and $\nu\in \cP_2(\rmH)$.

Every geodesically convex function
$\phi:\cP_2(\rmH)\to\Rinf$
is
linearly bounded from below \cite{Muratori-Savare20}, in the sense
that
there exist constants $a,b\ge0$ such that
\begin{equation}
  \label{eq:linear-BB}
  \phi(\mu)\ge -a-b\,\sfm_2(\mu)\quad 
\end{equation}
When $\phi$ is totally convex, we immediately get
a refined lower bound.
\begin{lemma}
  \label{le:lower-B}
  If $\phi:\cP_2(\rmH)\to\Rinf$ is totally convex and lower semicontinuous then
  there exist $a\in\R$ and $\nu\in \cP_2(\rmH)$ such that
  \eqref{eq:lower-B} holds.
In particular $\phi^\star(\nu)\le a$ so that $\phi^\star$ is proper.
\end{lemma}
\begin{proof}
  It is sufficient to observe that $\hat\phi:=\phi\circ\iota$ is
  convex and lower semicontinuous in $\cH$ so that
  there exist $a\ge 0$ and $Y\in \cH$ such that
  \begin{displaymath}
    \hat\phi(X)\ge -a-\la X,Y\ra\quad\text{for every }X\in \cH.
  \end{displaymath}
  Setting $\nu:=\iota(Y)\in \cP_2(\rmH)$ and
  using the fact that $\hat\phi$ is invariant
  w.r.t.~measure-preserving isomorphisms, we get
  for every $\mu=\iota(X)\in \cP_2(\rmH)$
  and $\sfg\in \rmG(\rmH)$
  \begin{displaymath}
    \phi(\mu)=\hat\phi(X)
    =\hat\phi(X\circ \sfg^{-1})\ge -a+\la  X\circ\sfg^{-1},Y\ra=
    -a+\la  X,Y\circ\sfg\ra.
  \end{displaymath}
  Taking the supremum w.r.t.~$\sfg\in \rmG(\rmH)$
  and recalling 
  \eqref{eq:8} and \eqref{eq:66}
  we get
  \begin{displaymath}
    \phi(\mu)\ge 
    -a+
    \sup_{\sfg\in \rmG(\rmH)}
    \la  X,Y\circ\sfg\ra=
    -a+\msc \nu\mu. 
  \end{displaymath}
\end{proof}
We collect a few simple but relevant properties of the
Kantorovich-Legendre-Fenchel transform in the following Theorem.
\begin{theorem}
  \label{thm:obvious}
  Let $\phi:\cP_2(\rmH)\to\Rinf$ be a proper function satisfying
  \eqref{eq:lower-B} for some $a\in \R$ and $\nu\in \cP_2(\rmH)$.
  \begin{enumerate}
  \item The function $\phi^\star$ is proper,
    totally convex and lower semicontinuous.
  \item $\phi^\star$ satisfies the commutation property
    \begin{equation}
      \label{eq:9}
      (\phi^\star)\circ\iota=(\phi\circ\iota)^\Ast
    \end{equation}
    and it is the unique function satisfying \eqref{eq:9}.
    \item The function $\phi^{\star\star}=(\phi^\star)^\star$ is the
      largest totally convex and lower semicontinuous function
      dominated by $\phi$.
    \end{enumerate}
\end{theorem}
\begin{proof}
  Claim 1 is a direct consequence of Lemma \ref{le:lower-B}
  (for properness) and Proposition \ref{prop:total-convexity-C} (for
  convexity and lower semicontinuity).

  In order to check \eqref{eq:9} of Claim 2, we consider $Y\in \cH$ and
  $\nu=\iota(Y)$; \eqref{eq:66} yields
  \begin{align*}
    \phi^\star(\nu)
    &=
      \sup_{\mu\in \cH_2(\rmH)}
      \msc \nu\mu-\phi(\mu)
      =
      \sup_{X\in \cH_2(\rmH)}
      \msc \nu{\iota( X)}-\phi(\iota(X))
     =  
    \sup_{\sfg\in \rmG(\OOmega)}\sup_{X\in \cH_2(\rmH)}
    \la Y,X\circ g\ra-\phi(\iota(X))
    \\&=
    \sup_{\sfg\in \rmG(\OOmega)}\sup_{X\in \cH_2(\rmH)}
    \la Y,X\circ \sfg\ra-\phi(\iota(X\circ \sfg))
    =
    \sup_{X'\in \cH_2(\rmH)}
    \la Y,X'\ra-\phi(\iota(X'))
      =(\phi\circ\iota)^\Ast(Y).
  \end{align*}
  The uniqueness follows by the law invariance of $(\phi\circ \iota)^*$.
 %
 Claim 3 is then an obvious consequence of Claim 2 and the corresponding property
 for the Legendre-Fenchel transformation in $\cH$. 
\end{proof}
We can easily derive natural properties from the Hilbertian theory.
\begin{corollary}[Kantorovich-Fenchel inequality]
  \label{cor:KF-inequality}
  For every $\ggamma\in \Gamma(\mu,\nu)$ we have
  \begin{gather}
    \label{eq:23}
    \int_{\rmH^2} \la x,y\ra\,\d\ggamma(x,y)\le \msc\mu\nu\le \phi(\mu)+\phi^\star(\nu)\\ 
    \phi^\star(\nu) = \sup
    \Big\{\int_{\rmH^2} \la x,y \ra \, \d\ggamma(x,y) - \phi(\pi^1_\#\ggamma)
    :\ggamma\in \cP_2(\rmH\times \rmH),\ 
    \pi^2_\sharp\ggamma=\nu\Big\}.
  \end{gather}
  Moreover
    \begin{equation}
    \label{eq:23bis}
    \int_{\rmH^2} \la x,y\ra\,\d\ggamma(x,y)=
    \phi(\mu)+\phi^\star(\nu)
    \quad
    \Leftrightarrow\quad
    \begin{cases}
    \phi(\mu)+\phi^\star(\nu) =\msc\mu\nu,
    \\
    \ggamma\in \Gamma_o(\mu,\nu).
    \end{cases}
  \end{equation}
\end{corollary}
\begin{corollary}
  \label{cor:convex}
  Let $\phi: \cP_2(\rmH)\to\Rinf$
  be a proper function. 
  The following properties are equivalent:
  \begin{enumerate}
  \item $\phi$ is totally convex and lower
    semicontinuous;
    \item $\phi\circ\iota$ is convex and lower semicontinuous in $\cH$;
    \item
      $\phi=\phi^{\star\star}$.
          \item $\phi$ 
    is $\msc\cdot\cdot$-convex, i.e. there exists a set
      $G\subset \cP_2(\rmH)\times \R$ such that
      \begin{equation}
        \label{eq:13}
        \phi(\mu)=\sup_{(\nu,a)\in G}\msc \mu\nu-a.
      \end{equation}
  \end{enumerate}
\end{corollary}
\begin{proof}
  The implication 1$\Leftrightarrow$2~is a consequence of the
  definition of total convexity and Lemma \ref{le:invariance}.

  1$\Leftrightarrow$3~is a consequence of Claim 3 of Theorem \ref{thm:obvious}.
  3$\Rightarrow$4 is also immediate by the definition of
  $(\phi^\star)^\star$; the converse implication
  is a consequence of the fact that the map $\mu\mapsto \msc\mu\nu$ is
  totally convex and continuous.
\end{proof}
Thanks to the commutation identity
\eqref{eq:9}, 
it is possible to 
derive interesting calculus rules 
for $\phi^\star$ 
from the corresponding 
formulae for $\hat \phi^\Ast$. We 
present two simple examples  that will turn to be useful in the following.
Let us first introduce the
dilation maps 
and the Moreau-Yosida regularizations of $\phi.$
\begin{definition}
    \label{def:dil-Yos}
    For every $a>0$ we set
    $\frd_a:\cP_2(\rmH)\to
    \cP_2(\rmH)$
    \begin{equation}
        \label{eq:dilation}
        \frd_a[\mu]:=(a \ii)_\sharp \mu,\quad
        \frd_a[\mu](B)=\mu(a^{-1}B)
        \quad\text{for every Borel set }
        B\subset \rmH.
    \end{equation}
    For every $\tau>0$ 
    and every proper 
    l.s.c.~and totally convex function 
    $\phi:\cP_2(\rmH)
    \to\Rinf$
    the $\tau$-Moreau-Yosida regularization of 
    $\phi$ is defined by 
    \begin{equation}
        \label{eq:Yosida}
        \phi_\tau(\mu):=
        \min_{\nu\in \cP_2(\rmH)}
        \frac 1{2\tau}\sfw_2^2(\mu,\nu)+
        \phi(\nu).
    \end{equation}
\end{definition}
Dilations are clearly related to 
scalar multiplication of Lagrangian maps via the formula
\begin{equation}
    \label{eq:identity-dil}
    \mu=\iota(X)\quad\Rightarrow\quad
    \frd_a[\mu]=\iota(a\,X).
\end{equation}
The Yosida regularization \eqref{eq:Yosida}
plays a crucial role in evolution problems via the JKO-Minimizing Movement scheme
\cite{JKO,AGS08}. 
Existence of a minimizer of 
\eqref{eq:Yosida} for an arbitrary geodesically convex functional follows by the results of \cite{Naldi-Savare21}; in the present case, we can invoke the following important commutation property 
\begin{equation}
    \label{eq:identity-Y}
    \widehat{\phi_\tau}=
    \hat\phi_\tau
\end{equation}
where $\hat\phi_\tau$ is
the Yosida regularization of $\hat\phi$
in $\cH$ defined by
\begin{equation}
    \label{eq:Y-in-cH}
    \hat\phi_\tau(X)=
    \min_{Y\in \cH}
    \frac 1{2\tau}\|X-Y\|^2_{\cH}+
    \hat \phi(Y),
\end{equation}
which also shows that $\phi_\tau$ is 
totally convex as well. 
In order to check \eqref{eq:identity-Y}
we first oberve that \eqref{eq:8} yields
\begin{displaymath}
    \mu=\iota(X)\quad\Rightarrow\quad
    \hat\phi_\tau(X)\ge 
    \phi_\tau(\mu);
\end{displaymath}
the converse inequality follows by the 
law invariance of $\hat \phi$ and 
\eqref{eq:66}.
\begin{corollary}
\label{cor:calculus-tools}
    Let 
    $\phi:\cP_2(\rmH)
    \to\Rinf$ be a proper 
    l.s.c.~and totally convex function.
    For every $a,b>0$ 
    we have
    \begin{equation}
        \label{eq:conj1}
        (a\phi)^\star=
        a\phi^\star\circ\frd_{a^{-1}},
    \end{equation}
    \begin{equation}
        \label{eq:conj2}
        \Big(a\phi+\frac b2 \sfm_2^2
        \Big)^\star=
        a(\phi^\star)_{b/a}\circ \frd_{a^{-1}}
    \end{equation}
\end{corollary}
\begin{proof}
    \eqref{eq:conj1} follows from the
    corresponding identity for 
    $\hat\phi$:
    \begin{displaymath}
        (a\hat\phi)^\Ast(X)=
        a\hat\phi^\Ast(a^{-1}X).
    \end{displaymath}
    \eqref{eq:conj2} follows by the 
    fact that
    \begin{displaymath}
        \Big(a\phi
        + \frac b2 \sfm_2^2
        \Big)^{\wedge}=
       a\hat\phi+ \frac b2 \|\cdot\|_\cH^2
    \end{displaymath}
    and by the corresponding formula
    for the Legendre transform
    of a quadratic perturbation
    of $\hat\phi$ in $\cH$ (which is a
    particular case of infimal convolution):
    \begin{equation}
    \label{eq:useful-for-dopo}
        \Big(a\hat\phi+
        \frac b2 \tau\|\cdot\|_\cH^2\Big)^\Ast=
        a(\hat\phi^\Ast)_{b/a}(\cdot/a).
    \end{equation}
\end{proof}
\subsection{\texorpdfstring{$\sfc$}{}-concave functions and their \texorpdfstring{$\sfc$}{}-super
  differentials}
  \label{subsec:c-concave}
As in the case of functions defined in a Hilbert space,
we can connect the notion of totally convex functionals and
Kantorovich-Legendre-Fenchel transform
in $\rW2$ 
with the corresponding notion of $\sfc$-concavity and $\sfc$-transform
for the cost $\sfc:=\frac 12\sfw_2^2$.

Let us first recall the main definition for a continuous
and symmetric cost function $\sfc:\cP_2(\rmH)\times \cP_2(\rmH)\to \R$:
in our case we will mainly use the choice $\sfc:=\frac 12\sfw_2^2$.
\begin{definition}[Concave $\sfc$-transform, $\sfc$-concavity and
  $\sfc$-superdifferential]
  \label{def:c-concave}
  Let $f:\cP_2(\rmH)\to \R\cup\{-\infty\}$ be a proper function.
  Its concave $\sfc$-transform
  $f^\sfc:\cP_2(\rmH)\to \R\cup\{-\infty\}$ is defined by
  \begin{equation}
    \label{eq:34-1}
    f^\sfc(\nu):=\inf_{\mu\in \cP_2(\rmH)}\sfc(\mu,\nu)-f(\mu).
  \end{equation}
  $f$ is $\sfc$-concave if
\begin{equation}
  \label{eq:33}
  \exists\,A\subset \cP_2(\rmH)\times \R:\quad
  f(\mu)=\inf_{(\nu,a)\in A}\sfc(\mu,\nu)-a.
\end{equation}
If $\mu,\nu\in \cP_2(\rmH)$ with $f(\mu)\in \R$, we say
that $\nu$ belongs to the $\sfc$-superdifferential of $f$ at $\mu$,
denoted by
$\partial_{\sfc}^+f(\mu)$, if
\begin{equation}
  \label{eq:35}
  f(\mu')-f(\mu)\le \sfc(\mu',\nu)-\sfc(\mu,\nu)\quad\text{for every
  }\mu'
  \in \cP_2(\rmH).
\end{equation}
\end{definition}
We collect a series of well known properties of the above notion.
\begin{theorem}
  \label{thm:c-conc}
  Let $f:\cP_2(\rmH)\to \R\cup\{-\infty\}$ be a proper function.
  \begin{enumerate}
  \item $f$ is $\sfc$-concave iff $f=g^\sfc$ for some function
    $g:\cP_2(\rmH)\to\R\cup\{-\infty\}$,
  \item $f$ is $\sfc$-concave if and only if
    $f=f^{\sfc\sfc}=(f^\sfc)^\sfc$.
  \item for every $\mu,\nu\in \cP_2(\rmH)$, $f(\mu)+f^{\sfc}(\nu)\le \sfc(\mu,\nu)$
  \item $\nu\in \partial_\sfc^{+}f(\mu)$ if and only if
    $f(\mu)+f^{\sfc}(\nu)= \sfc(\mu,\nu)$.
  \end{enumerate}
\end{theorem}
We could introduce the analogous concepts of $\sfc$-convexity, $\sfc$-convex transform and
$\sfc$-subdifferential: since in this context we will only use the
pseudo-scalar cost $\msc\cdot\cdot$,
$\msc\cdot\cdot$-convexity is equivalent to total convexity by
Corollary \eqref{cor:convex}, so that 
we will keep the notation $\phi^\star$
of Definition \ref{def:measure-conj} for the $\msc\cdot\cdot$-conjugate.
We will just introduce the notion
of $\msc\cdot\cdot$-subdifferential: for every $\mu\in \cP_2(\rmH)$ with $\phi(\mu)\in \R$, we also set 
\begin{equation}
  \label{eq:34}
  \partial^-\phi(\mu):=\Big\{
  \nu\in \cP_2(\rmH):
  \phi(\mu')-\phi(\mu)\ge
  \msc{\mu'}\nu-\msc\mu\nu
  \quad\text{for every }\mu'\in \cP_2(\rmH)\Big\}.
\end{equation}
As for property 4 of Theorem \ref{thm:c-conc} above, we easily get
\begin{equation}
  \label{eq:36}
  \nu\in \partial^-\phi(\mu)\quad
  \Leftrightarrow
  \quad
  \msc\mu\nu=\phi(\mu)+\phi^\star(\nu).
\end{equation}
As in the case of the $L^2$-Wasserstein metric in $\rmH$, we have the
following simple but crucial properties.
\begin{corollary}
  \label{cor:semiconc}
 Let us consider the cost function $\sfc:=\frac 12\sfw_2^2$, let
  $\phi,\psi:\cP_2(\rmH)\to\Rinf$
  and $\cU$, $\cV$ be defined as
  \begin{equation}
    \label{eq:37}
    \cU:=\frac 12\sfm_2^2-\phi,\quad
    \cV:=\frac 12\sfm_2^2-\psi.
  \end{equation}
  The following hold true:
  \begin{enumerate}
  \item
    \begin{equation}
    \label{eq:16}
    \psi=\phi^\star\quad\Leftrightarrow\quad 
    \cV=\cU^\sfc
  \end{equation}
\item
  $\phi$ is totally convex and lower semicontinuous
  if and only if
  $\cU$ is $\sfc$-concave (and thus upper semicontinuous).
\item For every $\mu,\nu\in \cP_2(\rmH)$ with $\phi(\mu)\in \R$
  (and thus $\cU(\mu)\in \R$ as well)
  \begin{equation}
    \label{eq:38}
    \nu\in \partial^-\phi(\mu)
    \quad\Leftrightarrow\quad
    \nu\in \partial_\sfc^+\cU(\mu).
  \end{equation}
  \end{enumerate}
\end{corollary}
\begin{proof}
  We repeatedly use the identity in \eqref{eq:W-decomp}.
  
  Claim 1.
  For
  every $\nu\in \cP_2(\rmH)$
  \eqref{eq:37} yields
  \begin{align*}
    \cU^\sfc(\nu)
    &=
      \inf_{\mu\in \cP_2(\rmH)} \frac 12\sfw_2^2(\mu,\nu)-\cU(\mu)
      =
     \inf_{\mu\in \cP_2(\rmH)} \frac 12\sfw_2^2(\mu,\nu)-\Big(\sfm_2^2(\mu)- \phi(\mu)\Big)
     \\& =
    \inf_{\mu\in \cP_2(\rmH)}
    \frac 12\sfm_2^2(\nu)-\msc\mu\nu+\phi(\mu)
    =
    \frac 12\sfm_2^2(\nu)+  \inf_{\mu\in
    \cP_2(\rmH)}\Big(-\msc\mu\nu+\phi(\mu)\Big)
    \\&=
    \frac 12\sfm_2^2(\nu)- \sup_{\mu\in
    \cP_2(\rmH)}\Big(\msc\mu\nu-\phi(\mu)\Big)
    \\&=
    \frac 12\sfm_2^2(\nu)-\phi^\star(\nu).
  \end{align*}
  The second claim immediately follows
  by the first one, Claim 3 of Corollary \ref{cor:convex} and Claim 2
  of Theorem \ref{thm:c-conc}.

  Claim 3 then follows by \eqref{eq:36} and
  Claim 4 of Theorem \ref{thm:c-conc}.
\end{proof}
In the next section we will discuss a more refined representation of
$\partial^-\phi$, where we
replace $\cP_2(\rmH)\times \cP_2(\rmH)$ by $\cP_2(\rmH\times \rmH)$.

\section{Multivalued probability 
  fields}
\label{sec:MPVF}
\newcommand{\brmF}{\boldsymbol{\rmF}}
\newcommand{\brmP}{\boldsymbol{\rmP}}

We recall the notion of multivalued probability vector field (MPVF)
$\brmF$, introduced in a different context by
\cite{Piccoli19} and in full generality by \cite{CSS1}:
they are a natural extension  of the usual notion of vector field for studying
the evolution of probability measures.

The simplest way to define a MPVF $\brmF$ is just to consider it as a nonempty
subset of $\cP_2(\rmH\times \rmH)$.
We set $D(\brmF):=\Big\{\pi^1_\sharp\ggamma:\ggamma\in \brmF\Big\}$
and for every $\mu\in D(\brmF)$ we consider the sections $\brmF[\mu]:=\Big\{\ggamma:
\pi^1_\sharp\ggamma=\mu\Big\}$.
Sections define a multivaled map $\brmF[\cdot]:\rW2\to
2^{\cP_2(\rmH\times \rmH)}$ with the property
\begin{displaymath}
  \ggamma\in \brmF[\mu]\quad \Rightarrow\quad
  \pi^1_\sharp\ggamma=\mu.
\end{displaymath}
When $\brmF[\cdot]$ is single valued we say that $\brmF$ is a
\emph{probability vector field (PVF).} A particular case, that play a
crucial role, is given by
\emph{deterministic} PVFs:
they are characterized by the property
\begin{equation}
  \label{eq:70}
  \brmF\subset \cPdet2(\rmH^2)\quad
  \text{so that }
  \brmF[\mu]=(\ii\times \ff)_\sharp \mu
  \quad\text{for a (unique) $\rmH$-valued map }\ff\in L^2(\rmH,\mu;\rmH),
\end{equation}
where we adopted the obvious notation
$\cPdet2(\rmH^2):=\cP_2(\rmH^2)\cap \cPdet{}(\rmH^2)$.
If $\brmF$ is a deterministic PVF, for every $\mu\in D(\brmF)$ we can
thus define
a nonlocal vectorfield 
\begin{equation}
  \label{eq:72}
  \ff[\mu]=\ff(\cdot,\mu)\in L^2(\rmH,\mu;\rmH)\quad
  \text{such that 
  \eqref{eq:70} holds;}
\end{equation}
notice that
\begin{equation}
  \label{eq:73}
  \pi^2_\sharp \brmF[\mu]=
  \ff[\mu]_\sharp\mu.
\end{equation}
Every MPVF $\brmF$
admits a Lagrangian representation (or lifting)
$\LF\subset \cH\times \cH$ defined by
\begin{equation}
  \label{eq:19}
  (X,Y)\in \LF\quad\Leftrightarrow\quad
  \iota^2(X,Y)=(X,Y)_\sharp\P\in \brmF.
\end{equation}
It is immediate to check that $\LF$ is law invariant
\begin{equation}
  \label{eq:74}
  (X,Y)\in \LF, \ \iota^2(X',Y')=\iota^2(X,Y)\quad\Rightarrow\quad
  (X',Y')\in \LF  
\end{equation}
and thus invariant with respect to the action of measure-preserving
isomorphisms:
\begin{equation}
  \label{eq:75}
  (X,Y)\in \LF\quad\Rightarrow\quad
  (\sfg^*X,\sfg^*Y)\in \LF\quad\text{for every }\sfg\in \rmG(\OOmega).
\end{equation}
Conversely, if a subset $\bB\subset \cH\times \cH$ is
invariant by m.p.i.~and \emph{closed}, then it is also law invariant
and it is the Lagrangian representation of a unique MPVF $\brmF$
\cite{CSS2}.
\subsection{Totally monotone and cyclically monotone MPVF}
\label{susbsec:totally-monotone}
Inspired by the the definition of totally monotone MPVF (introduced in
\cite{CSS22}) we 
introduce here the 
corresponding notion of totally cyclically monotone MPVF. We will adopt the
notation
\begin{equation}
  \label{eq:31}
  \ppi^n:(\rmH^2)^N\to\rmH^2,\quad
  \ppi^n((x_1,y_1),(x_2,y_2),\cdots,(x_N,y_N))=(x_n,y_n),\quad
  n=1,\cdots,N.
\end{equation}
\begin{definition}
[Totally monotone and cyclically monotone MPVF]
  \label{def:monotone}
  A multivalued probability vector field (MPVF) $\brmF\subset \cP_2(\rmH^2)$
  is totally monotone if
  for every $\ttheta\in \cP_2(\rmH^2\times \rmH^2)$
  with $\ppi^{1}_\sharp\ttheta\in \brmF,\ \ppi^{2}_\sharp\ttheta\in
  \brmF$ we have
  \begin{equation}
    \label{eq:18bis}
    \int \la y_2-y_1,x_2-x_1\ra\,\d\ttheta(x_1,y_1;x_2,y_2)\ge 0.
  \end{equation}
  $\brmF$ is a maximal totally monotone MPVF if
  it is totally monotone and
  the inclusion $\brmF\subset \boldsymbol{\rmG}$, $\boldsymbol{\rmG}$
  totally monotone, yields $\brmF=\boldsymbol{\rmG}$.

  $\brmF$
  is totally cyclically monotone if
  for every $N\in \N$, $\ttheta\in \cP_2((\rmH^2)^N)$,
  with $\ppi^n_\sharp\ttheta\in \brmF$, $n=1,\cdots,N$,
  and $\sigma\in \rmS_N=\operatorname{Sym}(\{1,\cdots,N\})$ permutation
  \begin{equation}
    \label{eq:18}
    \sum_{n=1}^N\int \la y_n,x_n-x_{\sigma(n)}\ra\,\d\ttheta
    (x_1,y_1;x_2,y_2;\cdots;x_N,y_N)\ge0.
  \end{equation}
\end{definition}
In the case of a deterministic PVF
induced by a nonlocal vector field
$\ff$
as in \eqref{eq:70},\eqref{eq:72}, the total monotonicity condition
\eqref{eq:18} reads as
\begin{equation}
  \label{eq:77}
  \int \la
  \ff(x_2,\mu_2)-\ff(x_1,\mu_1),x_2-x_1\ra\,\d\mmu(x_1,x_2)\ge 0
  \quad\text{for every }
  \mu_i\in D(\brmF),\ \mmu\in \Gamma(\mu_1,\mu_2);
\end{equation}
Similarly, \eqref{eq:18bis} reads as
 \begin{equation}
    \label{eq:77bis}
    \sum_{n=1}^N\int \la
    \ff(x_n,\mu_n),x_n-x_{\sigma(n)}\ra\,\d\mmu\ge0
    \quad\text{for every }
      \mu_i\in D(\brmF),\ \mmu\in \Gamma(\mu_1,\mu_2,\cdots,\mu_N).
  \end{equation}
Let $\LF$ be the Lagrangian representation of $\brmF$; it is easy to check that
\begin{equation}
  \label{eq:20}
  \begin{aligned}
    \brmF\text{ is totally monotone}\quad&\Leftrightarrow\quad \LF\text{ is
      monotone in }\cH\times \cH,
  \end{aligned}
\end{equation}
and, according to the main result of \cite{CSS2}, we also have
\begin{equation}
  \label{eq:21-1}
    \brmF\text{ is maximal totally monotone}\quad\Leftrightarrow\quad \LF\text{ is
     maximal  monotone in }\cH\times \cH.
 \end{equation}
A maximal totally monotone MPVF has many important properties, see
\cite{CSS22}.
We quote here the most relevant ones for our discussion; the first one involves the 
(Borel) barycentric map 
$\bb:\rmH\times \cP_2(\rmH\times \rmH)\to\rmH$
defined by using the universal disintegration kernel \eqref{eq:61}:
\begin{equation}
    \label{eq:barycentric}
    \bb(x,\ggamma):=
    \int y\,\d(\cK(x,\ggamma)(y))=
    \int y\,\d\kappa_x(y),\quad
\end{equation}
where $\kappa_x=\cK(x,\ggamma) $
      is the disintegration of $\ggamma$ w.r.t.~its first marginal.
\begin{proposition}
  \label{prop:tot-max-mon}
  Let $\brmF$ be a maximal totally monotone MPVF.
  \begin{enumerate}
  \item
    If $\ggamma\in \brmF$ and 
    $\mu=\pi^1_\sharp\ggamma\in D(\brmF)$,
    then the barycentric projection 
    $\bb(\cdot,\ggamma)$ of $\ggamma$
    satisfies
    \begin{equation}
      \label{eq:64}
      (\ii\times \bb(\cdot,\ggamma))_\sharp\mu\in \brmF[\mu].
    \end{equation}
    \item There exists a unique minimal section $\brmF^\circ\subset
      \brmF\cap
      \cPdet2(\rmH\times\rmH)$ 
      and a Borel nonlocal vector field
      $\ff^\circ=
      \bb(\cdot,\brmF^\circ)
      :\rmH\times D(\brmF)
      \to \rmH$ such that
      for every $\mu\in D(\brmF)$
      we have
      $\brmF^\circ[\mu]=(\ii\times \ff^\circ(\cdot,\mu))_\sharp\mu$ and 
      \begin{equation}
        \label{eq:79}
        \int_\rmH \big|\ff^\circ(x,\mu)\big|^2\,\d\mu(x)
        \le
        \int_\rmH \big|\bb(x,\ggamma)\big|^2\,\d\mu(x)
        \le
        \int_\rmH \big|y\big|^2\,\d\ggamma(x,y)
        \quad\text{for every }\ggamma\in \brmF[\mu].
      \end{equation}
    \item If $\LF$ is the Lagrangian representation of $\brmF$, then
      $D(\LF)=(\iota_\sharp)^{-1}(D(\brmF))$ and the minimal section
      $\LF{}^\circ$
      of $\LF$
      satisfies
      \begin{equation}
        \label{eq:80}
        Y=\LF{}^\circ(X)\quad\Leftrightarrow\quad
        Y(\oomega)=\ff^\circ(X(\oomega),\mu)=
        \ff^\circ[\mu](X(\oomega))\quad
        \text{for $\P$-a.e.~$\oomega$}, \ \mu=\iota(X).
      \end{equation}
  \end{enumerate}
\end{proposition}
Let us focus now on totally cyclically monotone MPVFs: first of all, we have a simple lifting result.
\begin{proposition}[Lifting of totally cyclically monotone MPVF]
    \label{prop:trivial}
    A MPVF $\brmF$ is totally cyclically monotone
    if and only if $\LF$ is
          cyclically monotone in $\cH\times \cH.$
\end{proposition}
\begin{proof}
    It is sufficient to observe that 
    for every $\ttheta\in \cP_2((\rmH^2)^N)$
    with $\ppi^n_\sharp\ttheta\in \brmF$, $n=1,\cdots,N$,
    we can find a Borel map
    $Z=((X_1,Y_1),\cdots,(X_N,Y_N)):\OOmega\to 
    (\rmH^2)^N$
    such that 
    $Z_\sharp\P=\ttheta$, so that 
    $(X_n,Y_n)\in \LF$ for every $n\in \{1,\cdots,N\}.$
    Conversely, if 
    $(X_n,Y_n)\in \LF$ then 
    $\ttheta=Z_\sharp\P$
    belongs to 
    $\cP_2((\rmH^2)^N)$ and 
    $\ppi^n_\sharp\ttheta\in \brmF$.

    For every permutation $\sigma\in \rmS_N$ 
    we can then use the identity
    \begin{displaymath}
         \sum_{n=1}^N\int \la y_n,x_n-x_{\sigma(n)}\ra\,\d\ttheta
    =
    \sum_{n=1}^N\int \la Y_n,X_n-X_{\sigma(n)}\ra\,\d\P=
    \sum_{n=1}^N \la Y_n,X_n-X_{\sigma(n)}\ra_\cH
    \end{displaymath}
    which shows the equivalence between condition
    \eqref{eq:18}
    and the corresponding condition expressing the cyclical monotonicity of $\LF$ in $\cH\times\cH$.
    \end{proof}
As in the Hilbertian framework, we will show that totally cyclically monotone MPVFs
are strictly linked with the notion of 
\emph{total subdifferential}.
Let us first recall the definition
(we refer to 
\cite{CSS22} for a detailed comparison with other notion of subdifferentiability, in particular the ones introduced in 
\cite{AGS08}, from a metric perspective).
\begin{definition}[Total subdifferential]
  \label{def:total-sub}
  The total subdifferential of
  $\phi:\cP_2(\rmH)\to\Rinf$
  is the set $\bpartialt\phi\subset \cP_2(\rmH\times \rmH)$ characterized by the following property:
  a plan $\ggamma\in \cP_2(\rmH^2)$ belongs to $\bpartialt\phi$ if and only if
  $\mu=\pi^1_\sharp\ggamma\in D(\phi)$ and for every $\nu\in D(\phi)$
  and $\ttheta\in \Gamma(\ggamma,\nu)$ we have
  \begin{equation}
    \label{eq:21}
    \phi(\nu)-\phi(\mu)\ge
    \int_{\rmH^2\times \rmH}\la y_1,x_2-x_1\ra\,\d\ttheta(x_1,y_1;x_2).
  \end{equation}
\end{definition}
Notice that 
if $\bpartialt\phi$ is not empty then
$\phi$ satisfies 
  \eqref{eq:lower-B}:
  if $\ggamma\in
  \bpartialt\phi[\mu]$, 
  $\mu':=\pi^2_\sharp\ggamma$,
  $\ggamma'\in \Gamma_o(\nu,\mu')$
  we can apply the glueing Lemma to select 
  $\ttheta\in\cP_2(\rmH^3)$ 
  with $\pi^{1\,2}_\sharp \ttheta=\ggamma$,
  $\pi^{2\,3}_\sharp\ttheta=\ggamma'$
  so that 
  \eqref{eq:21} yields
  \begin{displaymath}
    \phi(\nu)\ge \phi(\mu)-
    \int_{\cH^2}\la y,x\ra\,\d\ggamma(x,y)+
    \int_{\cH^2}
    \la y,z\ra\,\d\ggamma'(y,z)
    = -a+\msc\nu{\mu'}
  \end{displaymath}
  where $a:= \int_{\cH^2}\la y,x\ra\,\d\ggamma(x,y)-\phi(\mu)$.

Applying 
\cite[Proposition 5.2]{CSS22} 
we can show that there is a strong relation between the total subdifferential
of $\phi$ and the 
(convex) subdifferential of its Lagrangian lifting $\hat \phi=\phi\circ\iota$ in $\cH$.
\begin{theorem}[Lifting of total subdifferentials]
  \label{prop:collection}
    Let $\phi:\rW2\to\Rinf$ be a proper lower semicontinuous function
  with Lagrangian lift $\hat\phi:=\phi\circ\iota$.
  The (possible empty) convex subdifferential $\partial\hat\phi\subset \cH\times \cH$
  is law invariant and coincides with the Lagrangian representation of $\bpartialt\phi$, i.e.
  \begin{equation}
      \label{eq:subdifferential-commutation}
      \widehat{\bpartialt\phi}=
      \partial\hat\phi.
  \end{equation}
  In particular
  $\bpartialt\phi$ is totally cyclically monotone and
  \begin{equation}
      \label{eq:inclusions}
      \bpartialt\phi\subset 
  \bpartialt(\phi^{\star\star}),
  \qquad
  \bpartialt\phi[\mu]\neq\emptyset
  \quad \Leftrightarrow\quad
  \phi(\mu)=\phi^{\star\star}(\mu),\ 
  \bpartialt\phi[\mu]=
  \bpartialt\phi^{\star\star}[\mu]\neq \emptyset.
  \end{equation}
   If moreover $\phi$ is also totally convex then $\bpartialt\phi$ 
   is maximal totally
  monotone and its minimal section $\bpartialt^\circ\phi$
  is associated with the minimal section $\partial^\circ\hat\phi$ through \eqref{eq:80}.
\end{theorem}
\begin{proof}
  Since $\hat\phi$ is l.s.c., 
  $\partial\hat\phi$ is a closed 
  subset of $\cH\times \cH$.
  By Claim 1 of \cite[Proposition 5.2]{CSS22}
  we deduce that $\partial\hat\phi$ is law invariant
  and 
  we can then apply Claim 3 of \cite[Proposition 5.2]{CSS22}
  (which uses only the law invariance of $\partial\hat\phi$)
  to deduce 
    that
  $\partial\hat\phi$ is the
  Lagrangian representation of
  $\bpartialt\phi$ according to \eqref{eq:subdifferential-commutation}.
  \eqref{eq:inclusions} then follow
  by the corresponding well known properties
  of $\partial\hat\phi$.

  When $\phi$ is also totally convex, 
  the same Proposition 5.2 shows that
  $\bpartialt\phi$ is maximal totally monotone and \eqref{eq:20} shows
  that it is totally cyclically monotone (since $\partial\hat\phi$ is
  cyclically monotone in $\cH\times \cH$).
  We conclude by applying Claim 3 of Proposition \ref{prop:tot-max-mon}.
\end{proof}
Thanks to the previous result, we can now extend the celebrated Rockafellar Theorem to totally
cyclically monotone MPVF.
\begin{theorem}[Totally cyclically maximal monotone MPVFs are total subdifferentials]
  \label{thm:Rockafellar}
  If $\brmF$ is a totally cyclically monotone MPVF in
  $\cP_2(\rmH^2)$ then there exists
  a proper, totally convex and lower semicontinuous function
  $\phi:\cP_2(\rmH)\to\Rinf$
  such that $\brmF\subset \bpartialt\phi$.
  In particular, $\brmF$ has a maximal 
  totally monotone
  extension which is cyclically monotone
  and every totally cyclically 
  maximal monotone MPVF is the total
  subdifferential
  of a totally convex function.
\end{theorem}
\begin{proof}
  Let us denote by $\LF$ the Lagrangian representation of $\brmF$ in
  $\cH^2$; we know that $\LF$ is invariant by measure preserving
  isomorphisms, i.e.
  \begin{equation}
    \label{eq:26}
    (X,Y)\in \LF\quad\Rightarrow\quad
    (X\circ\sfg,Y\circ \sfg)\in \LF\quad\text{for every }\sfg\in \rmG(\OOmega).
  \end{equation}
  We fix $(X_0,Y_0)\in \LF$ and use Rockafellar construction
  \cite{Rockafellar66} to define
  \begin{equation}
    \label{eq:27}
    \Phi(X):=\sup\Big\{\la X-X_N,Y_N\ra_\cH+
    \sum_{n=1}^N \la X_n-X_{n-1},Y_{n-1}\ra_\cH:
    N\in \N,\ (X_n,Y_n)\in \LF\Big\}.
  \end{equation}
  We know that $\Phi$ is convex, proper, lower semicontinuous, and $\LF\subset \partial\Phi$.
  Let us prove that $\Phi$ is invariant by measure-preserving
  isomorphisms:
  we use the fact that for every $N$-tuple $(X_n,Y_n)\in \LF$
  we also have $(X_n\circ \sfg,Y_n\circ\sfg)\in \LF$ so that
  \begin{align*}
    \Phi(X\circ\sfg)
    &\ge
      \la X\circ\sfg -X_N\circ\sfg ,Y_N\circ\sfg\ra_\cH+
    \sum_{n=1}^N \la
      X_n\circ\sfg-X_{n-1}\circ\sfg,Y_{n-1}\circ\sfg\ra_\cH
    \\&=
    \la X -X_N ,Y_N\ra_\cH+
    \sum_{n=1}^N \la
      X_n-X_{n-1},Y_{n-1}\ra_\cH
  \end{align*}
  so that taking the supremum with respect to $(X_n,Y_n)\in \LF$ and
  $N\in \N$ we get
  \begin{displaymath}
    \Phi(X\circ\sfg)\ge \Phi(X).
  \end{displaymath}
  Applying the same inequality to $X\circ\sfg$ with $\sfg$ replaced by
  $\sfg^{-1}$ we also obtain
  $\Phi(X\circ\sfg\circ\sfg^{-1}) =\Phi(X) \ge \Phi(X\circ\sfg)$
  so that $\Phi(X\circ\sfg)=\Phi(X)$ for every $X$ and every $\sfg\in
  \rmG(\OOmega)$.
  Since $\Phi$ is lower semicontinuous and invariant by the action
  of $\rmG(\OOmega)$, Lemma \ref{le:invariance} shows that 
  $\Phi=\phi\circ\iota$ for a lower semicontinuous
  function
  $\phi:\cP_2(\rmH)\to\Rinf$, which is clearly totally convex and
  satisfies $\brmF\subset \bpartialt\phi$ by Proposition \ref{prop:collection}.
\end{proof}
We briefly discuss the case when $\phi$ 
is differentiable, in a suitable Wasserstein sense.
First of all 
by using a (Borel) barycentric map $\bb$ as in 
\eqref{eq:barycentric} we can define the 
nonlocal deterministic field 
$\nabla_W\phi:\rmH\times D(\bpartialt\phi)\to\rmH$
as in Proposition \ref{prop:tot-max-mon}:
\begin{equation}
    \label{eq:W-gradient}
    \nabla_W\phi(x,\mu)=
    \nabla_W\phi[\mu](x):=
    \bb(x,\bpartialt^\circ\phi[\mu]).
\end{equation}
Recall that if a convex function $\psi:\cH\to \Rinf$ is Gateaux differentiable at 
$X\in D(\psi)$ then 
$\partial\psi(X)$ is a singleton
(thus coinciding with the minimal section
$\partial^\circ\psi$). The converse is also true if in addition $\psi$ is continuous at $X$
\cite[Chap.~I, Prop.~5.3]{Ekeland-Temam76}.
We can say that 
\begin{equation}
    \label{eq:defW}
    \text{$\phi$ is
$W$-differentiable at $\mu$ if 
$\bpartialt\phi[\mu]$ 
contains a unique element},
\end{equation} 
which in turn coincides with the minimal section,  expressed through $\nabla_W\phi$. 
\begin{proposition}
    \label{prop:diff}
    Let $\phi:\cP_2(\rmH)\to\Rinf$
    be a proper, l.s.c., and totally convex function.
    \begin{enumerate}
        \item $\phi$ is $W$-differentiable 
        at $\mu$ if and only if
        $\partial\hat \phi$ 
        is a singleton at every 
        $X\in \iota^{-1}(\mu).$
        \item 
        If $\phi$ is continuous 
        at $\mu$ then 
        it is 
        $W$-differentiable if and only if 
        $\hat\phi$ is Gateaux-differentiable
        at every point of 
        $\iota^{-1}(\mu)$.
        \item For every $\tau>0$
        $\phi_\tau$ is 
        $W$-differentiable everywhere,
        the map $(x,\mu)\mapsto \nabla_W\phi_\tau(x,\mu)$
        is everywhere defined and continuous
        in 
        \begin{equation}
        \label{eq:natural-domain-Cmaps}
            \cS(\rmH):=\Big\{(x,\mu)\in \rmH\times\cP_2(\rmH):x\in \supp(\mu)\Big\}
        \end{equation}
        for every $\mu\in \cP_2(\R^d)$ 
        the map $x\mapsto \nabla_W\phi_\tau(x,\mu)$
        is $\tau^{-1}$-Lipschitz in $\supp(\mu)$,
        and the map
        $\mu\mapsto 
        \nabla_W\phi_\tau(\cdot,\mu)_\sharp\mu$ is $\tau^{-1}$-Lipschitz
        in $\cP_2(\rmH).$
    \end{enumerate}
\end{proposition}
\begin{proof}
    To prove the first claim
    it is sufficient to recall that 
    \begin{displaymath}
        (X,Y)\in \partial\hat\phi\quad
        \Leftrightarrow\quad
        (X,Y)_
        \sharp\P\in \bpartialt\phi
    \end{displaymath}
    $\bpartial\phi[\mu]$ 
    contains just one element $\ggamma$
    if and only if it is 
    reduced to the 
    minimal section, i.e.~$\ggamma=
    (\ii\times \nabla_W\phi(\cdot,\mu))_\sharp\mu$
    so that 
    $\iota^2(X,Y)=\ggamma
    $ implies $Y=\nabla_W(X,\mu)$
    and therefore $\partial\phi(X)$
    is reduced to a singleton for every
    $X\in \iota^{-1}(\mu).$

    The second claim then follows 
    since continuity of $\phi$ at $\mu$ 
    implies continuity of $\hat\phi$ 
    at every $X\in \iota^{-1}(\mu).$

    Finally, the third claim
    follows by
    \eqref{eq:identity-Y} and 
    the well known properties of the Moreau-Yosida regularization in the Hilbert space
    $\cH$: 
    $\hat\phi_\tau$ is of class $\rmC^{1,1}$ and its differential 
    $\rmD\hat\phi_\tau$ (which 
    provides the unique element
    of $\partial\hat\phi_\tau$) is 
    $\tau^{-1}$-Lipschitz.
    We can then apply 
    the results of 
    \cite[Section 4]{CSS2} and 
    \cite[Section 5]{CSS22} to conclude.    
\end{proof}
\subsection{Marginal projections of MPVF}
\label{subsec:marginal-proj}
We notice that a MPVF $\brmF
\subset \cP_2(\rmH\times \rmH)$ induces a subset $\frF\subset
\cP_2(\rmH)\times \cP_2(\rmH)$ by applying 
marginal projections
\begin{equation}
  \label{eq:24}
  \frF=(\pi^1_\sharp,\pi^2_\sharp)\brmF,\quad
  \text{or, equivalently,}\quad 
  (\mu,\nu)\in \frF\quad\Leftrightarrow\quad
  \brmF\cap\Gamma(\mu,\nu)\neq \emptyset.  
\end{equation}
It is therefore natural to investigate the relations between
the total subdifferential $\bpartialt\phi$ (a MPVF) and
the $\msc\cdot\cdot$-subdifferential of $\phi$ (a subset of
$\cP_2(\rmH)\times \cP_2(\rmH)$).

We start with a nice characterization of the total subdifferential of $\phi$
in terms of its Kantorovich-Legendre-Fenchel transform.
\begin{theorem}[Total subdifferential and Kantorovich-Legendre-Fenchel transform]
  \label{thm:sub-KLF}
  A plan $\ggamma\in \cP_2(\rmH^2)$
  belongs to $\bpartialt\phi[\mu]$
  if and only if
  $\mu=\pi^1_\sharp \ggamma$,
  $\nu=\pi^2_\sharp\ggamma\in D(\phi^\star)$ and
  \begin{equation}
    \label{eq:22}
    \int_{\rmH^2}\la x,y\ra\,\d\ggamma(x,y)=\phi(\mu)+\phi^\star(\nu).
  \end{equation}
  In particular $\ggamma\in \Gamma_o(\mu,\nu)$
  and $\int_{\rmH^2}\la x,y\ra\,\d\ggamma(x,y)=\msc\mu\nu$.
\end{theorem}
\begin{proof}
  We have seen that $\ggamma\in\bpartialt\phi[\mu]$ if and only if
  for every $(X,Y)\in \cH^2$ with $\iota^2(X,Y)=\ggamma$ we have
  $(X,Y)\in \partial\hat\phi$. In turn, this is equivalent to 
  \begin{displaymath}
    \la X,Y\ra_\cH=\hat\phi(X)+(\hat\phi)^\Ast(Y)=
    \phi(\mu)+\phi^\star(\nu).
  \end{displaymath}
  Since $\la X,Y\ra_\cH=\int \la x,y\ra\,\d\ggamma$ we obtain \eqref{eq:22}.
\end{proof}
Combining Theorem \ref{thm:sub-KLF} with \eqref{eq:36} we can immediately link the total subdifferential 
$\bpartialt\phi$ with the $\msc\cdot\cdot$ subdifferential $\partial^-\phi$.
\begin{corollary}[Total and $\msc\cdot\cdot$ subdifferential]
  \label{cor:total-vs-sub}
  Let $\phi:\cP_2(\rmH)\to \Rinf$ be a function satisfying the
  lower bound \eqref{eq:lower-B}, $\mu,\nu\in \cP_2(\rmH)$
  with $\phi(\mu)\in\R$.
  Then
  \begin{equation}
    \label{eq:39}
    \ggamma\in \bpartialt\phi[\mu]
    \quad
    \Leftrightarrow\quad
    \nu\in \partial^-\phi(\mu),\quad
    \ggamma\in \Gamma_o(\mu,\nu).
  \end{equation}
  In particular 
  $\bpartialt\phi$ 
  can be obtained as the image of the
  of the graph of $\partial^-\phi$ in
  $\cP_2(\rmH)\times \cP_2(\rmH)$ through the 
  multivalued map $\Gamma_o:\cP_2(\rmH)\times \cP_2(\rmH)
  \rightrightarrows
  \cP_2(\rmH\times \rmH)$.\\
  Conversely,
  the graph of $\partial^-\phi$ can be obtained as the image
  of $\bpartialt\phi\subset \rW2\times \rW2$ through the map
  $\pi^1_\sharp\times \pi^2_\sharp$.
\end{corollary}
It is worth highlighting 
the following  consequence of the previous result, which does not immediately appear from the definition of total subdifferential.
\begin{corollary}[Total subdifferentials are optimal couplings]\label{cor: tot subdiff are opt}
  \label{cor:trivial}
  For every function $\phi:\cP_2(\rmH)\to \Rinf$
  its total subdifferential 
  $\bpartialt\phi$ is contained in the set of optimal couplings
  $\cP_{2,o}(\rmH\times\rmH)$.
\end{corollary}
The above corollary suggests a lifting procedure of a subset $\frF\subset
\cP_2(\rmH)\times \cP_2(\rmH)$ which 
defines
a MPVF $\brmF$ by the formula
\begin{equation}
  \label{eq:25}
  \brmF:=\Big\{\ggamma\in \cP_2(\rmH^2):
  \ggamma\in \Gamma_o(\mu,\nu)\quad\text{for some }(\mu,\nu)\in \frF\Big\}
  =(\pi^1_\sharp,\pi^2_\sharp)^{-1}(\frF)
  \cap\cP_{2,o}(\rmH\times \rmH).
\end{equation}
\begin{theorem}[Total cyclical monotonicity and $\sfw_2^2$-cyclical monotonicity]
  \label{thm:brmF-vs-frF}
  If the MPVF $\brmF$ is totally cyclically monotone then
  the set 
  $\frF=(\pi^1_\sharp,\pi^2_\sharp)\brmF$
  defined as in \eqref{eq:24} is $\sfw_2^2$-cyclically monotone in
  $\cP_2(\rmH)\times
  \cP_2(\rmH)$
  in the sense that for every
  $N\in \N$, every choice of pairs
  $(\mu_1,\nu_1),\cdots,(\mu_N,\nu_N)\in \frF$ and every permutation
  $\sigma\in \rmS_N$ 
  \begin{equation}
    \label{eq:15}
    \sum_{n=1}^N\sfw^2(\mu_n,\nu_n)
    \le
    \sum_{n=1}^N\sfw^2(\mu_n,\nu_{\sigma(n)}).
  \end{equation}
  Conversly, if 
  $\frF$ is  $\sfw_2^2$-cyclically monotone in
  $\cP_2(\rmH)\times
  \cP_2(\rmH)$ according to \eqref{eq:15}
  then $\brmF$ defined by \eqref{eq:25} is totally
  cyclically monotone.
\end{theorem}
\begin{proof}
  Let us first observe that for arbitrary choice of $(\mu_n,\nu_n)\in 
\cP_2(\rmH)\times \cP_2(\rmH)$ the identity \eqref{eq:W-decomp} yields
\begin{align*}
  \frac12\sum_{n=1}^N\sfw^2(\mu_n,\nu_{\sigma(n)})-
  \frac 12 \sum_{n=1}^N\sfw^2(\mu_n,\nu_n)=
  \sum_{n=1}^N\Big(\msc{\mu_n}{\nu_n}-\msc{\mu_n}{\nu_{\sigma(n)}}\Big).
\end{align*}
  Suppose that $\brmF$ is totally cyclically monotone;
  by Theorem \ref{thm:Rockafellar} we can find a proper totally convex
  l.s.c.~function $\phi:\cP_2(\rmH)\to\Rinf$
  such that $\brmF\subset \bpartialt\phi$ so that in particular
  all the elements of $\brmF$ are optimal couplings of
  $\cP_{2,o}(\rmH\times \rmH)$.
  For every pair $(\mu_n,\nu_n)\in \frF$, $n=1,\cdots,N$
  we thus get
  \begin{align*}
    \sum_{n=1}^N \msc{\mu_n}{\nu_n}-
    \msc{\mu_n}{\nu_{\sigma(n)}}
    &\ge
      \sum_{n=1}^N \phi(\mu_n)+\phi^\star(\nu_n)-
      \Big(\phi(\mu_n)+\phi^\star(\nu_{\sigma(n)})\Big)
    \\&=
    \sum_{n=1}^N \phi^\star(\nu_n)-
    \sum_{n=1}^N \phi^\star(\nu_{\sigma(n)})= 0.
  \end{align*}
  Conversely, if $\frF$ is $\sfw_2^2$-cyclically monotone
  and $\brmF$ is defined as in \eqref{eq:25}
  we have for every $\ttheta\in \cP_2((\rmH^2)^N)$ with
  $\pi^n_\sharp\ttheta=\ggamma_n\in \brmF$
  \begin{align*}
    \sum_{n=1}^N\int \la y_n,x_n-x_{\sigma(n)}\ra\,\d\ttheta
    &=
      \sum_{n=1}^N\int \la y_n,x_n\ra\,\d\ttheta
      -
      \sum_{n=1}^N\int \la y_n,x_{\sigma(n)}\ra\,\d\ttheta
    \\&=
    \sum_{n=1}^N\int \la y,x\ra\,\d\ggamma_n(x,y)
    -
    \sum_{n=1}^N\int \la y_n,x_{\sigma(n)}\ra\,\d\ttheta
    \\&=
    \sum_{n=1}^N\msc{\mu_n}{\nu_n}
    -
    \sum_{n=1}^N\int \la y_n,x_{\sigma(n)}\ra\,\d\ttheta
    \\&\ge
    \sum_{n=1}^N\msc{\mu_n}{\nu_n}
    -
    \sum_{n=1}^N\msc{\mu_n}{\nu_{\sigma(n)}}
    \ge 0
  \end{align*}
\end{proof}
\section{\texorpdfstring{$L^2$}{}-Random Optimal Transport}
\label{sec:ROT}
In this section we will apply the main results of Section
\ref{sec:totally-convex} for totally convex functions
and of Section \ref{sec:MPVF} for totally cyclically monotone MPVF
to optimal transport in $\RW2$.
\subsection{Random couplings and couplings of random measures}
\label{susbec:random-couplings}
Let us first observe that
to every random coupling law
$\bttP\in 
\cPP_2(\rmH\times \rmH)=
\cP_2(\cP_2(\rmH\times \rmH))$ we can associate a
coupling between (laws of) random measures
$\pPi$ by the formula
\begin{equation}
  \label{eq:32}
  \pPi=(\pi^1_\sharp,\pi^2_\sharp)_\sharp \bttP,\quad
  \pPi=\int \delta_{\pi^1_\sharp\ggamma}\otimes
  \delta_{\pi^2_\sharp\ggamma}\,\d\bttP(\ggamma).
\end{equation}
If $\ttM_i=\pi^i_\sharp \pPi$,
it is not difficult to check that
\begin{equation}
  \label{eq:17}
  \sfW_2^2(\ttM_1,\ttM_2)\le
  \int \sfw_2^2(\mu_1,\mu_2)\,\d\pPi(\mu_1,\mu_2)
  \le \int \Big(\int |x_1-x_2|^2\,\d\ggamma(x_1,x_2)\Big)\,\d\bttP(\ggamma).
\end{equation}
Conversely, 
using the fact that the continuous map
\begin{equation}
    \label{eq:surj}
    \pi^1_\sharp\times \pi^2_\sharp:
    \cP_{2,o}(\rmH\times \rmH)
    \to \cP_2(\rmH)\times \cP_2(\rmH)
    \quad\text{is surjective,}
\end{equation}
given a coupling $\pPi\in
\cP_2(\cP_2(\rmH)\times \cP_2(\rmH))$
we can find a $\bttP\in \cP_2(\cP_{2,o}(\rmH\times \rmH))$
(thus concentrated on optimal couplings)
such that
\begin{equation}
  \label{eq:40}
  \pPi=(\pi^1_\sharp,\pi^2_\sharp)_\sharp \bttP,\quad
  \int \sfw_2^2(\mu_1,\mu_2)\,\d\pPi(\mu_1,\mu_2)
  = \int \Big(\int |x_1-x_2|^2\,\d\ggamma(x_1,x_2)\Big)\,\d\bttP(\ggamma).
\end{equation}
We immediately deduce an equivalent characterization of $\sfW_2$ in
terms of random coupling laws. We write
$\pi^i_{\sharp\sharp}:=(\pi^i_\sharp)_\sharp$
and we call
\begin{equation}
  \label{eq:48}
  \RGamma(\ttM_1,\ttM_2):=
  \Big\{\bttP\in \cPP_2(\rmH\times \rmH),
     \pi^i_{\sharp\sharp}\bttP=\ttM_i\Big\}.
\end{equation}
\begin{proposition}[Random couplings formulation of OT between random measures]
  \label{prop:W_2-random}
  For every $\ttM_1,\ttM_2\in \RW2$ we have
  \begin{equation}
    \label{eq:42}
    \sfW_2^2(\ttM_1,\ttM_2)=
    \min
    \Big\{
     \int \Big(\int
     |x_1-x_2|^2\,\d\ggamma(x_1,x_2)\Big)\,\d\bttP(\ggamma):
     \bttP\in \RGamma(\ttM_1,\ttM_2)
    \Big\}.
  \end{equation}
  Moreover, $\bttP$ is optimal for \eqref{eq:42} if and only if
  $\bttP$ is concentrated on the optimal couplings of
  $\cP_{2,o}(\rmH\times \rmH)$ and 
  $\pPi=(\pi^1_\sharp,\pi^2_\sharp)_\sharp \bttP\in \Gamma_o(\ttM_1,\ttM_2)$.

  We denote by $\RGamma_o(\ttM_1,\ttM_2)$ the set of optimal random couplings
  for \eqref{eq:42}.
\end{proposition}
\begin{proof}
  By \eqref{eq:17} it clear that the infimum of the quantities in the
  right-hand side of
  \eqref{eq:42} is larger than $\sfW_2^2(\ttM_1,\ttM_2)$ and
  that if $\bttP$ satisfies the optimal condition
  \begin{displaymath}
    \sfW_2^2(\ttM_1,\ttM_2)=\int \Big(\int |x_1-x_2|^2\,\d\ggamma(x_1,x_2)\Big)\,\d\bttP(\ggamma),
  \end{displaymath}
  then $\pPi=(\pi^1_\sharp,\pi^2_\sharp)_\sharp \bttP$ is optimal thanks
  to \eqref{eq:17}.
  
  On the other hand, we can prove the equality by choosing an optimal coupling $\pPi\in
  \Gamma_o(\ttM_1,\ttM_2)$
  and a corresponding lift $\bttP$ as in \eqref{eq:40}.
\end{proof}

\subsection{\texorpdfstring{$L^2$}{}-Optimal Transport for laws of random measures}
\label{subsec:L2OT}
Let us first introduce the natural  pairing in $\RW2$
associated with
the maximal correlation pairing
$\msc\cdot\cdot$.
\begin{definition}
  \label{def:msc2}
  For every $\ttM_1,\ttM_2\in \RW2$ we set
  \begin{equation}
    \label{eq:30}
    \mmsc {\ttM_1}{\ttM_2}:=
    \max\Big\{\int_{\cP_2(\rmH)\times \cP_2(\rmH)} \msc\mu\nu\,\d\pPi(\mu,\nu):
    \pPi\in \Gamma(\ttM_1,{\ttM_2})\Big\}.
  \end{equation}
\end{definition}
\noindent
Recalling the identity \eqref{eq:W-decomp} we immediately have
the corresponding property at the level of $\sfW_2$:
\begin{lemma}
  \label{le:trivial}
  For every $\ttM_1,\ttM_2\in \RW2$ we have
  \begin{equation}
    \label{eq:41}
    \sfW_2^2(\ttM_1,{\ttM_2})=\sfM_2^2(\ttM_1)+\sfM_2^2({\ttM_2})-
    2\mmsc {\ttM_1}{\ttM_2}.
  \end{equation}
  In particular, the class of optimal couplings for \eqref{eq:30}
  coincides with
  the class $\Gamma_o(\ttM_1,\ttM_2)$ of optimal couplings for \eqref{eq:5}.
  Moreover, we have the equivalent formulation
  \begin{equation}
    \label{eq:81}
    \mmsc {\ttM_1}{\ttM_2}=
    \max
    \Big\{
     \int \Big(\int
     \langle x_1,x_2\rangle \,\d\ggamma(x_1,x_2)\Big)\,\d\bttP(\ggamma):
     \bttP\in \RGamma(\ttM_1,\ttM_2)
    \Big\}
  \end{equation}
  whose solution is provided by the same class
    $\RGamma_o(\ttM_1,\ttM_2)$  of optimal random couplings of \eqref{eq:42}.
  \end{lemma}
The general duality result for Optimal Transport and 
Corollary \ref{cor:semiconc} then yield:
\begin{theorem}[Optimal Kantorovich potentials for ROT]
  \label{thm:duality}
  For every proper lower semicontinuous
  function $\zeta:\cP_2(\rmH)\to\Rinf$
  \begin{equation}
    \label{eq:44}
    \begin{aligned}
      \int \zeta(\mu)\,\d \ttM_1(\mu)+ \int \zeta^\star(\nu)\,\d \ttM_2(\nu)
      \ge
      \mmsc {\ttM_1}{\ttM_2}
    \end{aligned}
  \end{equation}
  and there exists a totally convex, lower
    semicontinuous and proper function
    $\phi:\cP_2(\rmH)\to\Rinf$ such that
    \begin{equation}
      \label{eq:28bis}
      \int \phi(\mu)\,\d \ttM_1(\mu)+
      \int\phi^\star(\nu)\,\d \ttM_2(\nu)=\mmsc {\ttM_1}{\ttM_2}.
    \end{equation}
    The corresponding potential $\cU
    =\frac 12 \sfm_2^2-\phi$ defined as in \eqref{eq:37} 
    satisfy
    \begin{equation}
      \label{eq:29}
      \cV=\cU^\sfc=
      \frac 12\sfm_2^2-\phi^\star,\quad
      \int \cU(\mu)\,\d \ttM_1(\mu)+
      \int \cU^\sfc(\nu)\,\d \ttM_2(\nu)=\frac 12\sfW_2^2({\ttM_1},{\ttM_2}),
    \end{equation}
    with respect to the cost
    $\sfc:=\frac 12\sfw_2^2$ 
\end{theorem}
We collect now the main results concerning optimality and duality.
\begin{theorem}[Optimality conditions]
  \label{thm:Random-Brenier}
  Let $\ttM_1,\ttM_2\in \RW2$,
  $\pPi\in \Gamma(\ttM_1,\ttM_2)\subset
  \cP_2(\cP_2(\rmH)\times \cP_2(\rmH))$,
  $\bttP\in \RWW2$ supported in
  $ \cP_{2,o}(\rmH\times \rmH)$
  and associated with $\pPi$ via $(\pi^1_\sharp,\pi^2_\sharp)_\sharp\bttP=
  \pPi$ as in \eqref{eq:40}, so that 
  in particular 
  \begin{equation}
      \label{eq:support-relations}
      (\pi^1_\sharp,\pi^2_\sharp)(\supp \bttP)
      \subset \supp\pPi
      \subset 
      \overline{(\pi^1_\sharp,\pi^2_\sharp)(\supp \bttP)},\quad 
      \supp\bttP
      \subset 
      (\pi^1_\sharp,\pi^2_\sharp)^{-1}(\supp\pPi)\cap\cP_{2,o}(\rmH^2).      
  \end{equation}
  The following properties are equivalent:
  \begin{enumerate}
  \item
    $\pPi$ is an optimal plan in $\Gamma_o({\ttM_1},{\ttM_2})$
    for $\sfW_2$ or, equivalently, for $\mmsc\cdot\cdot$.
  \item $\supp(\pPi)$ is $\sfw_2^2$-cyclically monotone (recall Theorem \ref{thm:brmF-vs-frF}).
  \item 
  $\bttP$ is an optimal random coupling law in $\RGamma_o(\ttM_1,\ttM_2)$
  for $\sfW_2$ (according to 
  \eqref{eq:42}) or, equivalently, 
  for $\mmsc\cdot\cdot$ (according to 
  \eqref{eq:81}).
  \item $\supp (\bttP)$ is totally cyclically monotone.
  \item There exists a totally convex, lower
    semicontinuous and proper function
    $\phi:\cP_2(\rmH)\to\Rinf$ such that 
    \begin{equation}
      \label{eq:28}
      \phi(\mu)+\phi^\star(\nu)=\msc\mu\nu
      \quad\text{for $\pPi$-a.e.~$(\mu,\nu)\in \cP_2(\rmH)\times \cP_2(\rmH)$},
    \end{equation}
    i.e.~$\supp(\pPi)\subset 
    \partial^-\phi$. Moreover
    such a property holds for 
    \emph{every} pair of optimal Kantorovich potentials satisfying \eqref{eq:28bis}.
    \item There exists a totally convex, lower
    semicontinuous and proper function
    $\phi:\cP_2(\rmH)\to\Rinf$ such that
    \begin{equation}
      \label{eq:28-2}
      \phi(\pi^1_\sharp\ggamma)+\phi^\star(\pi^2_\sharp\ggamma)=\int \la x_1,x_2\ra\,\d\ggamma 
      \quad\text{for $\bttP$-a.e.~$\ggamma\in \cP_2(\rmH\times \rmH)$},
    \end{equation}
    i.e.~$\supp(\bttP)\subset \bpartialt\phi.$
    Moreover, such a property holds for every pair of optimal Kantorovich potentials satisfying 
    \eqref{eq:28bis}.
  \end{enumerate}
\end{theorem}
\begin{proof}
  The equivalence 1$\Leftrightarrow$2 follows by the general theory of
  optimal transport.

  The equivalence 1$\Leftrightarrow$3 follows by 
  Proposition \ref{prop:W_2-random}.

  Clearly 2$\Rightarrow$3
  by the last inclusion of 
  \eqref{eq:support-relations} 
  and the second part of Theorem \ref{thm:brmF-vs-frF}.
  On the other hand, if $\supp(\bttP)$ is totally cyclically monotone
  the same Theorem 
  \ref{thm:brmF-vs-frF} and 
  the inclusion 
  $\supp\pPi \subset 
  \overline{(\pi^1_\sharp,\pi^2_\sharp)
  \supp\bttP}$ 
  shows that 
  $\supp\pPi $ is $\sfw_2^2$-cyclically
  monotone, and therefore $\pPi$ is optimal. 
  
  In order to prove the implication
  1$\Rightarrow$5 it is sufficient to select
  an optimal totally convex proper and l.s.c.~function $\phi$
  satisfying
  \eqref{eq:28bis}.
  The optimality of $\pPi$ yields
  \begin{equation}
    \label{eq:45}
    \int \Big(\phi(\mu_1)+\phi^\star(\mu_2)-\msc{\mu_1}{\mu_2}\Big)\,\d\pPi(\mu_1,\mu_2)=0
  \end{equation}
  which implies \eqref{eq:28} thanks to the Kantorovich-Fenchel
  inequality \eqref{eq:23}.
  On the other hand,
  if $\pPi$ satisfies \eqref{eq:28} we get
  \eqref{eq:45} and the optimality of $\pPi$ thanks to \eqref{eq:44}.

  A similar argument shows the equivalence with Claim 6.
\end{proof}
As in the usual deterministic case,
when $\ttM_2$ is concentrated on a set of measures with
uniformly bounded quadratic moment, we
can find a Lipschitz totally convex optimal Kantorovich potential.
\begin{corollary}
  \label{cor:Lip}
  Let us suppose that there exists $R>0$ such that 
  \begin{equation}
    \label{eq:100}
    \sfm_2(\mu)\le R\quad\text{for $\ttM_2$-a.e.~$\mu\in \cP_2(\rmH)$}.
  \end{equation}
  Then we can find a totally convex $R$-Lipschitz function
  $\phi_R:\cP_2(\rmH)\to \R$ satisfying
  \eqref{eq:28bis}.
\end{corollary}
\begin{proof}
  Let $\phi$ as in Claim 5 of Theorem \ref{thm:Random-Brenier}
  and let us set
  \begin{displaymath}
    \psi_R(\nu):=
    \phi^\star(\nu)\quad\text{if }\sfm_2(\nu)\le R,\quad
    \psi_R(\nu):=+\infty\text{ otherwise.}
  \end{displaymath}
  Clearly $\psi_R$ is proper, totally convex and lower semicontinuous;
  the pair $(\phi,\psi_R)$ satisfies
  \begin{displaymath}
    \phi(\mu)+\psi_R(\nu)\ge \msc\mu\nu,\quad
    \phi(\mu)+\psi_R(\nu)= \msc\mu\nu
    \quad\text{for $\pPi$-a.e.~$(\mu,\nu)\in \cP_2(\rmH)
      \times \cP_2(\rmH)$}.
  \end{displaymath}
  
  It is easy to check that $\phi_R:=(\psi_R)^\star$ still provides
  an optimal totally convex Kantorovich potential
  with
  $(\phi_R)^\star=\psi_R$; $\phi_R$ is also 
  $R$-Lipschitz since
  \begin{equation}
    \label{eq:101}
    \phi_R(\mu)=\sup\Big\{\msc\mu\nu-\phi^\star(\nu):
    \sfm_2(\nu)\le R\Big\}.
  \end{equation}
\end{proof}
\subsection{The structure of minimal geodesics in \texorpdfstring{$\cPP_2(\rmH)$}{}}
\label{subsec:G-structure}
We can apply 
the previous results to identify
minimal geodesics 
in $\cPP_2(\rmH)$
and to characterize their
transport structure.
Let us first recall that 
a minimal geodesic
$(\ttM_t)_{t\in [0,1]}$
connecting $\ttM_0$ to $\ttM_1$
is a curve in $\cPP_2(\rmH)$ 
satisfying
\begin{equation}
\label{eq:geodesic}
    \sfW_2(\ttM_s,\ttM_t)
    =
    |t-s|\sfW_2(\ttM_0,\ttM_1)
    \quad\text{for every $s,t\in [0,1]$.}
\end{equation}
\let\tbar\tilde
We also say that $\tbar\ttM$ is a $t$-intermediate point between $\ttM_0$ and $\ttM_1$, $t\in (0,1)$, if 
\begin{equation}
        \label{eq:intermediate}
        \sfW_2(\ttM_0,\tbar\ttM)=t
        \sfW_2(\ttM_0,\ttM_1),\quad
        \sfW_2(\tbar\ttM,\ttM_1)=(1-t)
        \sfW_2(\ttM_0,\ttM_1).
    \end{equation}
We will use the interpolating maps
$\pi^{1\shto2}_t$ of 
\eqref{eq:disp-interpolation-intro}
and we will consider Borel maps defined
in sets of the form
(recall \eqref{eq:natural-domain-Cmaps})
\begin{equation}
    \label{eq:domain}
    \cS(\rmH,D):=
    \Big\{(x,\mu)\in \rmH\times D:
    x\in \supp\mu\Big\},\quad
    D\text{ Borel subset of }\cP_2(\rmH),
    \quad 
    \cS(\rmH)=\cS(\rmH,\cP_2(\rmH)),
\end{equation}
which are Borel subsets of 
$\rmH\times \cP_2(\rmH)$
\cite[(4.23)]{CSS22}.

\begin{theorem}[Structure of minimal geodesics in $\cPP_2(\rmH)$]
    \label{thm:geo1}
    Let $\ttM_0,\ttM_1\in \cPP_2(\rmH)$,
    let $\phi,\phi^\star$
    be a pair of optimal Kantorovich potentials for $\ttM_i$,
    and let $\brmF=\bpartialt\phi\subset 
    \cP_2(\rmH\times \rmH).$
    \begin{enumerate}
        \item 
        For every optimal random coupling law
    $\bttP\in \RGamma_o(\ttM_0,\ttM_1)$
    the curve
    \begin{equation}
        \ttM_t:=(\pi_t^{1\shto2})_{\sharp\sharp}\bttP,\ t\in [0,1]\quad\text{is a minimal geodesic.}        
    \end{equation}
    \item For every $t\in (0,1)$
    the set 
    $\brmF_t:=(\pi^{1\shto2}_t)_\sharp(\brmF)$
    is closed in $\cP_2(\rmH)$
    and there exists two uniquely characterized continuous maps
    $\ff_{t,i}:\cS(\rmH,\brmF_t)\to 
    \rmH
    $, $i=0,1,$
    inverting $(\pi^{1\shto2}_t)_\sharp$ in the sense that 
    \begin{equation}
        \label{eq:inversion}
        \ggamma\in \brmF,\quad 
        \mu=(\pi^{1\shto2}_t)_\sharp\ggamma
        \quad\Rightarrow\quad
        \ggamma=(\ff_{t,0}(\cdot,\mu),
        \ff_{t,1}(\cdot,\mu))_\sharp\mu.
    \end{equation}
    Moreover, 
    $\ff_{t,i}(\cdot,\mu)$ is Lipschitz 
    in $\supp(\mu)$ and cyclically monotone
    in $\rmH$,
    the maps 
    $\FF_{t,i}:\mu\mapsto 
    \ff_{t,i}(\cdot,\mu)_\sharp\mu$
    are Lipshitz from
    $\brmF_t$ to $\cP_2(\rmH)$
    and cyclically monotone in $\cP_2(\rmH).$
  \item 
    If $t\in (0,1)$ and
    $\tbar\ttM$ is a $t$-intermediate point between $\ttM_0$ and $\ttM_1$ 
    then 
    $\supp(\tbar\ttM)\subset \brmF_t$ 
    and 
    the formula
     \begin{equation}
        \label{eq:P-from-barM}
        \tbar\bttP=
        (\GG_t)_\sharp \tbar\ttM
        \quad\text{where}
        \quad
        \GG_t(\mu):=
        (\ff_{t,0}(\cdot,\mu),
        \ff_{t,1}(\cdot,\mu))_\sharp\mu
    \end{equation}
    provides the unique 
    $\tbar\bttP\in \RGamma_o(\ttM_0,\ttM_1)$
    such that 
    $\tbar \ttM=(\pi_t^{1\shto2})_{\sharp\sharp}\tbar\bttP$
    and correspondingly the unique geodesic
    $(\ttM_s)_{s\in [0,1]}$
    connecting $\ttM_0$ to $\ttM_1$
    such that $\ttM_t=\tbar\ttM$.   
    Moreover 
    $\RGamma_o(\tbar\ttM,\ttM_i)$,
    $i=0,1$, 
    contains the unique element
    $\tbar\bttP_{t,i}$ given by
    \begin{equation}
    \label{eq:bttp-i}
    \begin{aligned}
    \tbar\bttP_{t,0}&=
    (\pi^1,\pi^{1\shto2}_t)_{\sharp\sharp}\tbar\bttP=
        (\GG_{t,0})_\sharp\tbar\ttM,\quad
        \GG_{t,0}(\mu)=
        (\ff_{t,0}(\cdot,\mu),\ii)_\sharp\mu\\
    \tbar\bttP_{t,1}&=
    (\pi^{1\shto2}_t,\pi^2)_{\sharp\sharp}\tbar\bttP=
        (\GG_{t,1})_\sharp\tbar\ttM,\quad
        \GG_{t,1}(\mu)=
        (\ii,\ff_{t,1}(\cdot,\mu))_\sharp\mu
    \end{aligned}
    \end{equation}
    which is concentrated on 
    deterministic optimal couplings, and 
    \begin{equation}
        \label{eq:ppi-i}
        \tbar\pPi_{t,i}=(\pi^1_\sharp,\pi^2_\sharp)_\sharp
    \tbar\bttP_{t,i}=
    (\id\times \FF_{t,i})_\sharp\tbar\ttM,
    \quad i=0,1,
    \end{equation}
    is the unique optimal coupling
    in $\Gamma_o(\tbar\ttM,\ttM_i)$.
        \item 
    For every 
    $t\in (0,1)$ the conjugate
    functions (recall the definition of Moreau-Yosida regularization
    \eqref{eq:Yosida})
    \begin{equation}
        \label{eq:pair1}
        \phi^t:=\frac {1-t}2\sfm_2^2+
    t\phi,\quad 
    (\phi^t)^\star=
    t(\hat \phi^\star)_{1/t-1}\circ \frd_{t^{-1}}
    \end{equation}
    provide a pair of optimal Kantorovich potentials for $\ttM_0$
 and any $t$-intermediate point 
 $\tbar \ttM$ between $\ttM_0$ and $\ttM_1.$
 Similarly 
 \begin{equation}
        \label{eq:pair2}
    (1-t)\phi_{t/(1-t)}\circ \frd_{(1-t)^{-1}},\quad 
    (1-t)\phi^\star+\frac t2\sfm_2^2
    \end{equation}
    is a pair of optimal Kantorovich potentials for $\tbar \ttM$
 and $\ttM_1$.
    \end{enumerate}
\end{theorem}
\begin{remark}
    \label{rem:comparison}
    The above Theorem recovers in a much more precise form
    various results that 
    hold for $\cP_2(\rmX)$ 
    in suitable classes of 
    metric spaces $\rmX$, see 
    \cite[Chap.~7]{Villani09}.
    See in particular the 
    nonbranching property 
    (stated in locally compact spaces)
    \cite[Corollary 7.32]{Villani09}
    for Claim 2, 
     the 
    ``interpolation of prices''
    \cite[Theorem 7.36]{Villani09}
    concerning Claim 3, and the related bibliographical notes.    
\end{remark}
\begin{proof}
    \underline{Claim 1.}
    For $0\le s< t\le 1$
    we define 
    the maps
    $\pi^{1\shto2}_{s,t}:\rmH\times \rmH\to\rmH\times\rmH$ 
    \begin{displaymath}
        \pi^{1\shto2}_{s,t}(x_0,x_1):=
        (\pi^{1\shto2}_s(x_0,x_1),
        \pi^{1\shto2}_t(x_0,x_1));
    \end{displaymath}
    we observe that 
    \begin{displaymath}
        |\pi^{1\shto2}_s(x_0,x_1)-\pi^{1\shto2}_t(x_0,x_1)|=|t-s|
    \cdot |x_1-x_0|
    \end{displaymath}
    so that 
    the random coupling
    $\bttP_{s,t}=
    (\pi^{1\shto2}_{s,t})_{\sharp\sharp}\bttP
    $
    belongs to $\RGamma(\ttM_s,\ttM_t)$
    and thus yields
    \begin{displaymath}
        \sfW_2(\ttM_s,\ttM_t)
        \le |t-s|\sfW_2(\ttM_0,\ttM_1)
        \quad\text{if }0\le s<t\le 1.
    \end{displaymath}
    The triangle inequality then yields 
    \eqref{eq:geodesic}.

    \smallskip\noindent
    \underline{Claim 2.}
    Let us consider the
    Lagrangian lifting 
$\hat\brmF=
\partial\hat \phi\subset \cH\times \cH$
of 
the total subdifferential 
$\brmF=\bpartialt\phi$ of $\phi$ and let us set
\begin{displaymath}
    I_t(X_0,X_1):=
    (1-t)X_0+
    tX_1,\quad 
    \hat\brmF_t=I_t(\hat\brmF)=\Big\{\tbar X=(1-t)X_0+
    tX_1:(X_0,X_1)\in \hat\brmF\Big\}
\end{displaymath}
which is clearly a set invariant by m.p.i.
By monotonicity, if 
$\tbar X=I_t(X_0,X_1),$ $\tbar X'=I_t(X_0',X_1')$ for $(X_0,X_1),(X_0',X_1')\in \hat\brmF$
    we have
    \begin{equation}
        \label{eq:estimate}
        |\tbar X-\tbar X'|^2\ge (1-t)|X_0-X_0'|^2+
        t|X_1-X_1'|^2,
    \end{equation}
which shows that $\hat\brmF_t$ is closed
and there exist Lipschitz maps
$F_{t,i}:\hat\brmF_t\to\rmH$ such that
\begin{equation}
    \tbar X=I_t(X_0,X_1),\ 
    (X_0,X_1)\in \hat\brmF\quad\Rightarrow\quad
    X_i=F_{t,i}(\tbar X).
\end{equation}
Since $F_{t,i}$ are also invariant by m.p.i., 
the general extension and representation Theorem 
4.8 of \cite{CSS2} 
shows that there is a unique pair of continuous maps 
$\ff_{t,i}:\cS(\rmH,\brmF_t)\to \rmH$
representing $F_{t,i}$ as 
\begin{equation}
    \label{eq:representation-f}
    F_{t,i}[\tbar X](\oomega)=
    \ff_{t,i}(\tbar X(\oomega),\iota(\tbar X))
\end{equation}
with the properties stated in Claim 2. 
\eqref{eq:representation-f} clearly yields
\eqref{eq:inversion} since
$\supp(\mu)\subset \brmF_t\subset \iota(\hat\brmF_t)$.
The cyclical monotonity of $\ff_{t,i}(\cdot,\mu)$
and of $\FF_{t,i}$ follows from the corresponding
cyclical monotonicity of 
the maps $F_{t,i}$ in $\cH$, which 
in turn follows by the fact that 
they are the inverse of 
the cyclically monotone sets 
\begin{equation}
    \label{eq:restriction}
    \begin{aligned}
    \hat\brmF_{t,0}:={}& (I_0,I_t)(\hat\brmF)=
    \Big\{(X_0,(1-t)X_0+tX_1):(X_0,X_1)\in 
    \hat\brmF \Big\},\\
     \hat\brmF_{t,1}:={}& (I_1,I_t)(\hat\brmF)=
    \Big\{(X_1,(1-t)X_0+tX_1)):(X_0,X_1)\in 
    \hat\brmF \Big\}.
    \end{aligned}
\end{equation}

    \smallskip
    \noindent\underline{Claim 3.}
    Let $\pPi_{0}\in \Gamma_o(\ttM_0,\tbar\ttM)$
    and $\pPi_{1}\in \Gamma_o(
    \tbar\ttM,\ttM_1)$.
    By the glueing Lemma 
    we find a tri-plan
    $\pPi\in \Gamma(\ttM_0,
    \tbar\ttM,\ttM_1)$
    such that 
    $\pi^{1,2}_\sharp \pPi=
    \pPi_0$,
    $\pi^{2,3}_\sharp\pPi=\pPi_1.$
    Consider now the closed set 
    \begin{equation}
        \label{eq:multiple-R}
        \cQ:=
        \Big\{
        \ggamma\in 
        \cP_2(\rmH\times \rmH\times\rmH):
        \pi^{1,2}_\sharp\ggamma
        \text{ and }
        \pi^{2,3}_\sharp\ggamma
        \text{ belong to } \cP_{2,o}(\rmH\times \rmH)\Big\}.
    \end{equation}
    Since the map 
    $(\pi^1_\sharp\times 
    \pi^2_\sharp\times \pi^3_\sharp):
    \cQ\to \big(\cP_2(\rmH)\big)^3$
    is surjective, we can find
    $\bttP\in \cP_2(\cQ)
    \subset \cPP_2(\rmH\times \rmH\times\rmH)$
    such that 
    $(\pi^1_\sharp\times 
    \pi^2_\sharp\times \pi^3_\sharp)_\sharp\bttP= \pPi$.

    We have
    \begin{displaymath}
        \pi^{1,2}_{\sharp\sharp}(\bttP)
        \in \RGamma_o(\ttM_0,\tbar\ttM),\quad
        \pi^{2,3}_{\sharp\sharp}(\bttP)
        \in \RGamma_o(\tbar \ttM,\ttM_1)
    \end{displaymath}
    and by the elementary  inequality
    $(a+b)^2\le \frac 1ta^2+\frac1{1-t}b^2$
    \begin{align*}
           \sfW_2^2(\ttM_0,\ttM_1)
           &\le
           \int \int |x_1-x_3|^2\,\d\ggamma\bttP(\ggamma)
           \\&
           \le 
            \frac 1t\int \int |x_1-x_2|^2\,\d\ggamma\bttP(\ggamma)
            +
             \frac 1{1-t}\int \int |x_2-x_3|^2\,\d\ggamma\bttP(\ggamma)
             \\&=
             \frac 1t\sfW_2^2(\ttM_0,\tbar\ttM)+
             \frac 1{1-t}
             \sfW_2^2(\tbar\ttM,\ttM_1)=
             \sfW_2^2(\ttM_0,\ttM_1)
    \end{align*}
we deduce that 
$\tbar\bttP:=\pi^{1,3}_{\sharp\sharp}\bttP\in \RGamma_o(\ttM_0,\ttM_1)$
and 
\begin{displaymath}
    |x_1-x_3|^2=
    \frac 1t |x_1-x_2|^2+
    \frac 1{1-t}|x_2-x_3|^2
    \text{ on }\supp(\ggamma)
    \quad\text{for $\bttP$-a.e.~$\ggamma$},
\end{displaymath}
so that 
$\bttP$-a.e.~$\ggamma$ is supported in the set 
\begin{equation}
    \label{eq:linear}
    \rmH^3_t:=
    \Big\{(x_1,x_2,x_3)\in \rmH^3:
    x_2=(1-t)x_1+tx_3\Big\}
\end{equation}
and therefore 
$\tbar\ttM=(\pi^{1\shto2}_t)_{\sharp\sharp}\tbar\bttP.$

In order to show that $\tbar\bttP$ is unique,
we observe that any random coupling
in $\RGamma_o(\ttM_0,\ttM_1)$
has support in $\brmF$ 
so that $\tbar\ttM$ has support
in $\brmF_t$ and 
we can then apply \eqref{eq:inversion}
which yields \eqref{eq:P-from-barM}.
The above discussion also shows that 
$\RGamma_o(\tbar\ttM,\ttM_i)$
and $\Gamma_o(\tbar\ttM,\ttM_i)$
are uniqueley characterized by 
\eqref{eq:bttp-i} and \eqref{eq:ppi-i}

  \smallskip
    \noindent\underline{Claim 4.}
    In order to check \eqref{eq:pair1} 
    (the argument for \eqref{eq:pair2} is similar)
    we use the fact that 
    $\tilde\ttM$ is supported
    in $\brmF_t=\iota(\hat\brmF_t)$.
    On the other hand,
    the set  
    $\hat \brmF_{t,0}$ defined by 
    \eqref{eq:restriction}
    is the graph of the 
    subdifferential of
    the Lagrangian lifting of $\phi^t$
    \begin{displaymath}
        \hat\phi^t(X):=\frac {1-t}2\|X\|_{\cH}^2
        +t\hat \phi(X)
    \end{displaymath}
    whose Legendre-Fenchel transform
    can be expressed in terms of 
    the Moreau-Yosida regularization of 
    $\hat\phi^\Ast$ (recall \eqref{eq:Y-in-cH}
    and \eqref{eq:useful-for-dopo})
    and corresponds to the Lagrangian lifting
    of the functional $(\phi^t)^\star$ given in 
    \eqref{eq:pair1}.
    Since
    \begin{displaymath}
        \hat\phi^t(X)+
        (\hat\phi^t)^\Ast(Y)=
        \langle X,Y\rangle_\cH=
        \msc{\iota(X)}{\iota(Y)}
        \quad \text{for every }
        (X,Y)\in \hat\brmF_{t,0}
    \end{displaymath}
    we get the proof of the claim. 
\end{proof}

\subsection{Lifting (laws of) random measures of \texorpdfstring{$\RW2$}{} to measures on Lagrangian maps in \texorpdfstring{$\cP_2(\cH)$}{}.}
\label{subsec:lifting}
In the previous sections, we exploited the lifting
technique of Proposition \ref{prop:W_2-random}
in order to describe optimal couplings in $\cP_2(\rW2\times \rW2)$
in terms of laws of random optimal couplings in $\RWW2$.

There is another lifting technique which is induced by 
the $1$-Lipschitz and surjective law map $\iota:\cH\to \cP_2(\rmH)$.
The corresponding push-forward transformation $\iota_\sharp$ still provides
a surjective map from $\cP_2(\cH)$
(the space of measures on Lagrangian maps of $\cH$)
to $\RW2$ (the space of (laws of) random measures), so that
it is natural to study the relations between optimal transport
problems
in $\RW2$ and in $\cP_2(\cH)$.

First of all, since $\iota$ is $1$-Lipschitz,
we observe that for every $\frm_i\in \cP_2(\cH)$, $i=1,2$, we have
\begin{equation}
  \label{eq:43}
  \begin{aligned}
    \ttM_i=\iota_\sharp {\frm_i}\quad&\Rightarrow\quad
    \sfW_2(\ttM_1,\ttM_2)\le \sfW_{2,\cH}({\frm_1},{\frm_2}).
  \end{aligned}
\end{equation}
Similarly, it is not difficult to check that given a coupling
$\frp\in \cP_2(\cH\times \cH)$ and setting
$\iota_i:=\iota\circ\pi^i$,
we have 
\begin{equation}
  \label{eq:43bis}
  \begin{aligned}
   \pPi=(\iota_1,\iota_2)_\sharp{\frp}
    \quad&\Rightarrow\quad
    \int \sfw_2^2(\mu_1,\mu_2)\,\d\pPi(\mu_1,\mu_2)
    \le
    \int \|X_1-X_2\|_\cH^2 \,\d{\frp}(X_1,X_2).\\
  \end{aligned}
\end{equation}
Eventually, still starting from
$\frp\in \cP_2(\cH\times \cH)$ and using
$\iota^2(X,Y):=(X,Y)_\sharp\P$, we have
\begin{equation}
  \label{eq:43tris}
  \begin{aligned}
    \bttP=\iota^2_\sharp {\frp}
    \quad&\Rightarrow\quad
    \int \int |x_1-x_2|^2\,\d\ggamma\,\d\bttP(\ggamma)
    =
    \int \|X_1-X_2\|_\cH^2 \,\d{\frp}(X_1,X_2).
  \end{aligned}
\end{equation}
Since $\iota^2$ is surjective from $\cH\times \cH$ to
$\cP_2(\rmH\times \rmH)$, 
$\iota^2_\sharp$ is surjective as well, so that given  $\bttP\in
\cP_2(\cP_2(\rmH\times\rmH))$
it is always possible to find
a lifting $\frp^\ell$ such that
$\bttP=\iota^2_\sharp {\frp^\ell}$.

If we want to 
lift $\pPi$ so that 
\eqref{eq:43bis} holds as an equality, 
we can first select 
$\bttP\in \cP_2(\cP_{2,o}(\rmH\times
\rmH))$
so that \eqref{eq:40} holds: the lifting $\frp^\ell$
satisfies the identity
\begin{equation}
  \label{eq:43bis-o}
  \begin{aligned}
   \pPi=(\iota_1,\iota_2)_\sharp\frp^\ell,
    \quad
    \int \sfw_2^2(\mu_1,\mu_2)\,\d\pPi(\mu_1,\mu_2)
    =
    \int \|X_1-X_2\|_\cH^2 \,\d\frp^\ell(X_1,X_2).
  \end{aligned}
\end{equation}
If moreover $\pPi\in \Gamma_o(\ttM_1,\ttM_2)$
then $\bttP\in \RGamma_o(\ttM_1,\ttM_2)$
and setting $\frm_1^\ell=\pi^1_\sharp \frp^\ell,
\frm_2^\ell=\pi^2_\sharp\frp^\ell$ we get
\begin{equation}
  \label{eq:43-o}
  \begin{aligned}
    \ttM_i=\iota_\sharp \frm_i^\ell,\quad
    \sfW_2(\ttM_1,\ttM_2)=\sfW_{2,\cH}(\frm_1^\ell,\frm_2^\ell).
  \end{aligned}
\end{equation}
We recap the above argument in the next proposition.
\begin{proposition}[From random OT to OT in $\cH$]
  \label{prop:lift-cH}
  For every pair $\ttM_1,\ttM_2\in \RW2$ there exists a pair
  $\frm_1^\ell,\frm_2^\ell\in \cP_2(\cH)$ such that \eqref{eq:43-o} holds,
  so that
  \begin{equation}
    \label{eq:83}
    \sfW_2(\ttM_1,\ttM_2)=
    \min\Big\{\sfW_{2,\cH}(\frm_1,\frm_2):\frm_i\in \cP_2(\cH),\
    \iota_\sharp \frm_i=\ttM_i\Big\}.
  \end{equation}
  If $\frm_1^\ell,\frm_2^\ell\in \cP_2(\cH)$ 
  are minimizers of \eqref{eq:83} and $\frp^\ell\in \Gamma_o(\frm_1^\ell,\frm_2^\ell)$ in $\cP_2(\cH\times \cH)$
  then
  $\bttP=\iota^2_\sharp\frp^\ell\in \RGamma_o(\ttM_1,\ttM_2)$
  and
  $\text{$\mathit\pPi$}=(\pi^1_\sharp,\pi^2_\sharp)\bttP=(\iota_1,\iota_2)_\sharp\frp^\ell\in \Gamma_o(\ttM_1,\ttM_2)$.
\end{proposition}
We want to highlight that the optimal lifted measures
$\frm_1^\ell,\frm_2^\ell$ given by the previous proposition
typically depend on both the measures $\ttM_1,\ttM_2$
and in general we cannot fix an arbitrary $\frm_1^\ell$ such that
$\iota_\sharp \frm_1^\ell=\ttM_1$. 
We want to find a sufficient condition on $\ttM_1,\ttM_2$ for which
the following property holds:
\begin{equation}
  \label{eq:50}
  \begin{gathered}
    \text{for every $\frm_1\in \cP_2(\cH)$ such that\quad 
      $\iota_\sharp \frm_1=\ttM_1$\quad  there exists $\frm_2\in \cP_2(\cH)$ with}\\
    \iota_\sharp \frm_2=\ttM_2,\quad
    \sfW_2(\ttM_1,\ttM_2)=\sfW_{2,\cH}({\frm_1},{\frm_2}).
  \end{gathered}
\end{equation}
A crucial role in this respect is played by 
the set of ``deterministic'' couplings
$\cP_2^{\rm det}(\rmH\times \rmH)$ 
which are concentrated on maps:
%
%
\begin{equation}
  \label{eq:49}
  \begin{aligned}
    \cPdet2(\rmH\times \rmH):={}& \Big\{(\ii\times
    \ff)_\sharp\mu:\mu\in \cP_2(\rmH),\ \ff\in L^2(\rmH,\mu;\rmH)\Big\},
    \\\cP_{2,o}^{\rm det}(\rmH\times \rmH):={}& \cP_2^{\rm
      det}(\rmH\times \rmH)\cap \cP_{2,o}(\rmH\times \rmH).
  \end{aligned}
\end{equation}
First of all, we will show a simple condition for which
there exists an optimal coupling $\bttP\in \RGamma_o(\ttM_1,\ttM_2)$
which is concentrated on $\cPdet{2,o}(\rmH\times \rmH)$.
\begin{lemma}
  \label{le:selection}
  An optimal coupling $\pPi\in \cP_{2,o}(\cP_2(\rmH)\times
  \cP_2(\rmH))$ satisfies the property
  \begin{equation}
    \label{eq:51}
    \text{for $\pPi$-a.e.~$(\mu_1,\mu_2)$}
    \quad
    \Gamma_o(\mu_1,\mu_2)\cap \cP_2^{\rm det}(\rmH\times \rmH)\neq \emptyset
  \end{equation}
  if and only if
  \begin{equation}
  \text{there exists
    $\bttP\in \cP_2(\cP_{2,o}^{\rm det}(\rmH\times \rmH))$ such that}
  \quad
  \pPi=(\pi^1_\sharp,\pi^2_\sharp)_\sharp\bttP.
  \label{eq:52}
\end{equation}  
\end{lemma}
\begin{proof}
    Let us set $\cO:=
    \Big\{(\mu_1,\mu_2)\in \cP_2(\rmH)
    \times \cP_2(\rmH): 
    \Gamma_o(\mu_1,\mu_2)\cap \cP_2^{\rm det}(\rmH\times \rmH)\neq \emptyset\Big\}.$
    We can equivalently characterize $\cO$ 
    as the image 
    of the Borel set $\cP_{2,o}^{\rm det}(\rmH\times \rmH)$ through the continuous map 
    $\pi^1_\sharp\times \pi^2_\sharp$,
    so that $\cO$ is a Souslin (and therefore universally measurable) set.  

    \eqref{eq:51} just says that  $\pPi$ is concentrated on $\cO$, and therefore its equivalence with \eqref{eq:52} follows
    by Theorem \ref{thm:VN}. 
\end{proof}
\begin{theorem}
  \label{thm:det-is-useful}
  Let $\text{$\mathit \Pi$}\in \Gamma_o(\ttM_1,\ttM_2)$ satisfy \eqref{eq:51},
  let $\bttP$ as in \eqref{eq:52}, 
  and let $\frm_1\in \cP_2(\cH)$ such that
  $\iota_\sharp\frm_1=\ttM_1$.
  Then there exists
  $\frp^\ell\in \cP_{2,o}(\cH\times \cH)$ such that
  $\pi^1_\sharp \frp^\ell=\frm_1$ and 
  $\iota^2_\sharp\frp^\ell=\bttP$, so that, in particular,
  $(\iota_1,\iota_2)_\sharp\frp=\pPi$
  and setting $\frm_2=\pi^2_\sharp\frp^\ell$ we have
  \begin{equation}
    \label{eq:53}
    \iota_\sharp\frm_2=\ttM_2,\quad
    \sfW_2(\ttM_1,\ttM_2)=\sfW_{2,\cH}(\frm_1,\frm_2).
  \end{equation}
\end{theorem}
\begin{proof}
  Let us first consider the
  closed subset $\cA$ of $\cH\times \cP_2(\rmH)\times \cP_2(\rmH\times \rmH) $ defined by 
  \begin{equation}
    \label{eq:84}
    \cA:=\Big\{(X,\mu,\ggamma)\in \cH\times \cP_2(\rmH)\times\cP_2(\rmH\times \rmH):
    \iota(X)=\mu=\pi^1_\sharp(\ggamma)\Big\},
  \end{equation}
  and the map $\rmA:\cH\times \cH\to \cA$
  \begin{equation}
    \label{eq:85}
    \rmA(X,Y):=(X,\iota(X),\iota^2(X,Y))=
    (X,X_\sharp\P,(X,Y)_\sharp\P).
  \end{equation}
  $\rmA$ is continuous but in general it is not surjective.
  However, it is not difficult to check that the image of $\rmA$
  contains the Borel set
  \begin{equation}
    \label{eq:86}
    \begin{aligned}
      \cA^{\rm det}:=
      {}&\Big\{(X,\mu,\ggamma)\in \cH\times
      \cP_2(\rmH)\times \cPdet2(\rmH\times \rmH):
      \iota(X)=\mu=\pi^1_\sharp(\ggamma)\Big\}
      \\={}& \cA\cap \Big(\cH\times
      \cP_2(\rmH)\times \cPdet2(\rmH\times \rmH)\Big).
    \end{aligned}
  \end{equation}
  In fact, if $(X,\mu,\ggamma)\in \cA^{\rm det}$ then
  $\mu=\pi^1_\sharp\ggamma$
  and we can find a map $f_\ggamma\in L^2(\rmH,\mu;\rmH)$ such that
  $\ggamma=(\ii\times f_\ggamma)_\sharp\mu$.
  Defining $Y:=f_\ggamma\circ X$
  we immediately see that $\iota^2(X,Y)=\ggamma$, so that
  $\rmA(X,Y)=(X,\mu,\ggamma)$.

  Let us now set
  $\bar \frm_1:=(\id\times\iota)_\sharp \frm_1
  \in \cP_2(\cH\times \cP_2(\rmH))$
  and
  $\bar \bttP:=(\pi^1_\sharp\times \id)_\sharp\bttP
  \in \cP_2(\cP_2(\rmH)\times \cPdet2(\rmH\times \rmH))$.
  By assumption
  \begin{displaymath}
    \pi^2_\sharp\bar \frm_1=\iota_\sharp \frm_1=\ttM_1,\quad
    \pi^1_\sharp \bar\bttP=(\pi^1_\sharp)_\sharp\bttP=\ttM_1
  \end{displaymath}
  so that, by the gluing Lemma, we can find
  a plan $\qQ\in \cP_2(\cH\times \cP_2(\rmH)\times \cPdet2(\rmH\times
  \rmH))$
  such that $\pi^{1\,2}_\sharp\qQ=\bar \frm_1$ and
  $\pi^{2\,3}_\sharp \qQ=\bar\bttP$.
  By construction, $\bar \frm_1$ is concentrated
  on the set
  $$\Big\{(X,\mu)\in \cH\times \rW2:\iota(X)=\mu\Big\}$$
  and
  $\bar\bttP$ is concentrated on the set
  $$\Big\{(\mu,\ggamma)\in \cH\times \cPdet2(\rmH\times \rmH):
  \mu=\pi^1_\sharp\ggamma\Big\}$$
  we deduce that $\qQ$ is concentrated on $\cA^{\rm det}$.

  By Theorem \ref{thm:VN}
  we can find a probability measure
  $\frp^\ell\in \cP_2(\cH\times\cH)$ such that
  $\rmA_\sharp \frp^\ell=\qQ$.
  By the very definition of $\rmA$ we get
  \begin{displaymath}
    \pi^1_\sharp \frp^\ell=(\pi^1\circ\rmA)_\sharp \frp^\ell=
    \pi^1_\sharp \qQ=\frm_1,\quad
    \iota^2_\sharp \frp^\ell=(\pi^3\circ\rmA)_\sharp\bttP=
    \pi^3_\sharp\qQ=\bttP
  \end{displaymath}
  and the thesis follows.
\end{proof}
\begin{corollary}
  \label{cor:regular1}
  Let $\ttM_1,\ttM_2\in \RW2$
  and let us suppose that $\ttM_1$ is concentrated on $\cP_2^r(\rmH)$. 
  Then \eqref{eq:50} holds.  
\end{corollary}
\section{Random Gaussian-null sets and strict Monge formulation of OT via nonlocal totally cyclically monotone fields}
\label{sec:Monge}

In this last section we want to address the uniqueness and the 
Monge formulation of the $L^2$-OT problem in $\RW2$.
These questions can be settled at the usual level of
couplings of (laws of) random measures. We then seek for conditions
on $\ttM_i\in \RW2$ ensuring that
the class of optimal coupling
$\Gamma_o(\ttM_1,\ttM_2)$ contains a unique element $\pPi$
which is concentrated on the graph of a Borel map
$\FF:\rW2\to\rW2$, so that
\begin{equation}
  \label{eq:88}
  \pPi=(\id\times \FF)_\sharp \ttM_1,\quad
  \ttM_2=\FF_\sharp \ttM_1,\quad
  \sfW_2^2(\ttM_1,\ttM_2)=\int_{\rW2}
  \sfw_2^2(\mu,\FF(\mu))\,\d \ttM_1(\mu);
\end{equation}
this property is equivalent to asking for
$\pPi\in \cPdet2(\rW2\times\rW2)$.

A second formulation involves random couplings and provides a more
refined description of $\FF$: we look for conditions ensuring
that $\RGamma_o(\ttM_1,\ttM_2)$ contains a unique element $\bttP$
that is concentrated on the graph of a deterministic
totally cyclically monotone field $\ff:\rmH\times \cP_2(\rmH)\to\rmH$.

\subsection{The strict Monge formulation}
\label{subsec:strict-Monge}

In order to describe such a construction, we first observe that every
$\bttP\in \cPP_2(\rmH\times \rmH)$ can be disintegrated with respect to
$\ttM_1=\pi^1_{\sharp\sharp}\bttP$ to obtain a
Borel family $\bttP_\mu\in \RWW2$ indexed by $\mu\in \rW2$ and
concentrated on the set of couplings
$\Gamma(\mu):=\Big\{\ggamma\in \cP_2(\rmH\times\rmH):
\pi^1_\sharp\ggamma=\mu\Big\}$
for $\ttM_1$-a.e.~$\mu\in \cP_2(\rmH)$.
When $\supp(\bttP)\subset 
\cP_{2,o}(\rmH\times \rmH)$ 
(as in the case of optimal couplings)
then 
$\supp(\bttP_\mu)\subset 
\Gamma_o(\mu):=\Gamma(\mu)\cap\cP_{2,o}(\rmH\times \rmH) $.

If $\pPi=(\pi^1_\sharp,\pi^2_\sharp)_\sharp\bttP$ 
then $\bttP_\mu$ 
characterizes the 
family of measures in 
$\cPP_2(\rmH)$ arising from the disintegration
$(\pPi_\mu)_{\mu\in \cP_2(\rmH)}$
of $\pPi$ with respect to its first marginal
through the formula
\begin{equation}
    \label{eq:disintegrations}
    \pPi_\mu=\pi^2_{\sharp\sharp}\bttP_\mu
    \quad
    \text{for $\ttM_1$-a.e.~$\mu$}.
\end{equation}

If $\pi^1_\sharp$ is essentially injective with respect to $\bttP$ then there is a Borel map
$\GG:\cP_2(\rmH)\to 
\cPP_2(\rmH\times \rmH)$ such that 
for $\ttM_1$-a.e.~$\mu$
$\bttP_\mu$ is concentrated on
a unique coupling $\ggamma=\GG(\mu)$ with first marginal $\mu$, so
that $\pPi$ is deterministic as in \eqref{eq:88} and 
we have
\begin{equation}
  \label{eq:92}
  \bttP_\mu=\delta_{\GG(\mu)},
  \quad \GG(\mu)\in   \Gamma(\mu,\FF(\mu)),\quad
  \FF(\mu)=\pi^2_{\sharp}\GG(\mu),\quad
  \ttM_2=(\pi^2_\sharp\circ\GG)_\sharp \ttM_1=
  \pi^2_{\sharp\sharp}\GG_\sharp \ttM_1.
\end{equation}
We can then
apply to $\bttP_\mu$
the disintegration map $\cK$, obtaining the decomposition
$\GG(\mu)=\mu\otimes \kappa_{x,\mu}$,
$\kappa_{x,\mu}=\cK(x,\GG(\mu))$.
If $\bttP$ is concentrated on
$\cPdet2(\rmH\times \rmH)$ then
$\bttP_\mu$ is concentrated on $\cPdet2(\rmH\times \rmH)$ as well,
so that $\kappa_{x,\mu}=\delta_{\ff(x,\mu)}$ for some Borel map
$\ff:\rmH\times \cP_2(\rmH)\to \rmH$.
We then obtain
\begin{equation}
  \label{eq:92bis}
  \GG(\mu)=(\ii\times \ff(\cdot,\mu))_\sharp \mu ,\quad 
  \FF(\mu)=\ff(\cdot,\mu)_\sharp \mu.
\end{equation}
Since $\FF(\mu)\in \cP_2(\rmH)$ and
$\ttM_2\in \RW2$ we have
\begin{equation}
  \label{eq:93}
  \begin{aligned}
    \int_{\rmH}|\ff(x,\mu)|^2\,\d\mu(x)&=
    \sfm_2^2(\FF(\mu))<\infty
    \quad\text{for $\ttM_1$-a.e.~$\mu\in \cP_2(\rmH)$}\\
    \int_{\rW2}
    \Big(\int_{\rmH}|\ff(x,\mu)|^2\,\d\mu(x)\Big)
    \,\d \ttM_1(\mu)&= \int_{\rW2} \sfm_2^2(\FF(\mu))\,\d
    \ttM_1(\mu)=\sfM_2^2(\ttM_2)<\infty.
  \end{aligned}
\end{equation}
%
%
%
It is then convenient to represent $\ff$ as a $\rmH$-valued $L^2$ map
of the unfolded measure
\begin{equation}
  \label{eq:55}
  \bar \ttM_1:=\int (\mu\otimes \delta_\mu)\,\d \ttM_1(\mu)\in
  \cP_2(\rmH\times \cP_2(\rmH)), 
\end{equation}
which satisfies $\pi^2_\sharp\bar \ttM_1=\ttM_1$.
In fact if $\ff\in L^2(\bar \ttM_1;\rmH)$ then
\begin{equation}
  \label{eq:56}
  \|\ff\|_{L^2(\bar \ttM_1;\rmH)}^2=
  \int_{\cP_2(\rmH)}\int_\rmH |\ff(x,\mu)|^2\,\d\mu\,\d \ttM_1(\mu)<\infty
\end{equation}
and we can represent the 
corresponding unfolded measure $\bar \ttM_2$ as 
\begin{equation}
    \label{eq:unfolded-representation}
    \bar \ttM_2=(\ff,\FF)_\sharp \bar \ttM_1,
    \quad 
    \bar \ttM_i=
    \int (\mu\otimes \delta_\mu)\,\d \ttM_i(\mu),
\end{equation}
since
\begin{align*}
    \int \xi(y,\nu)\,\d\bar \ttM_2(y,\nu)
    &=
    \int \Big(
    \int \xi(y,\nu)\,\d\nu(y)\Big)
    \,\d \ttM_2(\nu)
    \\&=
    \int \Big(
    \int \xi(y,\FF(\mu))\,\d\big(\FF(\mu)\big)(y)\Big)
    \,\d \ttM_1(\mu)
     \\&=
    \int \Big(
    \int \xi(\ff(x,\mu),\FF(\mu))\,\d\mu(x)\Big)
    \,\d \ttM_1(\mu)
    \\&=
    \int  \xi(\ff(x,\mu),\FF(\mu))\,\d\bar \ttM_1(x,\mu).
\end{align*}
The above remarks justify the following definition.
\begin{definition}[Fully deterministic random couplings]
  \label{def:detRC}
  We say that a random coupling law $\bttP\in \RWW2$
  is fully deterministic if
  \begin{equation}
    \label{eq:94}
    \pi^1_\sharp\text{ is $\bttP$-essentially injective and }
    \bttP\text{ is concentrated on }\cPdet2(\rmH\times \rmH).
  \end{equation}
  We denote by $\RWWdet2$ the set of fully deterministic random couplings.
\end{definition}
\begin{lemma}[Representation of fully deterministic random couplings]
  \label{le:fdet}
  A random coupling law \\
  $\bttP\in \RWW2$
  with first random marginal $\ttM_1=\pi^1_{\sharp\sharp}\bttP$
  is fully deterministic
  if and only if
  there exists a Borel map $\ff\in L^2(\bar \ttM_1;\rmH)$ such that
  \begin{equation}
    \label{eq:95}
    \bttP=\int \delta_{(\ii\times
      \ff(\cdot,\mu))_\sharp\mu}\,\d \ttM_1(\mu)=\GG_\sharp \ttM_1,\quad
    \GG(\mu):=(\ii\times
      \ff(\cdot,\mu))_\sharp\mu.
    \end{equation}
    In this case setting
    $\FF(\mu):=\ff(\cdot,\mu)_\sharp\mu$
    we have
    \begin{equation}
      \label{eq:96}
      \begin{gathered}
        \ttM_2=\pi^2_{\sharp\sharp}\bttP=\FF_\sharp \ttM_1\quad
       ,
        \bar \ttM_2=(\ff,\FF)_\sharp \bar \ttM_1
        \\
        \int_{\cP_2(\rmH\times
          \rmH)} \Big(\int_{\rmH\times \rmH}
        |x-y|^2\,\d\ggamma\Big)\,\d\bttP(\ggamma)= \int_{\rmH\times
          \cP_2(\rmH)} |\ff(x,\mu)-x|^2\,\d\mu(x)\,\d \ttM_1(\mu).
      \end{gathered}
    \end{equation}
  \end{lemma}
  Using the unfolding $\bar \bttP$ of $\bttP$
  we can also express \eqref{eq:96} as
  \begin{equation}
      \label{eq:unfolded-full}
      \bar\bttP=
      (\ii\times \ff,\GG)_\sharp\bar\ttM_1,\quad 
      \int |x-y|^2\,\d\bar\bttP(x,y,\ggamma)=
      \int|\ff(x,\mu)-x|^2\,\d\bar \ttM_1(x,\mu).
  \end{equation}
  \begin{proof}
    Because of the previous digression, we have just to show the converse direction: given a Borel map
    $\ff\in L^2(\bar \ttM_1;\rmH)$
    the coupling $\bttP$ given by \eqref{eq:95} is well defined,
    i.e.~the map $\GG:\cP(\rmH)\to\cP(\rmH\times \rmH)$
    is Borel. 
    This property follows by standard argument, see e.g.~
    \cite[Lemma D.2 and Corollary D.7]{Pinzi-Savare25}; here is a self-contained discussion.

    Let us first consider for a bounded Borel real function
    $z:\rmH\times \cP(\rmH)\to\R$
    the functional $\cZ:\cP(\rmH)\to\R$,
    \begin{equation}
      \label{eq:97}
      \cZ(\mu):=\int z(x,\mu)\,\d\mu\quad \mu\in \cP(\rmH),
    \end{equation}
    and let us call $\BB$ the class of functions $z$
    for which $\cZ$ is Borel.

    Clearly $\BB$ contains all the bounded and continuous functions $z\in
    \rmC_b(\rmH\times \cP(\rmH))$ (for which $\cZ$ is continuous).
    It is also easy to check that $\BB$ is closed with respect to
    uniform and monotone limits. By the functional monotone class
    Theorem
    \cite[Theorem 2.12.9]{Bogachev07} we deduce that
    $\BB$ contains all the bounded Borel functions.

    It follows that for every bounded continuous (or even Borel) map
    $\zeta:\rmH\times\rmH\to \R$ the map
    \begin{equation}
      \label{eq:98}
      \mu\mapsto
      \int_{\rmH\times \rmH}\zeta\,\d\GG(\mu)=
      \int \zeta(x,\ff(x,\mu))\,\d\mu
    \end{equation}
    is Borel, so the map $\GG$ is Borel as well, since
    the functionals $\ggamma\mapsto \int \zeta,\d\ggamma$,
    $\zeta\in \rmC_b(\rmH\times \rmH)$ generates
    the weak (Polish) topology of $\cP(\rmH\times \rmH)$. 

    Now, it is immediate to conclude that $\bttP$ is fully deterministic, since $\GG$ is injective and maps $\cP_2(\rmH)$ to $\cPdet2(\rmH\times \rmH)$.
  \end{proof}
  We end up with the following stronger Monge formulation of
  OT problem between (laws of) random measures.  
\begin{problem}[Strict Monge formulation]
  \label{prob:Monge}
  Given $\ttM_1,\ttM_2\in \rW2$ find
  a Borel map $\ff\in L^2(\bar \ttM_1;\rmH)$ such that
  setting
  $\FF(\mu):=
  \ff(\cdot,\mu)_\sharp\mu$ we have
  \begin{equation}
    \label{eq:90}
   \FF_\sharp \ttM_1=\ttM_2\quad 
    \text{and}
  \quad \sfW_2^2(\ttM_1,\ttM_2)=
    \int_{\rmH\times \rW2}\big|\ff(x,\mu)-x\big|^2\,\d\bar \ttM_1(x,\mu).
  \end{equation}
\end{problem}
Notice that the transformation 
$\ff\times \FF: (x,\mu)\to 
\big(\ff(x,\mu),\ff(\cdot,\mu)_\sharp\mu\big)$
maps $\rmH\times \cP_2(\rmH)$ to itself and satisfies
\begin{equation}
  \label{eq:99}
  (\ff,\FF)_\sharp\bar \ttM_1=\bar \ttM_2
\end{equation}
where, as usual, $\FF(\mu)=\ff(\cdot,\mu)_\sharp\mu$ and
$\bar \ttM_i=\int \mu\otimes \delta_\mu\,\d \ttM_i(\mu)$.

It could seem that the strict Monge formulation is considerably more demanding than the usual Monge formulation expressed by \eqref{eq:88}. For example
in the extreme case when $\ttM_i=
\delta_{\mu_i}$ are Dirac masses
concentrated in two measures
$\mu_1,\mu_2\in \cP_2(\rmH)$
clearly there is just one 
solution to the Monge formulation
\eqref{eq:88} but 
there could be many (or even no) solutions
to the strict Monge formulation, which 
reduces to the usual
Optimal Transport problem
between $\mu_1$ and $\mu_2$ in 
$\cP_2(\rmH).$

However, when we look for 
conditions on $\ttM_1$ 
which guarantee that 
\eqref{eq:88} is solvable
\emph{for every target measures $\ttM_2$}
then the two formulations are equivalent
and force uniqueness of solutions,
as the following result shows.
\begin{theorem}[Monge vs strict Monge]
    \label{le:M=strict-M}
    Given 
    $\ttM_1\in \cPP_2(\rmH)$,
    the following two properties are equivalent:
    \begin{enumerate}
        \item 
        for every $\ttM_2\in \cPP_2(\rmH)$
    there exists a 
    Borel map $\FF=\FF_{\ttM_2}:\cP_2(\rmH)\to\cP_2(\rmH)$
    (depending on $\ttM_2$) solving
    the OT problem in Monge form
    \eqref{eq:88};
    \item 
    for every $\ttM_2\in \cPP_2(\rmH)$
    there exists a 
    Borel map $\ff =\ff_{\ttM_2}\in L^2(\bar\ttM_1;\rmH)$ 
    (depending on $\ttM_2$)
    solving 
    the OT problem in the strict Monge form
    \eqref{eq:90}.
     \end{enumerate}
    In both cases, 
    for every $\ttM_2\in \cPP_2(\rmH)$ 
    the set of optimal couplings 
    $\Gamma_o(\ttM_1,\ttM_2)$ 
    contains the unique element
    $\pPi=\pPi_{\ttM_2}=(\id\times \FF_{\ttM_2})_\sharp \ttM_1$ 
    and, for $\pPi$-a.e.~$(\mu_1,\mu_2)$, the set
    $\Gamma_o(\mu_1,\mu_2)$
    contains the unique deterministic coupling
    $\ggamma=(\ii\times \ff_{\ttM_2}(\cdot,\mu_1))_\sharp\mu_1$
    so that $\ff$ corresponds to 
    the (unique) solution of the 
    strict Monge formulation 
    given in Problem \ref{prob:Monge}.
    We also have $\ff=\nabla_W\phi$ 
    for every optimal Kantorovich potential $\phi$
    (recall \eqref{eq:W-gradient}).

    Finally, every Lipschitz totally convex
    function $\phi$ is $W$-differentiable
    at $\ttM_1$-a.e.~$\mu$, according
    to \eqref{eq:defW}.
\end{theorem}
\begin{proof}
    Since 2.$\Rightarrow$1., it is sufficient to prove that 1.~implies 2.~and all the further
    properties stated by the Theorem.

    We thus fix $\ttM_2$,
    an optimal Kantorovich potential $\phi$ 
    for the pair $\ttM_1,\ttM_2$
    with $\brmF=\bpartialt\phi$
    and we denote by $\ff^\circ:\rmH\times \cP_2(\rmH)
    \to\rmH$ a Borel version of the minimal section
    of $\brmF$ (as in Proposition 
    \ref{prop:tot-max-mon})
    and we 
    set 
    $\FF^\circ(\mu):=\ff^\circ(\cdot,\mu)_\sharp\mu $.
    
    We then select
    an optimal random coupling law 
    $\bttP\in \RGamma_o(\ttM_1,\ttM_2)$
    with
    $\pPi=(\pi^1_\sharp,\pi^2_\sharp)_\sharp
    \bttP\in \Gamma_o(\ttM_1,\ttM_2)$;
    since $\bttP$ is optimal, 
     $\supp\bttP\subset \brmF$.
        
    By the minimality of $\ff^\circ$ 
    we have
    $\ff^\circ\in L^2(\bar\ttM_1;\rmH)$ since
    \begin{align*}
        \sfW_2^2(\ttM_1,\ttM_2)&=
        \int\int |x_2-x_1|^2\,\d\ggamma(x_1,x_2)
        \,\d\bttP(\ggamma)
        \ge 
        \int\int |\ff^\circ(x_1,\pi^1_\sharp\ggamma)-x_1|^2\,\d\ggamma(x_1,x_2)
        \,\d\bttP(\ggamma)
        \\&=
        \int\int |\ff^\circ(x_1,\mu)-x_1|^2\,\d\bar\ttM_1(x,\mu)
    \end{align*}

    We then introduce
    \begin{displaymath}
        \bttP':=\frac 12 \bttP+
        \frac 12 \GG^\circ_\sharp\ttM_1,
        \quad 
        \GG^\circ(\mu):=
        (\ii\times \ff^\circ(\cdot,\mu))_\sharp\mu,\quad 
        \ttM_2':=\pi^2_{\sharp\sharp}\bttP'.
    \end{displaymath}
    Since $\supp\bttP'\subset \brmF$
    we have
    $\bttP'\in \RGamma_o(\ttM_1,\ttM_2')$
    and
    its disintegration with respect to $\pi^1_\sharp$ 
    can be expressed via the corresponding
    disintegration
    $\bttP_\mu$ of $\bttP$ by
    \begin{displaymath}
        \bttP_\mu'=\frac12 \bttP_\mu
        +\frac 12 \delta_{\GG^\circ(\mu)}.
    \end{displaymath}
    We then select the midpoint
    $\tbar\ttM:=(\pi^{1\shto2}_t)_{\sharp\sharp}\bttP'$ between 
    $\bttM_1$ and $\bttM_2'$ 
    induced by $\bttP'$ corresponding to $t=1/2.$
    By Theorem 
    \ref{thm:geo1} 
    we know that 
    $\tbar\bttP:=(\pi^1,\pi^{1\shto2}_t)_{\sharp\sharp}
    \bttP'$ is the unique element
    of $\RGamma_o(\ttM_1,\tbar\ttM)$,
    and  
    $\tbar\pPi=(\pi^1_\sharp,\pi^2_\sharp)\tbar\bttP$
    is the unique element
    of $\Gamma_o(\ttM_1,\tbar\ttM)$
    so that 
    by assumption 
    $\tbar\pPi=(\id\times \tbar\FF)_\sharp \ttM_1$
    for $\tbar\FF=\FF_{\tbar\ttM}.$
    On the other hand
    setting
    \begin{displaymath}
        \bttP_{\mu,t}=
        (\pi^1,\pi^{1\shto2}_t)_{\sharp\sharp}
        \bttP_\mu,\
        \pPi_{\mu,t}=\pi^2_{\sharp\sharp}
        \bttP_{\mu,t},\
        \GG^\circ_t(\mu)=
        (\pi^1,\pi^{1\shto2}_t)_\sharp\GG^\circ(\mu),
        \
        \FF_t^\circ(\mu)=(\pi^{1\shto2}_t)_\sharp\GG^\circ(\mu)=
        \pi^2_\sharp \GG^\circ_t(\mu)        
    \end{displaymath}
    we have 
    \begin{displaymath}
        \tbar{\bttP}_\mu=
        \frac 12 \bttP_{\mu,t}+
        \frac 12 \delta_{\GG^\circ_t(\mu)},\quad 
        \tbar{\pPi}_\mu=
        \frac 12 \pPi_{\mu,t}+
        \frac12\delta_{\FF^\circ_t(\mu)}
        \quad\text{$\ttM_1$-a.e.}
    \end{displaymath}
    Since $\tbar\pPi_\mu=\delta_{\tbar\FF(\mu)}$,
    we deduce that $
    \pPi_{\mu,t}=\FF^\circ_t(\mu)$
    for $\ttM_1$-a.e.~$\mu$,
    so that $\tbar\ttM$ is the middle point
    also between $\ttM_1$ and $\ttM_2$
    and between $\ttM_1$ and 
    $\ttM_2^\circ=(\FF^\circ)_\sharp \ttM_1$
    (with respect to the same 
    optimal set $\brmF$).
    By the non-branching property 
    of Theorem \ref{thm:geo1}
    we deduce that $\ttM_2=\ttM_2^\circ$
    and by the strict minimality of the section $\ff^\circ$ 
    $\bttP=\GG^\circ_\sharp \ttM_1$
    is a strict Monge solution
    and is the unique element of $\RGamma_o(\ttM_1,\ttM_2)$;
    similarly $\pPi=(\id\times \FF^\circ)_\sharp\ttM_1$
    is deterministic and is 
    the unique element of $\Gamma_o(\ttM_1,\ttM_2).$

    Let us eventually check the last statement.
    We take a Lipschitz totally convex function
    $\phi$: if there exists a Borel 
    set $B\subset \cP_2(\rmH)$
    where $\bpartialt\phi$ is not a singleton
    with $\ttM_1(B)>0$,
    by a standard measurable selection
    we can construct
    an optimal random coupling law
    $\bttP$ with first random marginal $\ttM_1$
    and second random marginal $\ttM_2=
    \pi^2_{\sharp\sharp}\bttP\in \cPP_2(\rmH)$
    (thanks to the fact that $\phi$ is Lipschitz)
    which is not concentrated on
    the map $\nabla_W\phi$, contradicting the 
    above result.
\end{proof}
By the above result if we want to 
solve the Monge problem 
for arbitrary target $\ttM_2$ it seems natural
to start from measures
$\ttM_1$ concentrated on 
$\cP_2^r(\rmH)$. 
The next simpler Proposition shows that 
in this case there is also 
a one-to-one correspondence between
optimal couplings of (laws of) random measures and 
(laws of) random optimal couplings and 
for every fixed target $\ttM_2$
  the Monge and the strict Monge formulations of the $L^2$-Optimal Transport problem are equivalent
  as well.  
\begin{proposition}
  \label{prop:correspondence}
  Let $\ttM_1,\ttM_2\in \RW2$ and let us assume that
  $\ttM_1$ is concentrated on $\cP_2^r(\rmH)$, i.e.~$\mu\in \cP_2^r(\rmH)$ for
  $\ttM_1$-a.e.~$\mu$.
  Then the restriction of the map $(\pi^1_\sharp,\pi^2_\sharp)_\sharp$
  on $\RGamma_o(\ttM_1,\ttM_2)$ is injective and every
  $\bttP\in \RGamma_o(\ttM_1,\ttM_2)$ is concentrated on
  $\cP_2^{\rm det}(\rmH\times \rmH)$.
  Moreover, if
  $\bttP\in \RGamma_o(\ttM_1,\ttM_2)$ and 
  $\text{$\mathit \Pi$}=(\pi^1_\sharp,\pi^2_\sharp)_\sharp
    \bttP$
  is deterministic then $\bttP$ is fully deterministic.  
\end{proposition}
\begin{proof}
  Since $\bttP\in \RGamma_o(\ttM_1,\ttM_2)$ and
  $\ttM_1$ is concentrated on $\cP_2^r(\rmH)$, thanks to Theorem \ref{thm:Brenier}
  we can find a Borel set $B\subset\cP_2(\rmH\times \rmH)$
  of full $\bttP$-measure
  such that 
  $B\subset \cP_{2,o}^{\operatorname{det}}(\rmH\times \rmH)$
  and $\pi^1_\sharp(B)\subset \cP_2^r(\rmH)$
  (when $\rmH$ has finite dimension
  we just take $B=(\pi^1_\sharp)^{-1}\big(\cP_2^r(\rmH)\big)\cap \cP_{2,o}(\rmH \times \rmH)$, 
see Proposition \ref{prop:measurability}).

  The restriction of $(\pi^1_\sharp,\pi^2_\sharp)$ to $B$
  is injective, since,
  given $(\mu_1,\mu_2)\in (\pi^1_\sharp,\pi^2_\sharp)(B)$, the set
  $\Gamma_o(\mu_1,\mu_2)$ contains a unique element and it is deterministic.
  It follows that $(\pi^1_\sharp,\pi^2_\sharp)_\sharp$ is injective as
  well
  on $\RGamma_o(\ttM_1,\ttM_2)$.
  If moreover $\pPi$ is deterministic and induced by the Borel map
  $\FF$, we see that
  for $\bttP$-a.e.~$\ggamma$, 
  $\ggamma\in \Gamma_o(\pi^1_\sharp \ggamma,\FF(\pi^1_\sharp\ggamma))$
  so that $\pi^1_\sharp$ is $\bttP$-essentially injective.
  \end{proof}

\subsection{Regular and super-regular measures in \texorpdfstring{$\RW2$}{} and solution
  to the Monge problem}
  \label{subsec:regular-measures}
We first observe that the definitions of $\sigma$-\textrm{d.c.}~hypersurfaces and Gaussian null sets given in
Section \ref{subsec:regularity} also apply to the infinite dimensional Hilbert
space $\cH$;
we keep the notation $\cP_2^r(\cH),\cP_2^{{gr}}(\cH)$ 
to denote the corresponding class of regular measures,
thus vanishing on all \textrm{d.c.}~hypersurfaces and Gaussian-null Borel subsets of $\cH$ respectively.
\begin{definition}[Random exceptional and Gaussian null sets, regular and super-regular measures]
  \label{Random Gauss-null sets}
  \ 
  \begin{itemize}
      \item[-]
  We say that a Borel set $B\subset \cP_2(\rmH)$
  is a random exceptional (resp.~G-null) set if
  $\iota^{-1}(B)$ is 
  contained in a 
  $\sigma$-d.c.~hypersurface 
  of $\cH$ (resp.~Gaussian-null in $\cH$).
  \item[-] 
  We denote by $\cP_2^r(\cP_2(\rmH))$ 
  (resp.~$\cP_2^{gr}(\cP_2(\rmH))$) the set of regular (resp.~G-regular) measures
  $\ttM\in \cP_2(\rmH)$
  such that 
  $\ttM(B)=0$ for every random
  exceptional (resp.~G-null) Borel set
  $B\subset\cP_2(\rmH)$.
  \item[-]
  The set of super-regular measures 
  $\RWr2=\cP_2^r(\cP_2^r(\rmH))$ 
  (resp.~super-G-regular measures
  $\RWgr2=
  \cP_2^{gr}(\cP_2^{r}(\rmH))$)
  is the set of regular 
  (resp.~of G-regular) measures
  concentrated on $\cP_2^r(\rmH)$.
  \end{itemize}
\end{definition}
\noindent 
It is clear that 
\begin{equation}
    \label{eq:inclusions2}
    \cP_2^{gr}(\cP_2(\rmH))
    \subset \cP_2^r(\cP_2(\rmH),
    \qquad 
    \RWgr2\subset 
    \RWr2.
\end{equation}
Let us make a few comments 
on the previous definitions.
\begin{remark}[LGGRM measures]
    \label{rem:rGnull}
    We can say that 
    a measure $\ttG\in \RW2$
    is a Law of Gaussian Generated Random Measures (LGGRM) if
    $\ttG=\iota_\sharp \frg$ 
    for some nondegenerate Gaussian measure $\frg$ in $\cH$. 
    An equivalent way to say that 
    a Borel set in $\cP_2(\rmH)$ 
    is a random G-null set is
    \begin{equation}
        \label{eq:equivalent}
        \ttG(B)=0\quad\text{for every LGGRM $\ttG.$}
    \end{equation}
\end{remark}
\begin{remark}[Random exceptional and G-null sets
  are independent of the choice of $(\OOmega,\cF_\OOmega,\P)$]
  \label{rem:invariance}
  \ \upshape\\
  It is not difficult to see that the above definitions of random exceptional and 
  G-null sets 
  (and the corresponding classes of regular and super-regular measures) are independent of the choice of
  the nonatomic standard Borel space
  $(\OOmega,\cF_\OOmega,\P)$.
  In fact, if $(\OOmega',\cF_\OOmega',\P')$ is another standard Borel measure space
  endowed with a nonatomic measure $\P'$, we can
  find a measure preserving isomorphism
  $\mathtt h:\OOmega'\to \OOmega$ such that
  $\mathtt h_\sharp\P'=\P$ and $\mathtt h^{-1}_\sharp\P=\P'$.
  $\mathtt h$ induces a linear isometry of
  $\cH$ onto $\cH'=L^2(\OOmega',\P';\rmH)$
  defined by $\mathtt h^* X:=X\circ\mathtt h$, with
  $\iota'\circ\mathtt h^*=\iota$, since
  \begin{equation}
    \label{eq:82}
    X'=X\circ \mathtt h,\quad
    \iota'(X')=(X\circ\mathtt h)_\sharp\P'=
    X_\sharp\mathtt h_\sharp\P'=
    X_\sharp\P=\iota(X).
  \end{equation} 
  Since the \textrm{d.c.}~hypersurfaces are 
  preserved by isometric isomorphisms between Hilbert spaces, using 
  \eqref{eq:82} it is immediate to check that a random exceptional set w.r.t.~$\OOmega'$
  is also exceptional w.r.t.~$\OOmega$.

  In a similar way, if $B'$ is a random Gaussian-null set
  with respect to $\OOmega',\P'$,
  i.e.~$(\iota')^{-1}(B')$ is Gaussian-null in $\cH'$
  and let $\frg$ be an arbitrary nondegenerate Gaussian measure in
  $\cH$.
  We introduce $\frg'=(\mathtt
  h^*)_\sharp\frg$ and we observe that
  $\frg'$ is a nondegenerate Gaussian measure in $\cH'$
  (since $\mathtt h^*$ is a linear surjective isometry).
  By definition
  $\frg'((\iota')^{-1}B')=0$
  and therefore
  \begin{displaymath}
    \frg(\iota^{-1}B)=
    \frg\Big(\big(\iota'\circ\mathtt h^*\big)^{-1}B\Big)=
    \frg\Big((\mathtt h^*)^{-1}(\iota')^{-1}B\Big)=
    \frg'\Big((\iota')^{-1}B\Big)=0,
  \end{displaymath}
  so that $\iota^{-1}B$ is Gaussian null in $\cH$. 
\end{remark}
\begin{remark}[Stability in the class of mutually absolutely continuous measures]
\label{rem:stab-ac}
    The super-regularity condition is stable
with respect to multiplication by an integrable factor:
\begin{equation}
    \label{eq:invariance-super}
    \ttM\in \RWr2,\quad
    \ttM'\ll \ttM
    \quad\Rightarrow\quad
    \ttM'\in \RWr2.
\end{equation}
\end{remark}
\begin{remark}
    \label{rem:simpler}
    If $\ttM$ is concentrated 
    on $\cP_2^r(\rmH)$
    then it is super-regular
    if it vanishes on 
    all exceptional subsets of $\cP_2^r(\rmH)$,
    i.e.~it is sufficient to check that 
    \begin{equation}
        \label{eq:simpler}
        B\subset \cP_2^r(\rmH),\quad 
        \iota^{-1}(B)
        \text{ exceptional in $\cH$}
        \quad\Rightarrow\quad 
        \ttM(B)=0.
    \end{equation}
    Similarly, if 
    \begin{equation}
        \label{eq:simpler2}
        B\subset \cP_2^{gr}(\rmH),\quad 
        \iota^{-1}(B)
        \text{ Gaussian null in $\cH$}
        \quad\Rightarrow\quad 
        \ttM(B)=0
    \end{equation}
    then $\ttM$ is super-G-regular.
\end{remark}
There is a simple way to generate
super-regular measures.
\begin{lemma}
  \label{le:push-regular}
  Let $\frm$ be a regular measure 
  in $\cP_2^r(\cH)$
  (respectively $G$-regular in $\cP_2^{gr}(\cH)$)
  such that 
  \begin{equation}
    \label{eq:54}
    \iota(X)=X_\sharp\P\in \cP_2^r(\rmH)
    \quad\text{for $\frm$-a.e.~}X\in \cH.
  \end{equation}
  Then $\ttM:=\iota_\sharp \frm$ is a super-regular measure
  in $\RWr2$
  (resp.~super-G-regular in $\RWgr2$).
\end{lemma}
\begin{proof}
  It is immediate to see that 
  $\ttM:=\iota_\sharp \frm\in \cP_2^r(\cP_2(\rmH))$ is regular:
  in fact, if $B$ is a random 
  exceptional set
  \begin{displaymath}
    \ttM(B)=\frm(\iota^{-1}B)=0
  \end{displaymath}
  since $\iota^{-1}(B)$ is a $\sigma$-\textrm{d.c.}~hypersurface and $\frm$ is regular.
  Condition \eqref{eq:54} also shows that $M$
  is concentrated on $\cP_2^r(\rmH)$.
\end{proof}
It is also possible to change the reference measure $\P:$ we show two simple cases.
\begin{lemma}
    \label{le:change-P}
    Let $\P',\P''$ be atomless
    Borel probability measures on the 
    standard Borel space $(\OOmega,\cF_\OOmega)$
    with $\P''\le a\P'$ for some $a>0$
    (so that 
    the corresponding 
    Hilbert spaces $\cH',\cH''$
    satisfy $\cH'\subset \cH''$
    with continuous inclusion).
    Denote by $\iota':X\to X_\sharp\P'$,
    $\iota'':X\to X_\sharp\P''$ the corresponding
    law maps.
    
    If $\frm\in \cP_2^r(\cH')$
    and $\ttM'=\iota'_\sharp\frm$
    is super-regular 
    then also $\ttM''=\iota''_\sharp\frm$
    is super-regular.

    A similar result holds if 
    $\ttM''\ll \ttM'$,
    and $\frm\in \cP_2^r(\cB)$, for some separable Banach space $\cB\subset 
    \cH'\cap\cH''.$
\end{lemma}
\begin{proof}
    Let $N\subset \cH'$ 
    be a $\frm$-negligible Borel set
    such that $\iota'(X)\in \cP_2^r(\rmH)$
    for every $X\in \cH'\setminus N.$
    Since $\iota''(X)\le a\iota'(X)$
    for every $X\in \cH'$
    we deduce that 
    $\iota''(X)\in \cP_2^r(\rmH)$
    for every $X\in \cH'\setminus N$
    as well, so that $\ttM''$
    is concentrated in $\cP_2^r(\rmH).$

    If $B$ is a Borel exceptional set 
    of $\cP_2(\rmH)$
    then 
    $(\iota'')^{-1}(B)$
    is contained in a 
    $\sigma$-d.c.~hypersurface $S$ of 
    $\cH''$; since $S\cap\cH'$
    is a $\sigma$-d.c.~hypersurface as well, we
    deduce that 
    $\ttM''(B)\le 
    \frm(S)=\frm(S\cap \cH')=0$
    since $\frm$ is regular. We conclude that 
    $\ttM''$ is super-regular.
    
    A similar argument applies to the second statement.
\end{proof}

The relation between super-regular measures
and differentiability of 
Lipschitz totally displacement convex functions
is clarified by the next two results.
\begin{theorem}
    \label{thm:total-diff}
    If $\phi:\cP_2(\rmH)\to\R$
    is a totally displacement convex and Lipschitz function
    then the singular set  
    $\operatorname{Sing}^r(\phi):=\{\mu\in \cP^r_2(\rmH):
    \#\bpartialt\phi[\mu]>1\}$
    of \emph{regular} measures
    where the total subdifferential of 
    $\phi$ is not reduced to 
    a singleton (the minimal section)
    is exceptional.
\end{theorem}
\begin{proof}
  Since every measure
  $\mu\in \operatorname{Sing}^r(\phi)
  $ is regular and 
  $\bpartialt\phi[\mu]
  \subset \cP_{2,o}(\rmH\times \rmH)$ by Corollary \ref{cor: tot subdiff are opt}, all the elements of
  $\bpartialt\phi[\mu]$ are 
  deterministic couplings.
  If $\bpartial\phi[\mu]$
  contains 
  at least two different elements,
  they are 
  associated with 
  two different fields $\ff_1(\cdot,\mu),\ff_2(\cdot,\mu)$ 
  in $L^2(\rmH,\mu;\rmH)$. 
  
  If $X\in \iota^{-1}(\mu)$, setting 
  $Y_i=\ff_i(X,\mu)$ we deduce that 
  $(X,Y_i)_\sharp\P\in \bpartialt\phi[\mu]  $
  and therefore   
  $Y_i\in \partial \hat\phi(X)$, where $\hat\phi=\phi\circ\iota$.
  We conclude that $\partial\hat\phi(X)$
  contains two different elements and therefore 
  $\hat\phi$ is not Gateaux-differentiable at $X$. This arguments shows that 
  $\iota^{-1}(\operatorname{Sing}^r(\phi))$ is
  a subset where $\hat\phi$ is not Gateaux-differentiable  
  and therefore is a $\sigma$-\textrm{d.c.}~hypersurface in $\cH$ since
  $\hat\phi$ is a convex Lipschitz function.
\end{proof}
Combining Proposition \ref{prop:diff}, Remark \ref{rem:simpler} and the above Theorem we immediately get:
\begin{corollary}
    \label{cor:obvious}
    If $\ttM\in \RWr2$
    and $\phi:\cP_2(\rmH)\to \R$
    is a Lipschitz 
    totally displacement convex function,
    then 
    for $\ttM$-a.e.~$\mu$ we have:
    \begin{enumerate}
        \item 
        $\bpartialt\phi[\mu]
    =
    \bpartialt^\circ\phi[\mu]$
    is reduced to a single deterministic coupling in $\cP_{2,o}^{\rm det}(\rmH^2)$
    of the form\\
    $(\ii\times \nabla_W\phi(\cdot,\mu))_\sharp\mu$;
    \item 
    $\partial^-\phi(\mu)=
    \nabla_W\phi(\cdot,\mu)_\sharp\mu$.
    \end{enumerate}
\end{corollary}
\begin{theorem}[Solutions to the strict Monge problem for super-regular measures]
  \label{thm:Monge}
  If $\ttM\in \RWr2$ and
  $\ttN\in  \RW2$, then 
  $\RGamma_o(\ttM ,\ttN )$ and $\Gamma_o(\ttM ,\ttN )$
  contain a unique element
  $\bttP$ and $\pPi=(\pi^1_\sharp,\pi^2_\sharp)_\sharp \bttP$ respectively.\\
  $\bttP$ is fully deterministic and
  there exists a unique
  Borel map $\ff\in L^2(\bar \ttM ;\rmH)$
  solving the Monge OT problem \ref{prob:Monge}.
  $\ff$ is essentially totally cyclically monotone
  and coincides with
  the minimal section
  $\nabla_W\phi$
  of an optimal Kantorovich potential.
\end{theorem}
\begin{proof}
By Proposition \ref{prop:correspondence}
and Theorem \ref{le:M=strict-M}
  it is sufficient to prove that
  $\Gamma_o(\ttM ,\ttN )$ contains a 
  deterministic coupling $\pPi$.

  Let us first suppose that
  $\ttN $ is concentrated on the closed ball 
  \begin{equation}
    \label{eq:102}
    \cB(R):=\Big\{\nu\in \cP_2(\rmH):\sfm_2(\nu)\le R\Big\}.
  \end{equation}
  In this case, by Corollary \ref{cor:Lip} there is an optimal Kantorovich potential
  $\phi$ which is 
  a $R$-Lipschitz totally convex 
  function $\phi$ with  $S:=\supp(\pPi)\subset \partial^-\phi$
  and Corollary \ref{cor:obvious}
  shows that  there exists
  a (unique) strictly Monge solution
  which is given by
  \begin{equation}
    \label{eq:103}
    \ff (\cdot,\mu)=\nabla_W\phi(x,\mu).
  \end{equation}
In the general case, 
we argue as in the proof
  of Theorem 6.2.10 \cite{AGS08}.
For every $R>0$ we set
\begin{displaymath}
    \cB'(R):=\cP_2(\rmH)\times \cB(R)
\end{displaymath}
and for an optimal coupling 
$\pPi\in \Gamma_o(\ttM ,\ttN )$,
  we set for sufficiently large $n$
  \begin{equation}
    \label{eq:104}
    Z_n:=\pPi\big(\cB'(n)\big),\quad
    \pPi_n:=Z_n^{-1}\pPi\mres\cB'(n),
    \quad
    \ttM_n:=\pi^1_\sharp\pPi_n,\ \ttN _n:=
  \pi^2_\sharp \pPi_n. 
  \end{equation}
    Since optimality is preserved by
    restriction, 
    $\pPi_n\in \Gamma_o(\ttM_n,\ttN _n)$;
    since $\ttM_n\ll \ttM $ and
    $\ttN _n$ is concentrated on $\cB(n)$
    the previous argument shows that
    there exists a unique map 
    $\FF_n$ such that 
    $\pPi_n=(\id\times \FF_n)_\sharp \ttM_n.$

    Moreover, for $m>n$ we easily get
    $\FF_m=\FF_n$ $\ttM_n$-a.e., so that 
    there exists a map $\FF$ 
    such that $\FF=\FF_n$ $\ttM_n$-a.e.~for every $n\in \N$.
    Passing to the limit in the identity
    $\pPi_n=(\id\times \FF)_\sharp \ttM_n$
    we obtain $\pPi=(\id\times \FF)_\sharp \ttM .$
    It follows that $\pPi$ is deterministic.
\end{proof}

Anticipating 
some of the results of the next section, it is easy to see that 
$\RWr2$ is dense in $\RW2$.
\begin{proposition}
  \label{prop:density}
  If $\rmH$ has finite dimension then $\RWgr2$ is dense in $\RW2$.
  In particular, the class of initial measures for which the OT
  problem
  has a unique solution in (strict) Monge form is dense.
\end{proposition}
\begin{proof}
Thanks to 
\eqref{eq:invariance-super}, 
if $\ttM\in \RWgr2$ and $\ttM'\ll \ttM$ then also
$\ttM'\in \RWgr2$.
If $\rmH$ has finite dimension, we 
will see in the next section
(see Theorems \ref{thm:1d-gaussians-ok}
and 
\ref{thm:gaussian-C1-d})
that there exists a
reference measure $\ttG\in \RWgr2$ with full support.
It is then sufficient to observe that
the set 
\begin{displaymath}
  \Big\{\ttM'\in \RW2:\ttM'\ll \ttG\Big\}\quad\text{is dense in }\RW2,
\end{displaymath}
since its closure contains all the 
finite combination of Dirac masses in $\RW2$.
In fact, if 
$\ttM=\sum_{k=1}^n a_k\delta_{\mu_k}$
for distinct points $\mu_1,\cdots,\mu_n
\in \cP_2(\rmH)$,
we can choose $r>0$ so small that 
the balls $B_{r,k}=B_r(\mu_k)$ are disjoint.
Since $\ttG$ has full support, 
$Z_{r,k}:=G(B_{r,k})>0$ so that 
\begin{displaymath}
    \ttM_{r}:=
    \Big(\sum_{k=1}^n
    \frac {a_k}{Z_{r,k}}\nchi_{B_{r,k}}\Big)G
    \in \RWr2
\end{displaymath}
and $\ttM_r\to \ttM$ in $\RW2$ as $r\downarrow0.$
\end{proof}

\subsection{Examples of super-regular measures
for finite dimensional \texorpdfstring{$\rmH$}{}}
\label{subsec:examples}
In this last section we will 
exhibit many examples of super-regular measures.
We will focus on the relevant case of 
measures induced by nondegenerate Gaussian measures $\frg$ on $\cH$
as in Example \ref{ex:gaussian},
when $\rmH$ has finite dimension: by using
Lemma \ref{le:push-regular} it will be sufficient
to check that \eqref{eq:54} holds, i.e.~$\iota_\sharp\frg(\cP_2(\rmH)\setminus \cP_2^r(\rmH))=0.$

We start from the $1$-dimensional case, where
we will prove a very general result.
\subsubsection*{The \texorpdfstring{$1$}{}-dimensional case \texorpdfstring{$\rmH=\R$}{}}

\begin{example}[The random occupation measure
associated with Brownian motion in $[0,1{]}$]
\label{ex:Brownian}
Let $\OOmega=[0,1]$ endowed with the usual Lebesgue measure $\P$
and let $\frw$ be the standard Wiener measure concentrated on
$\rmC([0,1])$. We claim that
\begin{quote}
    \em 
    $W:=\iota_\sharp \frw$ is super-G-regular with full support.
\end{quote}
We can apply Lemma \ref{le:push-regular}.
Since $\frw$ is a Gaussian non-degenerate measure in $\cH=
L^2([0,1],\P)$, it is clearly a regular measure
in $\cP_2^r(\cH)$, so it is sufficient to check that
\begin{equation}
  \label{eq:105}
  X_\sharp\P\text{ is nonatomic for $\frw$-a.e.~path $X\in \cH$},
\end{equation}
i.e.
\begin{equation}
  \label{eq:106}
  \P\Big(\Big\{t\in [0,1]:X_t=y\Big\}\Big)=0
  \quad\text{for $\frw$-a.e.~$X\in \rmC([0,1])$}.
\end{equation}
In fact we have the much stronger result that
the so-called occupation measure $\iota(X)=X_\sharp\P$ 
is absolutely continuous with respect to the Lebesgue measure
in $\R$, see e.g.~\cite[Theorem 3.26]{Morters-Peres10}.
\end{example}
The previous example
is in fact a particular case of a general result: we will show that 
\emph{any} measure $G=\iota_\sharp\frg$
obtained as the push forward of
an arbitrary nondegenerate Gaussian
measure $\frg$ in $\cH$ 
is super-G-regular.

Before discussing this result, 
let us show a simple criterion
ensuring that a measure
$\ttM\in \cPP_2(\R)$ is concentrated on $\cP_2^r(\R).$ We set
\begin{equation}
    \label{eq:nchiD}
    \nchi_0(r):=\begin{cases}
        1&\text{if }r=0,\\
        0&\text{if }r\neq 0.
    \end{cases}
\end{equation}
\begin{lemma}
    \label{le:criterion}
    Let $\ttM=\iota_\sharp\frm$
    for $\frm\in \cP_2(\cH)$
    and let us suppose 
    that $\frm$ 
    admits the representation
    \eqref{eq:representation},
    \eqref{eq:joint}
    discussed in Section
    \ref{subsec:L-representation}.
    If
    \begin{equation}
        \label{eq:criterion-Xi}
        \int_\Thetao
        \nchi_0(\Xi(\thetao,\oomega_1)
        -\Xi(\thetao,\oomega_2))\,
        \d\Q(\thetao)=0
        \quad
        \text{for $\P\otimes \P$-a.e.~$(\oomega_1,\oomega_2)$
        }
    \end{equation}
    then $\ttM$ is concentrated on $\cP_2^r(\R).$
\end{lemma}
\begin{proof}
Let 
$D:=\{(x,x):x\in\R\}$ be the diagonal
in $\R^2$, whose characteristic
function is given by
$(x,y)\mapsto \nchi_0(x-y)$.
We recall that 
a measure $\mu\in \cP_2(\R)$
is atomless (and thus belongs to 
$\cP_2^r(\R)$) if and only if 
\begin{equation}
    \label{eq:atomless2}
    \mu\otimes \mu(D)
    = 
    \int_{\R^2} \nchi_0(x-y)
    \,\d\mu(x)\,\d\mu(y)=0
    .
\end{equation}
Recalling the definition 
\eqref{eq:k-proj} of
the $k$-projection of $\ttM$, 
we deduce that $\ttM$ is concentrated 
on $\cP_2^r(\R)$
if and only if 
    \begin{displaymath}
        \operatorname{pr}^2[\ttM](D)=
        \int \nchi_0(x-y)
        \,\d\operatorname{pr}^2[\ttM](x,y)=0.
    \end{displaymath}
Applying formula 
\eqref{eq:k-integrals}
we thus 
express the above integral as
    \begin{equation}
        \label{eq:condition-rbis}
        \int_\Thetao 
        \Big(\int \nchi_0(\Xi(\thetao,\oomega_1)-
    \Xi(\thetao,\oomega_2))\,\d \P^{\otimes 2}(\oomega_1,\oomega_2)\Big)\,\d \Q(\thetao)=0.
    \end{equation}
    An application of Fubini's Theorem
    yields \eqref{eq:criterion-Xi}.
\end{proof}
As an application of the
above Lemma we have the following general result.
\begin{theorem}[Push forward of 
nondegenerate Gaussian measures are superregular]
    \label{thm:1d-gaussians-ok}
    If $\rmH=\R$ and 
    $\frg$ is a nondegenerate
    Gaussian measure in $\cH$,
    then $\ttG:=\iota_\sharp \frg
    \in \RWgr2$
    is super-G-regular.
\end{theorem}
\begin{proof}
Since $\OOmega$ is a standard Borel space, we can find
a bounded metric $\sfd_\OOmega$ 
   in $\OOmega$ such that 
   $(\OOmega,\sfd_\OOmega)$ is
   a complete and separable 
   metric space  and
   $\cF_\OOmega$ coincides with the Borel $\sigma$-algebra induced by $\sfd_\OOmega.$
   We can also define
   $\sfd_{\OOmega^2}
   ((\oomega_1,\oomega_2),
   (\oomega_1',\oomega_2')):=
   \max 
   [\sfd_\OOmega(\oomega_1,\oomega_1'),
   \sfd_\OOmega(\oomega_2,\oomega_2')]
   $
   and the swap isometric map
   $S:\OOmega^2\to\OOmega^2$,
   $S(\oomega_1,\oomega_2):=
   (\oomega_2,\oomega_1).$
   
We adopt the notation of Example
\eqref{ex:gaussian}. Since $\frg$ is
nondegenerate, we can also assume that 
$\lambda_n>0$ for every $n\in \N_+.$
By Lemma 
    \ref{le:push-regular},  it is sufficient to prove that $\frg$ satisfies condition
    \eqref{eq:54}.
    Using the 
    representation 
    \eqref{eq:G-representation}
    and \eqref{eq:joint-expansion}, we can then apply Lemma 
    \eqref{le:criterion}:
    our thesis follows if we 
    prove \eqref{eq:criterion-Xi}.

    Thanks to 
    \eqref{eq:joint-expansion},
    we have
    \begin{equation}
        \label{eq:exp-2}
       D(\thetao;\oomega_1,\oomega_2):= \Xi(\thetao,\oomega_1)-
    \Xi(\thetao,\oomega_2)=
    \sum_n \xi_n(\thetao)
    \Big(\sfE_n(\oomega_1)-\sfE_n(\oomega_2)\Big)
    \end{equation}
    which is a series of independent random Gaussian variables.
    We know that 
    for $\P^{\otimes 2}$-a.e.~$(\oomega_1,\oomega_2)$
    \begin{itemize}
        \item 
    the series defining $D$
    converges in $L^2(\Thetao,\Q)$
    and also $\Q$-a.e.,
    \item its law
    $\nu_{\oomega_1,\oomega_2}:=
    D(\cdot;\oomega_1,\oomega_2)_\sharp\Q$
    is a Gaussian measure
    \item 
    $\nu_{\oomega_1,\oomega_2}=
    N(0,\lambda^2(\oomega_1,\oomega_2))$
    where 
    \begin{equation}
        \label{eq:variance}
        \lambda^2(\oomega_1,\oomega_2)
        =
        \sum_n 
        \lambda_n^2
        \Big(\sfE_n(\oomega_1)-\sfE_n(\oomega_2)\Big)^2.
    \end{equation}
    \end{itemize}
   We can observe that 
   the integral in \eqref{eq:criterion-Xi} is just
   $\nu_{\oomega_1,\oomega_2}(\{0\})$,
   so that 
   \eqref{eq:criterion-Xi} holds
   if $\nu_{\oomega_1,\oomega_2}$
   is non-degenerate, i.e.
   \begin{equation}
       \label{eq:non-degenerate}
      \lambda^2(\oomega_1,\oomega_2)
        >0
        \quad
        \text{
        for $\P^{\otimes 2}$-a.e.~$(\oomega_1,\oomega_2)$.
        }
   \end{equation}
   Let us denote by 
   $A\subset \OOmega^2$
   the set 
   where $\lambda$ vanishes
   and let $D_\OOmega
   :=\{(\oomega,\oomega):\oomega\in \OOmega\}$ be the diagonal in 
   $\OOmega^2.$
   Since $\P$ is diffuse,
   $\P^{\otimes 2}(D_\OOmega)=0$
   so that 
   we have to prove that 
   $\P^{\otimes 2}(A')=0$ 
   where $A':=A\setminus D_\OOmega$.
   Since $\lambda_n>0$ for every $n,$
   we immediately see that 
   \begin{equation}
       \label{eq:A}
       (\oomega_1,\oomega_2)\in A
       \quad
       \Leftrightarrow\quad
       \sfE_n(\oomega_1)=
       \sfE_n(\oomega_2)
       \quad\text{for every }n\in \N_+.
   \end{equation}
    We argue by contradiction
   and we suppose
   that $\P^{\otimes 2}(A')>0$. 
   We can thus find $(\bar \oomega_1,\bar \oomega_2)
   \in A'\cap \supp(\P)$
   and a sufficiently small ball
   $B=B_r(\bar \oomega_1,\bar \oomega_2)$
   such that 
   $S(B)\cap B=\emptyset$ and $\P^{\otimes 2}(A'\cap B) >0$.
   Since $A'$ is symmetric
   and $S_\sharp \P^{\otimes 2}=
   \P^{\otimes 2}$, 
   we have
   $\P^{\otimes 2}(A'\cap B)=
   \P^{\otimes 2}(A'\cap S(B))>0$.

   We set $B':=A'\cap B$ 
   and we eventually consider the bounded Borel function
   \begin{displaymath}
       f(\oomega_1,\oomega_2):=
       \nchi_{B'}(\oomega_1,\oomega_2)-
       \nchi_{S(B')}
       (\oomega_1,\oomega_2)=
        \nchi_{B'}(\oomega_1,\oomega_2)-
        \nchi_{B'}(\oomega_2,\oomega_1).
   \end{displaymath}
   We can 
   expand $f$ 
   as a orthogonal series 
   in $L^2(\OOmega^2,\P^{\otimes 2})$
   with respect to the 
   complete
   orthonormal system
   $\sfE_{m,n}(\oomega_1,\oomega_2):=
   \sfE_m(\oomega_1)
   \sfE_n(\oomega_2)$
   obtaining
   \begin{equation}
    \label{eq:fourier1}
       f=\sum_{m,n}\hat f_{m,n}\sfE_{m,n}
       \quad\text{converging in }
       L^2(\OOmega^2,\P^{\otimes 2})
   \end{equation}
   where
   \begin{equation}
   \label{eq:fourier2}
       \hat f_{m,n}:=
       \int_{\OOmega^2}
       f(\oomega_1,\oomega_2)
       \sfE_{m}(\oomega_1)
       \sfE_n(\oomega_2)
       \,\d \P^{\otimes 2}(\oomega_1,\oomega_2).
   \end{equation}
Since $f\equiv 0$ if $(\oomega_1,\oomega_2)\not\in A$,
the integral in \eqref{eq:fourier2}
can in fact be restricted to $A$.
Since \eqref{eq:A} implies in particular
\begin{displaymath}
    \sfE_{m,n}(\oomega_1,\oomega_2)=
    \sfE_{m}(\oomega_1)
    \sfE_n(\oomega_2)=
    \sfE_{n}(\oomega_1)
    \sfE_m(\oomega_2)=
    \sfE_{n,m}(\oomega_1,\oomega_2)
    \quad\text{$\P^{\otimes 2}$-a.e.~in $A$}
\end{displaymath}
we immediately get
$\hat f_{m,n}=\hat f_{n,m}.$
On the other hand, 
inverting the order of $\oomega_1,\oomega_2$ in \eqref{eq:fourier2},
using the invariance of 
$\P^{\otimes 2}$
and the anti-symmetry of $f$,
i.e.~$f(\oomega_2,\oomega_1)=
-f(\oomega_1,\oomega_2)$
we also get 
\begin{align*}
    \hat f_{m,n}&=
    \int_{\OOmega^2}
       f(\oomega_2,\oomega_1)
       \sfE_{m}(\oomega_2)
       \sfE_n(\oomega_1)
       \,\d \P^{\otimes 2}(\oomega_1,\oomega_2)
       \\&=
    -   \int_{\OOmega^2}
       f(\oomega_1,\oomega_2)
       \sfE_{m}(\oomega_2)
       \sfE_n(\oomega_1)
       \,\d \P^{\otimes 2}(\oomega_1,\oomega_2)=
       -\hat f_{n,m}
\end{align*}
We deduce that $\hat f_{m,n}=0$
for every pair of indexes,
a contradiction since
$f$ is not identically $0.$
\end{proof}
The range of application of the 
previous Theorem can be considerably extended
thanks to the following 
simple results.
\begin{corollary}
    \label{cor:G-null}
    Let $\cN:=\cP_2(\R)
    \setminus \cP_2^r(\R). $
    Then the (Borel) set $\iota^{-1}(\cN)$
    is Gaussian-null in $\cH.$
\end{corollary}
\begin{proof}
    Theorem 
    \ref{thm:1d-gaussians-ok}
    shows that 
    $\frg(\iota^{-1}(\cN))=0$
    for every non-degenerate Gaussian $\frg$,
    so that $\iota^{-1}(\cN)$
    is Gaussian-null by definition. 
\end{proof}
\begin{corollary}[Push forward of 
G-regular measures are super-G-regular]
    \label{cor:1d-gaussians-ok}
    If $\rmH=\R$ and 
    $\frr$ is a G-regular 
    measure in $\cP_2^{gr}(\cH)$,
    then $\ttR:=\iota_\sharp \frr$
    is super-G-regular.
\end{corollary}
\begin{proof}
    Recall that as a G-regular measure in 
    $\cH$ $\frr$ satisfies
    \begin{equation}
        \frr(B)=0\quad\text{for every 
        Gaussian null Borel set in $\cH$},
    \end{equation}
    in particular $\ff(\iota^{-1}(\cN))=0,$
    so that 
    $\iota(X)\in \cP_2^r(\R)$
    for $\frr$-a.e.~$X\in \cH.$
    We concude by Lemma \ref{le:push-regular}.
\end{proof}

Let us see two simple examples.
\begin{example}[Sum of Gaussians with random signs]
    \label{ex:random-conv}
    Let $\OOmega:=\{-1,+1\}^\N$ be the Cantor set endowed with 
    the uniform product measure
    $\P:=(\frac 12\delta_{-1}+\frac 12\delta_1)^{\otimes \N}$. 
    Every element $\oomega\in \OOmega$
    is a vector $(\oomega_i)_{i\in \N}$
    of signs $\pm1$ indexed by $i\in \N.$
    We denote by $\eps_i:\OOmega\to \R$ the
    $i$-th coordinate,
    and for every finite subset $I\subset \N$ we consider the Walsh function
    \begin{equation}
        \label{eq:Walsh}
        W_I(\oomega):=
        \prod_{i\in I}\eps_i,\quad
        W_\emptyset\equiv 1.
    \end{equation}
    If $\mathcal I$ denotes the (countable) collection of all the finite parts of $\N$, 
    the Walsh system $(W_I)_{I\in \cI}$ 
    is a complete orthonormal system in $\cH=L^2(\OOmega,\P)$. 

    We select a family of independent Gaussian random variables $\xi_I\sim N(0,\lambda_I^2)$
    indexed by $I\in \cI$ 
    and coefficients 
    $\lambda_I>0$ 
    for every $I\in \cI$ 
    such that $\Lambda:=\sum_{I\in \cI}\lambda_I^2<\infty$;
    we form the random vector 
    $\xxi=\sum_{I\in \cI} \xi_I W_I\in \cH$.
    Denoting by $\frg_W$ the law of $\xxi$ in $\cH$
    we obtain 
    a nondegenerate Gaussian measure $\frg_W\in \cP^r_2(\cH)$.
    Applying Theorem
    \ref{thm:1d-gaussians-ok}
    we immediately get
    \begin{quote}
        \em $G_W:=\iota_\sharp \frg_W$
    is super-regular.    
    \end{quote}
\end{example}

\begin{example}[The law of random Fourier series]
    Let us now select 
    $\OOmega:=(0,\pi)$ with
    the (normalized) Lebesgue measure $\P.$
    We consider the complete orthonormal system in $\cH:=L^2(0,\pi)$
    given by the usual Fourier basis
    \begin{equation}
        S_n(\oomega):=\sqrt 2\sin (n\oomega),\quad
        n\in \N_+.
    \end{equation}
    As for the previous example,
    we select a sequence of 
    independent Gaussian random variables
    $\xi_n\sim N(0,\lambda_n^2)$
    and coefficients
    $\lambda_n>0$ for every $n\in \N_+$
    with $\Lambda=\sum_{n=1}^\infty \lambda_n^2<\infty.$
    Denoting by $\frg_F$ the law of the random vector
    $\xxi=\sum_{n=1}^\infty  \xi_n S_n$
    we obtain a nondegenerate Gaussian measure
    $\frg_F\in \cP_2^r(\cH).$
    \begin{quote}
        \em $G_F:=\iota_\sharp \frg_F$
    is super-regular.    
    \end{quote}
    Notice that the case
    $\lambda_n:=\frac 1{\pi n}$
    corresponds to the (centered) Brownian bridge.
\end{example}
\subsubsection*{The finite dimensional case \texorpdfstring{$\rmH=\R^d$}{}, \texorpdfstring{$d>1$}{}}
Let us now discuss 
the case when $\rmH=\R^d$, $d>1$.
We will still focus on the construction
of suitable Gaussian measures $\frg$ on $\cH$
such that $G=\iota_\sharp \frg$ is super-regular
and we will consider two different approaches.
Unlike the previous $1$-d case, we will impose
further properties on $\frg$.

\begin{example}[Gaussian measures concentrated on $\rmC^1$ maps]
    In this first example, we 
    select $\OOmega:=(0,1)^d$
(or any smooth domain in $\R^d$)
endowed with the $d$-dimensional Lebesgue measure $\P$.
\end{example}
\begin{theorem}
    \label{thm:gaussian-C1-d}
    If 
$\frg$ is a non-degenerate Gaussian measure on
the Banach space $\cB=\rmC^1(\overline\OOmega;\R^d)\subset \cH$,
then 
$G=\iota_\sharp\frg$ is super-G-regular with full support.
\end{theorem}
\begin{proof}
    As before, we apply Lemma \ref{le:push-regular}.
Since $\frg$ is a Gaussian non-degenerate measure in $\cB$ and $\cB$ is dense in 
$\cH=
L^2(\OOmega,\P;\R^d)$, 
$\frg$ is clearly a Gaussian nondegenerate measure
in $\cP_2^{gr}(\cH)$, so it is sufficient to check that
\begin{equation}
  \label{eq:105bis}
  X_\sharp\P\text{ is absolutely continuous for $\frg$-a.e. $X\in \cB$}.
\end{equation}
By the area and co-area formulae
(see e.g.~\cite[Thm.~2.3]{Geman-Horowitz80}), 
the push forward $\mu_X=X_\sharp \P$ 
of a map $X\in \rmC^1(\overline\OOmega;\R^d)$
is absolutely continuous if 
\begin{equation}
  \label{eq:109pre}
  \P\Big(\big\{\oomega\in \OOmega:
  \det DX(\oomega)=0\big\}\Big)=0,
\end{equation}
so that 
\eqref{eq:105bis} is true if we show that
\eqref{eq:109pre} holds 
for $\frg$-a.e.~$X\in \cB.$
We can prove this property
by Fubini's Theorem.
We consider the product measure
$\tilde \frg:=\frg\otimes \P$ concentrated on
$\cB\times \OOmega$,
we denote by $\sfD:\cB\times\OOmega\to \R^{d\times d}$
the ``differential'' evaluation map $\sfD(X,\oomega):=
\rmD X(\oomega)$, 
and we introduce the closed set
\begin{equation}
  \label{eq:107}
  A:=\Big\{(X,\oomega)\in \cB\times\OOmega:
  \rmD X(\oomega)\in S\Big\},\quad
  S:=\Big\{D\in \R^{d\times d}:\det D=0\Big\}.
\end{equation}
We know that for every $\oomega\in \OOmega$
the map $X\mapsto \sfD(X,\oomega)$
is linear, continuous, and surjective
from $\cB$ to $\R^{d\times d}$.
We thus deduce that 
$\sfD(\cdot,\oomega)_\sharp \frg$ is a nondegenerate Gaussian
measure in $\R^{d\times d}$, so that
\[\frg\Big(\big\{X\in \cB:\sfD(X,\oomega)\in S\big\}\Big)=0
\quad\text{for every $\oomega\in \OOmega$}.\]
Integrating in $\OOmega$ we get
\begin{equation}
  \label{eq:108}
  \tilde\frg (A)=
  \int \nchi_{A}(X,\oomega)\,\d \tilde\frg=
  \int_{\OOmega} \frg\big(X\in \cB:(X,\oomega)\in A\big)
  \,\d \P(\oomega)=0.
\end{equation}
Applying Fubini's Theorem we thus deduce that
\begin{displaymath}
  \int_{\cB}
  \P\Big(\oomega\in \OOmega:
  (X,\oomega)\in A\Big)\,\d \frg(X)=0
  \end{displaymath}
  which yields \eqref{eq:109pre}
  for $\frg$-a.e.~$X\in \cB$.
\end{proof}
The next example is a natural generalization of 
Example \ref{ex:Brownian}.
\begin{example}[The occupation measure 
of the fractional Brownian motion]
    \label{ex:fractional}
    Let $\OOmega=[0,1]$ endowed with 
    the Lebesgue measure; we fix
    a Hurst parameter $H<1/d$ and we consider
    the $d$-dimensional fractional Brownian
    motion $(\Xi^H_t)_{t\in \OOmega}$
    \cite{BHOZ08}.
    Since $\Xi^H$
    has local time (or, equivalently, its
    occupation measure is absolutely continuous
    with square integrable density
    \cite{Pitt78,Geman-Horowitz80}, 
    \cite[Thm.~10.2.3]{BHOZ08}),
    its law 
    $\frw^H$ in $\cB:=\rmC([0,1];\R^d)$
    is a nondegenerate Gaussian measures
satisfying \eqref{eq:54}, so that 
$\ttW^H
=\iota_\sharp \frw^H\in \cPP_2(\rmH)$
is super-G-regular, according to Lemma 
\ref{le:push-regular}.
\end{example}
The above example is an application of a general technique due to Berman
\cite{berman1969local-f89}
and involving the Fourier transform.
In order to explain the main idea in a general case, 
let us denote by $\mu_X\in \cP_2(\R^d)$ 
    the law $\iota(X)=X_\sharp \P$ of 
    a generic 
    element $X\in \cH$.
    The Fourier transform of $\mu_X$
    is the continuous function 
    $\hat \mu_X:\R^d\to \mathbb{C}$ defined by 
    \begin{equation}
        \hat \mu_X(u):=
        \int_{\R^d}
        \rme^{i\, u\cdot x}
        \,\d\mu_X(x)=
        \int_\OOmega 
        \rme^{i\, u\cdot X(\oomega)}
        \,\d\P(\oomega)
        \quad u\in \R^d.
    \end{equation}
    Notice that $(X,u)\mapsto \hat\mu_X(u)$
    is jointly continuous in $\cH\times \R^d$.
    By Plancherel theorem, 
    $\mu_X$ 
    is absolutely continuous w.r.t.~the $d$-dimensional Lebesgue measure in $\R^d$
    with a density $\varrho_X\in L^2(\R^d)$
    if and only if
    $\hat \mu_X \in L^2(\R^d)$ and moreover
    \begin{equation}
        \label{eq:Plancherel}
        \int_{\R^d}\varrho_X^2\,\d x=
        \frac 1{(2\pi)^d}\int_{\R^d}|\hat \mu_X(u)|^2\,\d u
    \end{equation}
    It follows that if
    the measure $\frg\in \cP_2(\cH)$ satisfies
    \begin{equation}
        \label{eq:condition1}
        L^2:=\int_\cH\Big(
        \int_{\R^d}|\hat \mu_X(u)|^2\,\d u\Big)
        \,\d \frg(X)
        =
        \int_{\R^d}
        \Big(
        \int_\cH|\hat \mu_X(u)|^2
        \,\d \frg(X)
        \Big)
        \,\d u
        <\infty
    \end{equation}
    we get $\mu_X\in \cP_2^r(\R^d)$
    for $\frg$-a.e.~$X.$
    We can rewrite 
    \eqref{eq:condition1}
    by using the representation 
\eqref{eq:representation} by the process $\Xi$,
so that $X(\cdot)=\Xi(\thetao,\cdot)$
and therefore 
we can set $\mu_\thetao=\mu_{\Xi(\thetao,\cdot)}$.
We have 
\begin{align}
        \notag |\hat \mu_\thetao(u)|^2
        &=
        \hat \mu_\thetao(u)\cdot 
        \overline{\hat \mu_\thetao(u)}
        =
        \Big( \int_\OOmega 
        \rme^{i\, u\cdot \Xi(\thetao,\oomega_1)}
        \,\d\P(\oomega_1)\Big)
        \Big( \int_\OOmega 
        \rme^{-i\, u\cdot X(\thetao,\oomega_2)}
        \,\d\P(\oomega_2)\Big)
        \\
        \notag &=
        \int_{\OOmega^2}
        \rme^{i\, u\cdot \Xi(\thetao,\oomega_1)}
        \rme^{-i\, u\cdot \Xi(\thetao,\oomega_2)}
        \,\d\P^{\otimes 2}(\oomega_1,\oomega_2)
        \\&
        \label{eq:dacitare}=
        \int_{\OOmega^2}
        \rme^{i\, u\cdot (\Xi(\thetao,\oomega_1)-\Xi(\thetao,\oomega_2))}
        \,\d\P^{\otimes 2}(\oomega_1,\oomega_2).
\end{align}
Taking now the expectation
w.r.t.~$\Q$ and integrating in $\R^d$ w.r.t.~$u$
we end up with Berman condition 
\cite[Thm.~21.9]{Geman-Horowitz80}
\begin{equation}
    \label{eq:monstre}
     L^2=\int_{\OOmega^2}
    \bigg(
    \int_{\R^d}\mathbb{E}_\Q\Big[
        \rme^{i\, u\cdot (\Xi(\thetao,\oomega_1)-\Xi(\thetao,\oomega_2))}
        \Big]\,\d u\bigg)
        \,\d\P^{\otimes 2}(\oomega_1,\oomega_2)<\infty
\end{equation}
In the particular case 
when $\Xi=(\Xi^1,\cdots,\Xi^d)$ 
is a Gaussian process 
and 
the determinant $\Delta(\oomega_1,\oomega_2)$
of the covariance matrix 
of $\Xi(\cdot,\oomega_1)-
\Xi(\cdot,\oomega_2)$ is positive
for a.e.~$\oomega_1,\oomega_2$, we end up
with the sufficient
condition for the validity 
of \eqref{eq:monstre}
\cite[Thm. 22.1]{Geman-Horowitz80}
\begin{equation}
    \label{eq:sufficient-GH}
    \int_{\OOmega^2}\frac{1}{\Delta(\oomega_1,
    \oomega_2)^d}
    \,\d \P^{\otimes 2}(\oomega_1,\oomega_2)<\infty.
\end{equation}
In the case of the fractional Brownian motion
of Example \ref{ex:fractional} we
thus recover the condition $Hd<1.$

\begin{example}[Berman condition for Karhunen-Lo\`eve expansions]
    \label{ex:Berman-eigenvalue}
We slightly modify 
the above argument,
by considering 
an example inspired to the general framework discussed in Example \ref{ex:gaussian} and 
based on a complete orthonormal system
$\sfE'_n$ of 
the Hilbert space
$\cH':=L^2(\OOmega,\P;\R)$
of \emph{scalar valued} square summable functions.
If $\boldsymbol{e}_1,\cdots \boldsymbol{e}_d$
is an orthogonal basis of $\R^d$ (e.g.~the canonical one), we can then form
the complete orthonormal system 
$\sfE_{n,k}:=\boldsymbol{e}_k \sfE'_n$ of 
$\cH:=L^2(\OOmega,\P;\R^d)$.

We assign 
a sequence $\Sigma_n$, $n\in \N_+$,
of symmetric and positive definite matrices
in $\R^{d\times d}$ satisfying
the boundedness and coercivity condition
\begin{equation}
    \label{eq:uniform-Sigma}
    0<\alpha_n^2\le \Sigma_n \vv\cdot \vv\le \beta_n^2
    \quad \text{for every }
    \vv\in \R^d,\ |\vv|=1;\quad
    \sum_{n=1}^\infty \beta_n^2<\infty.
\end{equation}
We assign a sequence of centered 
 independent Gaussian
random variables in $\R^d$
$\xi_{n}\sim N(0,\Sigma_n)$;
since $\sum_n \beta_n^2<\infty$, 
we can form the random vector 
\begin{equation}
    \label{eq:G-representation-d}
    \xxi=\sum_n 
    \xi_n \sfE'_n
\end{equation}
corresponding to the 
the measurable process 
\begin{displaymath}
    \Xi(\thetao,\oomega)=
    \sum_n 
    \xi_n(\thetao)\sfE'_n(\oomega).
\end{displaymath}
It is clear that $\frg=\xxi_\sharp \Q$
is a nondegenerate Gaussian in $\cH$.
As in \eqref{eq:variance}
we consider the functions
\begin{equation}
    \label{eq:variance2}
    \alpha^2(\oomega_1,\oomega_2)
        =
        \sum_n 
        \alpha_n^2
        \Big(\sfE'_n(\oomega_1)-\sfE'_n(\oomega_2)\Big)^2,\quad
    \beta^2(\oomega_1,\oomega_2)
        =
        \sum_n 
        \beta_n^2
        \Big(\sfE'_n(\oomega_1)-\sfE'_n(\oomega_2)\Big)^2
\end{equation}
formed with the orthonormal system
of the scalar-valued $L^2$- space 
$\cH'$. 
We have already seen as a particular 
consequence of the calculations of Theorem 
\ref{thm:1d-gaussians-ok} 
that $\alpha^2(\oomega_1,\oomega_2)>0$ 
a.e.~if $\oomega_1\neq \oomega_2$.
\end{example}
\begin{theorem}
    \label{thm:criterion-stronger}
    If
    \begin{equation}
        \label{eq:integral-condition}
        \int_{\OOmega^2}
        \frac1{\alpha^d(\oomega_1,\oomega_2)}
        \,\d\P(\oomega_1)\d\P(\oomega_2)<\infty
    \end{equation}
    then $\ttG=\iota_\sharp\frg$ is super-G-regular.
\end{theorem}
Notice that the integral in \eqref{eq:integral-condition} can be restricted to the complement
of the diagonal $D_\OOmega$ in $\OOmega^2$, since
$\P$ is atomless.
\begin{proof}
We want to prove that 
\eqref{eq:monstre} holds.
We first integrate 
\eqref{eq:dacitare}
with respect to $\Q,$ obtaining
\begin{align}
      \mathbb{E}_\Q  |\hat \mu_\omega(u)|^2
        &=
    \int_\Thetao
    \bigg(\int_{\OOmega^2}
    \rme^{i\, u\cdot (\Xi(\thetao,\oomega_1)-\Xi(\thetao,\oomega_2))}
        \,\d\P^{\otimes 2}(\oomega_1,\oomega_2)
        \bigg)\,\d\Q(\thetao)
        \\&=
        \label{eq:L2norm}
      \int_{\OOmega^2}
      \bigg(
      \int_\Thetao
      \rme^{i\, u\cdot (\Xi(\thetao,\oomega_1)-\Xi(\thetao,\oomega_2))}
      \,\d\Q(\thetao)
      \bigg)
      \,\d\P^{\otimes 2}(\oomega_1,\oomega_2)
      \\&= 
      \int_{\OOmega^2}
      \mathbb{E}_\Q\bigg[
      \rme^{i\, u\cdot (\Xi(\cdot,\oomega_1)-\Xi(\cdot,\oomega_2))}
      \bigg]
      \,\d\P^{\otimes 2}(\oomega_1,\oomega_2).
\end{align}
We now observe that 
for $\P^{\otimes 2}$-a.e.~$(\oomega_1,\oomega_2)$
$\beta^2(\oomega_1,\oomega_2)$ is finite so that 
the expression
\begin{equation}
    \label{eq:inner-gaussian}
    D(\thetao;\oomega_1,\oomega_2)
    :=\Xi(\thetao,\oomega_1)-
    \Xi(\thetao,\oomega_2)=
    \sum_n \xi_n(\thetao)
    [\sfE'_n(\oomega)-\sfE'_n(\oomega_2)]
\end{equation}
is a series of 
independent Gaussian variables
pointwise converging $\Q$-a.e.
Its sum is a Gaussian random variable
with covariance matrix
\begin{equation}
    \label{eq:covariance}
    \Sigma(\oomega_1,\oomega_2):=
    \sum_{n=1}^\infty 
    \Sigma_n[\sfE'_n(\oomega)-\sfE'_n(\oomega_2)]^2,
\end{equation}
satisfying
\begin{equation}
    \label{eq:covariance-bounds}
    \alpha^2(\oomega_1,\oomega_2)
    \le \Sigma(\oomega_1,\oomega_2)
    \vv\cdot \vv\le 
    \beta^2(\oomega_1,\oomega_2)
    \quad 
    \text{for every }\vv\in \R^d,\ |\vv|=1,
\end{equation}
so that 
\begin{equation}
    \label{eq:inner-transform}
    \mathbb{E}_\Q\bigg[
      \rme^{i\, u\cdot (\Xi(\cdot,\oomega_1)-\Xi(\cdot,\oomega_2))}
      \bigg]
      =
      \exp\Big(-\frac 12 \Sigma (\oomega_1,\oomega_2)
      u\cdot u\Big).
\end{equation}
Combining \eqref{eq:inner-transform}
with \eqref{eq:L2norm} we get
\begin{equation}
    \label{eq:almost}
    \int_\cH|\hat \mu_X(u)|^2
        \,\d \frg(X)
        =
        \mathbb{E}_\Q  |\hat \mu_\omega(u)|^2=
    \int_{\OOmega^2}
        \exp\Big(-\frac 12 \Sigma(\oomega_1,\oomega_2)
       u\cdot u\Big)
      \,\d\P^{\otimes 2}(\oomega_1,\oomega_2).
\end{equation}
We 
can plug \eqref{eq:almost} in \eqref{eq:condition1}
obtaining
after a further application of Fubini's Theorem
\begin{align*}
    L^2&=
    \int_{\R^d}
    \Big(
    \int_{\OOmega^2}
        \exp\Big(-\frac 12 \Sigma 
        (\oomega_1,\oomega_2)
       u\cdot u\Big)
      \,\d\P^{\otimes 2}(\oomega_1,\oomega_2)
      \Big)\,\d u
      \\&=
    \int_{\OOmega^2}
    \Big(
    \int_{\R^d}
        \exp\Big(-\frac 12 \Sigma 
        (\oomega_1,\oomega_2)
       u\cdot u\Big)
      \,\d u  
      \Big)
      \,\d\P^{\otimes 2}(\oomega_1,\oomega_2)
\end{align*}
Since the inner integral is
\begin{displaymath}
    \int_{\R^d}
        \exp\Big(-\frac 12 \Sigma 
        (\oomega_1,\oomega_2)
       u\cdot u\Big)\,\d u
       \le 
       \int_{\R^d}
        \exp\Big(-\frac 12 \alpha^2  
        (\oomega_1,\oomega_2)
       |u|^2\Big)\,\d u
      \le 
      \frac{(2\pi)^{d/2}}
      {\alpha^d(\oomega_1,\oomega_2)}
\end{displaymath}
we conclude that 
\begin{displaymath}
    L^2\le   
    (2\pi)^{d/2}
    \int_{\OOmega^2}
    \frac 1{\alpha^d(\oomega_1,\oomega_2)}
    \,\d\P^{\otimes 2}(\oomega_1,\oomega_2),
\end{displaymath}
 so that $L^2$ is finite
 if and only if \eqref{eq:integral-condition}
holds.
\end{proof}
We apply the above result 
to the $d$-dimensional version of 
Example \eqref{ex:random-conv}.
\begin{example}[Sum of $d$-dimensional Gaussians with random signs]
\label{ex:random-conv2}
    Let $\OOmega, \P,$ 
    and the Walsh system $W_I$ as 
    in Example \eqref{ex:random-conv}:
    they form a complete orthonormal system
    for the ``scalar'' Hilbert space
    $\cH'=L^2(\OOmega,\P;\R)$.
    We now select a family of independent $\R^d$-Gaussian random variables
    $\xi_I\sim N(0,\Sigma_I)$ as in the previous discussion with
    \begin{equation}
        \label{eq:bounds-Walsh}
        0<\alpha^2_I\le \Sigma_I\vv\cdot\vv\le \beta_I^2
        \quad\text{for every }\vv\in \R^d,\quad
        B^2=\sum_I \beta_I^2<\infty.
    \end{equation}
    As in \eqref{eq:G-representation-d}
    we consider the random vector
    \begin{equation}
        \label{eq:G-Walsh-d}
        \xxi=\sum_{I\in \cI}
        \xi_IW_I
        \quad
        \text{with}\quad 
        \frg=\xxi_\sharp\Q.
    \end{equation}
    We decompose 
    the set $\cI$ of finite parts of $\N_+$
    in the disjoint union of
    $\cI_n$, $n\in \N$,
    with
    \begin{equation}
        \label{eq:partition}
        \cI_0=\{\emptyset\},\quad 
        \cI_n:=
        \Big\{I\in \cI:
        \max I=n\Big\},\quad n>0,
    \end{equation}
    i.e.
    \begin{displaymath}
        \cI_1=\big\{\{1\}\big\},\ 
        \cI_2=\big\{\{2\},\{1,2\}\big\},\
        \cI_3=\big\{\{3\},\{2,3\},
        \{1,3\},\{1,2,3\}\big\},\cdots
    \end{displaymath}
    Notice that for every $I\in \cI_n$
    the corresponding Walsh function
    can be factorized as 
    \begin{equation}
        \label{eq:factorization}
        W_I=\eps_n W_{I'}
        \quad\text{for some $I'\subset 
    \{1,\cdots,n-1\}.$}    
    \end{equation}
    For each $\cI_n$, $n\in \N$, we compute the 
    contribution of $\alpha_I^2$ 
    to the total sum
    \begin{equation}
        \label{eq:partial-sum}
        A^2_n:=\sum_{I\in \cI_n}\alpha_I^2,\quad
        \text{so that}\quad
        A^2=
        \sum_{n=0}^\infty A_n^2
        =\sum_I \alpha_I^2
        \le B^2.
    \end{equation}
    The next result shows that 
    $G=\iota_\sharp \frg$
    is super-G-regular if $A_n$ does not decay
    too fast.
    \end{example}
    \begin{theorem}
        \label{thm:Walsh-d}
        If
        \begin{equation}
            \label{eq:Walsh-d-condition}
            \sum_{n=1}^\infty \frac{1}{2^n A_n^d}<\infty
        \end{equation}
        then $\ttG=\iota_\sharp \frg$
        is super-G-regular.
    \end{theorem}
    Notice that 
    in the simplest
    case when $\Sigma_I=\lambda_I \sfI_{d\times d}$ 
    and $\alpha_I=\beta_I=\lambda_I$, asymptotic behaviours as 
    $A_n\sim n^{-\theta}$ with $\theta>1/2$
    or $A_n\sim a^{-n}$
    with $1<a<2^{1/d}$
    comply with \eqref{eq:Walsh-d-condition}
    and the summability of $n\mapsto A_n^2$
    (corresponding to $\sum_I\beta_I^2<\infty$).
    \begin{proof}
        We are in the setting of Theorem
        \eqref{thm:criterion-stronger},
        so it is sufficient to check
        that \eqref{eq:integral-condition} holds,
        where in our case
        \begin{equation}
            \label{eq:our-condition}
            \alpha^2(\oomega_1,\oomega_2)=
            \sum_{I\in \cI}
            \alpha_I^2\Big(W_I(\oomega_1)
            -W_I(\oomega_2)\Big)^2
            =
            \sum_{n\in \N}
            \sum_{I\in \cI_n}
            \alpha_I^2\Big(W_I(\oomega_1)
            -W_I(\oomega_2)\Big)^2
        \end{equation}
        For every pair $(\oomega_1,\oomega_2)\in \OOmega^2$ with $\oomega_1\neq \oomega_2$
        let us denote by 
        $N(\oomega_1,\oomega_2)$
        the first integer $n\in N_+$
        such that $\eps_n(\oomega_1)\neq 
        \eps_n(\oomega_2)$:
        \begin{equation}
            \label{eq:min-N}
            N(\oomega_1,\oomega_2):=
            \min\Big\{n\in\N_+:
            \eps_n(\oomega_1)\neq 
        \eps_n(\oomega_2)
        \Big\}.
        \end{equation}
        Since $\oomega_1\neq \oomega_2$
        the set in \eqref{eq:min-N} 
        is not empty, so that $N(\oomega_1,\oomega_2)$ 
        is well defined.
        Since $\eps_k(\oomega_1)=\eps_k(\oomega_2)$
        for every $k<N(\oomega_1,\oomega_2)$
        and therefore
        $W_{I'}(\oomega_1)=W_{I'}(\oomega_2)$
        for every $I'\subset \{1,\cdots,N(\oomega_1,\oomega_2)-1\}$.
        Therefore the factorization 
        \eqref{eq:factorization}
        shows that 
        \begin{equation}
            \label{eq:diff1}
            N=N(\oomega_1,\oomega_2),
            \quad
            I\in \cI_N
            \quad\Rightarrow\quad
            W_I(\oomega_1)=
            \eps_N(\oomega_1)W_{I'(\oomega_1)}
            \neq 
            W_I(\oomega_2)=
            \eps_N(\oomega_2)W_{I'(\oomega_1)}
        \end{equation}
        so that 
        \begin{equation}
            \label{eq:lower-bound}
            \alpha^2(\oomega_1,\oomega_2)
            \ge 
            \sum_{I\in \cI_N}
            \alpha_I^2\Big(W_I(\oomega_1)
            -W_I(\oomega_2)\Big)^2
            =
            4 
            \sum_{I\in \cI_N}
            \alpha_I^2
            =4A^2_N
        \end{equation}
        so that 
        \begin{equation}
            \label{eq:upper-bound}
            \int_{\OOmega^2}\frac1{\alpha^d}
            \,\d\P^{\otimes 2}
            \le 
            \sum_{n=1}^\infty
            \frac{1}{A_n^d}
            \P^{\otimes 2}\Big[
            \{(\oomega_1,\oomega_2)\in \OOmega^2:
            N(\oomega_1,\oomega_2)=n\}\Big].
        \end{equation}
        Recall now that
        $\eps_n$ are independent
        and 
        $\P^{\otimes 2}[\eps_k(\oomega_1)\neq 
        \eps_k(\oomega_2)]= \P^{\otimes 2}[\eps_k(\oomega_1)= 
        \eps_k(\oomega_2)]=1/2$ for all $k\in \N_+$.
        We thus obtain
        \begin{displaymath}
            \P^{\otimes 2}\Big[
            \{(\oomega_1,\oomega_2)\in \OOmega^2:
            N(\oomega_1,\oomega_2)=n\}\Big]=
            \frac 1{2^n}
        \end{displaymath}
        and inserting this expression in \eqref{eq:upper-bound}
        we eventually get
        \begin{displaymath}
            \int_{\OOmega^2}\frac1{\alpha^d}
            \,\d\P^{\otimes 2}
            \le 
             \sum_{n=1}^\infty
            \frac{1}{2^nA_n^d}<+\infty 
        \end{displaymath}
        thanks to \eqref{eq:Walsh-d-condition}.
    \end{proof}

\printbibliography

\end{document}